\documentclass[a4paper,11pt]{amsart}
\usepackage[utf8]{inputenc}
\usepackage{latexsym}

\usepackage[backend=bibtex,style=alphabetic,sorting=nyt]{biblatex} 
\renewbibmacro{in:}{}
\usepackage{amssymb, amsmath, amsthm,stmaryrd,bm,hyperref}
\usepackage{mdwlist}

\def\res{\hbox{ {\vrule height .3cm}{\leaders\hrule\hskip.3cm}}\hskip5.0\mu}

\newcommand\beqn{\begin{equation}}
\newcommand\eeqn{\end{equation}}
\newcommand\beqny{\begin{eqnarray}}
\newcommand\eeqny{\end{eqnarray}}
\newcommand\beqnyn{\begin{eqnarray*}}
\newcommand\eeqnyn{\end{eqnarray*}}

\usepackage{amsmath,amsfonts,amssymb,amsthm}
\newtheorem{theorem}{Theorem}[section]
\newtheorem{lemma}[theorem]{Lemma}

\newtheorem{corollary}[theorem]{Corollary}
\newtheorem{definition}[theorem]{Definition}
\newtheorem{remark}[theorem]{Remark}

\numberwithin{equation}{section}

\newcommand{\op}[1]{\operatorname{\text{\rm #1}}}

\theoremstyle{remark}
\newtheorem{rmk}{\bf Remark}[section]

\begin{document}

\setlength\parskip{5pt}
\title[Fine properties of branch point singularities]{Fine properties of branch point singularities: stationary two-valued graphs and stable minimal hypersurfaces near points of density $< 3$}
\author{Brian Krummel \& Neshan Wickramasekera}

\begin{abstract}
We study  (higher order) asymptotic behaviour near branch points of stationary $n$-dimensional two-valued $C^{1, \mu}$ graphs in an open subset of ${\mathbb R}^{n+m}$. Specifically, if $M$ is the graph of a two-valued $C^{1, \mu}$ function $u$ on an open subset $\Omega \subset {\mathbb R}^{n}$ taking values in the space of un-ordered pairs of points in ${\mathbb R}^{m}$, and if the integral varifold $V = (M, \theta),$ 
 where the multiplicity function $\theta \, : \, M \to \{1, 2\}$ is such that $\theta =2$ on the set where the two values of $u$ agree and $\theta =1$ otherwise, is stationary in $\Omega \times {\mathbb R}^{m}$ with respect to the mass functional, we show that at ${\mathcal H}^{n-2}$-a.e.\ point $Z$ along its branch locus $u$ decays asymptotically, modulo its single valued average,  to a unique non-zero two-valued  cylindrical harmonic tangent function $\varphi^{(Z)}$ which is homogeneous of some degree $\geq 3/2$. As a corollary, we obtain that the branch locus of $u$ is countably $(n-2)$-rectifiable, and near points $Z$ where the degree of homogeneity  of $\varphi^{(Z)}$ is equal to $3/2$, the branch locus is an embedded real analytic submanifold of dimension $n-2$. These results, combined with the recent works \cite{M} and \cite{MW}, imply a stratification theorem for the (relatively open) set of density $< 3$ points of a stationary codimension 1 integral $n$-varifold with stable regular part and no triple junction singularities. 
\end{abstract}

\maketitle
\tableofcontents

\section{Introduction} 
For certain types of singularities of stationary integral varifolds, the works of L.~Simon in the early 1990's, and of Naber--Valtorta more recently, have established various fine properties, including asymptotic decay of the varifold to a unique tangent cone at a.e.\ singular point (\cite{Sim93}, \cite{Sim93b}) and countable rectifiability of the singular set (\cite{Sim93}, \cite{Sim93b}, \cite{NV}). 
These works apply to singularities that are assumed to be non-degenerate, in the sense of requiring that no tangent cone to the varifold at a singular point is supported on a plane. Let us henceforth call a singular point a \emph{branch point} if one tangent cone at that point is supported on an $n$-dimensional plane and in no neighborhood of that point the varifold is the sum of smoothly embedded $n$-dimensional submanifolds. Simon's work rules out branch points a priori by requiring that the stationary varifold belongs to a  ``multiplicity 1 class,'' i.e.\ a class ${\mathcal M}$ consisting of multiplicity 1 varifolds of some dimension $n$ in a given Euclidean space ${\mathbb R}^{n+m}$, such that each varifold in ${\mathcal M}$ is stationary in an open subsets of ${\mathbb R}^{n+m}$ and ${\mathcal M}$ is closed under translations, dilations, rotations and taking varifold limits. This last closure requirement  with respect to varifold limits in particular imposes the condition that tangent cones must have multiplicity 1 a.e., ruling out branch points.  The Naber--Valtorta work imposes no restriction on the (integer) multiplicity of the varifolds nor does it rule out the existence of branch points, but  proves countable rectifiability results for those strata of the singular set that do not contain branch points.   (Both these theories extend to bounded mean curvature setting, and hence to Riemannian ambient spaces, with straightforward technical modifications to the proofs in the zero mean curvature case.) 

The nature of branch point singularities on the other hand are understood in much less generality and precision. Among the known results are  Almgren's monumental work (\cite{Almgren}; see also \cite{DeLSpa-I}, \cite{DeLSpa-II}, \cite{DeLSpa-III}) which establishes that the sharp Hausdorff dimension upper bound on the branch set of an $n$-dimensional area minimising integral current of codimension $>1$ is $(n-2)$; for two dimensional area minimising integral currents much more is known, namely, that the branch points are isolated and that locally near a branch point the current is a parameterised disk (\cite{Chang}, \cite{DSS-I}, \cite{DSS-II}, \cite{DSS-III}).  In codimension 1, area minimising integral currents do not develop branch points, but stable varifolds generally do (\cite{SimWic07}, \cite{Ros10}, \cite{Krum16}). Concerning stable codimension 1 integral varifolds,  the work (\cite{MW}) shows that at a density $Q$ branch point which is not the limit point of classical singularities  of density 
$< Q,$ the varifold has a unique (multiplicity $Q$) tangent hyperplane, and that the varifold near that point is a $Q$-valued 
``generalised-$C^{1, \alpha}$'' graph over the tangent plane; here a classical singularity is a singular point near which the varifold is the sum of regular sheets coming smoothly together along an $(n-1)$-dimensional submanifold. 
Combining this regularity result (taken with $Q= 2$) with the work \cite{SimWic16} gives that for stable codimension 1 integral varifolds, the set of density 2 branch points which are not the limit points of triple junction singularities (i.e.\ classical singularities of density $3/2$) has Hausdorff dimension at most $n-2$. 

Whether the structure of (non-isolated) branch sets must be constrained in a way similar to the aforementioned rectifiability results of Simon and Naber--Valtorta (for non branch point singularities) is a natural question and yet one that has remained open for any class of stationary varifolds. 
In addressing this question, the main difference, in contrast to the analysis of non branch point singularities, is that to reveal the structural properties of branch sets one would need some form of higher order (i.e.\ beyond the level of tangent cone) analysis.

The present work is part of a series of papers in which we establish such higher order asymptotics for certain mimimal submanifolds and multi-valued harmonic functions, and as a corollary also the countable rectifiability of the the corresponding branch loci.  As part of this study, in~\cite{KrumWic1} we considered the two-valued case of multi-valued Dirichlet energy minimizing functions $v$ as developed by Almgren in~\cite{Almgren} as part of his proof of the Hausdorff dimension upper bound on the branch sets of area minimizing integral currents, as well as  two-valued $C^{1,\mu}$ harmonic functions $v$ (not necessarily energy minimizing) which arise for instance in the work \cite{Wic08}, \cite{Wic} (both being subsumed by \cite{MW}) on regularity of stable minimal hypersurfaces near density 2 branch points. 
In either case, the main results of \cite{KrumWic1} established that  on approach to ${\mathcal H}^{n-2}$-a.e.\ point of its branch locus, $v$ decays asymptotically, modulo its single-valued harmonic average, to a unique non-zero homogeneous cylindrical two-valued harmonic function, and that the branch locus of $v$ is countably $(n-2)$ rectifiable. In~\cite{KrumWic2}, we showed that the main results of \cite{KrumWic1} (for two-valued Dirichlet energy minimizers) hold in full generality, i.e.\ for $q$-valued Dirichlet energy minimizing functions for arbitrary $q \geq 2.$ (A key difference in this extension is that when $q \ge 3,$ the frequency---equivalently, the degree of homogeneity of blow-ups---can be $< 1/2$ on a set of branch points of positive $(n-2)$-dimensional Hausdorff measure; analysis near such points requires considerably more effort than for the case $q=2$ where the frequency is always $\geq 1/2$).  

Here we consider a corresponding quasilinear problem, i.e.\ $C^{1,\mu}$ two-valued minimal (stationary) graphs of arbitrary dimension and codimension.  In view of the 
aformentioned regularity results for stable codimension 1 integral varifolds admitting no triple junction singularities, our results here are directly applicable to the simplest type of branch points of such varifolds, i.e.\ density 2 branch points. 

We consider $M \equiv {\rm graph} \, u$ where  $u$ is a $C^{1,\mu}$ two-valued function over a domain $\Omega \subset {\mathbb R}^{n}$, which at each point $X \in \Omega$ takes a pair of values $u(X) = \{u_1(X),u_2(X)\}$ with $u_1(X),u_2(X) \in \mathbb{R}^m$, not necessarily distinct, and 
where ${\rm graph} \, u = \{(X,Y) \in \Omega \times {\mathbb R}^{m}\, : \, Y = u_{1}(X) \;\;\mbox{or} \;\; Y = u_{2}(X)\}.$ We assume that $M$, equipped with multiplicity $\theta(X, Y) = 1$ whenever $u_{1}(X) \neq u_{2}(X)$ and $\theta(X, Y) = 2$ whenever $u_{1}(X) = u_{2}(X),$ is stationary in $\Omega \times {\mathbb R}^{m},$ as a varifold, with respect to the mass functional (in the sense of Definition~\ref{twoval minimal graph} below). Note that here $u_{j}(X)$, $j=1, 2$ denote the two unordered values of $u$ at $X$, and in particular $u_{1}$, $u_{2}$ need \emph{not} separately represent single-valued $C^{1}$ functions globally defined on $\Omega.$ A branch point of $M$ is then a point of $M$ for which this separation property fails in every neighborhood of its projection onto $\Omega$ . 
Thus $X_{0} \in M$ is a branch point if, 
writing $X_{0} = (X_{0}^{\prime}, Y_{0})$ with $X_{0}^{\prime} \in \Omega$ and $Y_{0} \in {\mathbb R}^{m}$, 
there is no $\sigma >0$ having the property that  there are two (single-valued) functions 
$u_{1}, u_{2} \, : \, B_{\sigma}^{n} (X_{0}^{\prime}) \to {\mathbb R}^{m}$ of class $C^{1}$ such that $u(X) = \{u_1(X),u_2(X)\}$ for all $X \in  B_{\sigma}^{n} (X_{0}^{\prime}).$

\subsection{Main results of the present work}  Let $M$ be as above. We develop estimates to study higher order asymptotic behavior of  $M$ on approach to its branch points, including the question of uniqueness of \emph{tangent functions} $\varphi$ along the branch set.  The first main result implied by our estimates,  Theorem~A below, can be seen as giving a asymptotic expansion for $M$ at $\mathcal{H}^{n-2}$-a.e.~branch point $P$. 

\noindent \textbf{Theorem~A.} \emph{Let $\mu \in (0,1]$.  Let $\Omega \subset {\mathbb R}^{n}$ be open, and let $M$ be the graph of a $C^{1,\mu}$ two-valued function 
$u : \Omega \rightarrow \mathcal{A}_2(\mathbb{R}^m)$. Suppose that $M$ is stationary 
in $\Omega \times \mathbb{R}^m$ in the sense of Definition~\ref{twoval minimal graph}.  For $\mathcal{H}^{n-2}$-a.e.~branch point $P$ of $M$, there exists a number $\rho_P > 0$ and an orthogonal rotation $q_P$ of $T_P M$ such that 
\begin{equation*}
	(M - P) \cap B^{n+m}_{\rho_P}(0) = {\rm graph}\,\widetilde{u}_P \cap B^{n+m}_{\rho_P}(0)
\end{equation*}
for some $C^{1,1/2}$ two-valued function $\widetilde{u}_P : B^{n+m}_{\rho_P}(0) \cap T_P M \rightarrow \mathcal{A}_2(T_P M^{\perp})$ and 
\begin{equation*}
	\widetilde{u}_P(X) = \{ \widetilde{h}_P(X) - \varphi^{(P)}_1(q_P X) - \epsilon_P(X), \widetilde{h}_P(X) + \varphi^{(P)}_1(q_P X) + \epsilon_P(X) \}
\end{equation*}
for each $X \in B^{n+m}_{\rho_P}(0) \cap T_{P} M$ where $\widetilde{h}_P : B^{n+m}_{\rho_P}(0) \cap T_P M \rightarrow T_P M^{\perp}$ is a $C^{1,1}$ single-valued function (which depends on $P$ and in fact equals the average of the two values of $\widetilde{u}_P$); 
$$\varphi^{(P)} : T_P M \rightarrow \mathcal{A}_2(T_P M^{\perp})$$ 
is a non-zero, symmetric, homogeneous, harmonic two-valued function given by 
\begin{equation*}
	\varphi^{(P)}(X) = \{ \pm \varphi^{(P)}_1(X) \} = \{ \pm \op{Re}((a^{(P)} + i b^{(P)}) \,(x_1+ix_2)^{k_P/2}) \}
\end{equation*}
for some integer $k_P \geq 3$ and some $a^{(P)}, b^{(P)} \in T_P M^{\perp}$ (where we use coordinates $X = (x_1,x_2,\ldots,x_n)$ with respect to an orthonormal basis for $T_P M$), and 
$$\{\pm \epsilon_P\} : B^{n+m}_{\rho_P}(0) \cap T_P M \rightarrow \mathcal{A}_2(T_P M^{\perp})$$ 
is a symmetric two-valued function with 
\begin{equation*}
	\sigma^{-n-k_P} \int_{B_{\sigma}(0)} |\epsilon_P|^2 \leq C_P \,\sigma^{\gamma_P}
\end{equation*}
for all $\sigma \in (0,\rho_P)$ and some constants $C_P, \gamma_P > 0$ independent of $\sigma$. }

Theorem~A in particular gives uniqueness of tangent functions $\varphi^{(P)}$  at $\mathcal{H}^{n-2}$-a.e.~branch point $P$.

Let $\mathcal{B}_u$ be the set of branch point of $u$, i.e.\ the set of points $Z \in \Omega$ for which there is no $\delta > 0$ such that $u = \{u_1,u_2\}$ in $B_{\delta}(Z)$ for two single-valued functions 
$u_{1}, u_{2} \in C^{1}(B_{\delta}(Z); {\mathbb R}^{m}).$   (Thus ${\mathcal B}_{u}$ is the projection of the branch set of ${\rm graph} \, u$). We have the following variant of Theorem~A which expresses the conclusion of Theorem~A as an asymptotic expansion of $u$ (as opposed to an expansion of $\widetilde{u}_{(Z, u_{1}(Z))}$) near ${\mathcal H}^{n-2}$ a.e.\ $Z \in {\mathcal B}_{u}$: 

\noindent \textbf{Theorem~A$^\prime$.} \emph{Let $\mu \in (0,1]$.  There exists $\varepsilon = \varepsilon(n,m,\mu) > 0$ such that the following holds true.  Let $u : B_1(0) \rightarrow \mathcal{A}_2(\mathbb{R}^m)$ be a $C^{1,\mu}$ two-valued function such that $\|Du\|_{C^{0,\mu}(B_1(0))} \leq \varepsilon$ and suppose that $M = {\rm graph}\,u$ be stationary in $B_1(0) \times \mathbb{R}^m$ in the sense of Definition~\ref{twoval minimal graph}.  For $\mathcal{H}^{n-2}$-a.e.~$Z \in \mathcal{B}_u \cap B_{1/2}(0)$, there exists a number $\rho_Z \in (0,1/16]$, an orthogonal rotation $q_Z$ of $\mathbb{R}^n$ and an invertible $n \times n$ real matrix $A_Z$ with $\|A_Z - I\| \leq C(n,m) \,|Du(Z)|$ such that 
\begin{equation*}
	u(Z+X) = \{ h(Z+X) - \varphi^{(Z)}_1(q_Z A_Z X) - \epsilon_Z(X), h(Z+X) + \varphi^{(Z)}_1(q_Z A_Z X) + \epsilon_Z(X) \}
\end{equation*}
for each $X \in B_{\rho_Z}(0)$ where $h : B_1(0) \rightarrow \mathbb{R}^m$ is a $C^{1,1}$ single-valued function independent of $Z$ (and in fact $h(X)$ is the average of the two values of $u(X)$); 
$$\varphi^{(Z)} : \mathbb{R}^n \rightarrow \mathcal{A}_2(\mathbb{R}^m)$$ 
is non-zero, symmetric, homogeneous, harmonic two-valued function given by 
\begin{equation*}
	\varphi^{(Z)}(X) = \{ \pm \varphi^{(Z)}_1(X) \} = \{ \pm \op{Re}(c^{(Z)} (x_1+ix_2)^{k_Z/2}) \}
\end{equation*}
for some integer $k_Z \geq 3$ and some $c^{(Z)} \in \mathbb{C}^m \setminus \{0\}$ (where the notation is such that $X = (x_1,x_2,\ldots,x_n)$); and $\{\pm \epsilon_Z\} : B_{\rho_Z}(0) \rightarrow \mathcal{A}_2(\mathbb{R}^m)$ is a symmetric two-valued function with 
\begin{equation*}
	\sigma^{-n-k_Z} \int_{B_{\sigma}(0)} |\epsilon_Z|^2 \leq C_Z \,\sigma^{\gamma_Z}
\end{equation*}
for all $\sigma \in (0,\rho_Z)$ and some constants $C_Z, \gamma_Z > 0$ independent of $\sigma$. }

Much more is true near any branch point $P_0$ where the expansion as in Theorem~A is valid with $k_{P_0}$ equal to the minimum value $k_{P_0} = 3$:

\noindent \textbf{Theorem~B.} \emph{Let $\mu \in (0,1]$ and let $u : \Omega \rightarrow \mathcal{A}_2(\mathbb{R}^m)$ be a $C^{1,\mu}$ two-valued function such that $M = {\rm graph}\,u$ is stationary in $\Omega \times \mathbb{R}^m$ in the sense of Definition~\ref{twoval minimal graph}.  If $k_{P_0} = 3$ for some branch point $P_0$ of $M$ at which the asymptotic expansion as in Theorem~A holds true, then the expansion as in Theorem~A is valid for every branch point $P \in B^{n+m}_{\rho_{P_0}/2}(P_0)$ of $M$ (where $\rho_{P_0}$ is as in Theorem~A) with $k_P = 3$, $\rho_P = \rho_{P_0}/2$, $C_P = C_{P_0}$ and $\gamma_P = \gamma_{P_0}$.  Furthermore, letting $\{\pm \epsilon_P\}$ be as in Theorem~A, we have in this case that
\begin{equation*}
	\sup_{B_{\sigma}(0)} |\epsilon_P|^2 \leq C_0 \,\sigma^{3 + \frac{\gamma_{P_0}}{2n}}
\end{equation*} 
for every branch point $P \in B_{\rho_{P_0}/2}(P_0)$ of $M$ and every $\sigma \in (0,\rho_{P_0}/16]$, where $C_0$ is independent of $\sigma$ and $P$.  We also have that the branch set of $M$ in $B^{n+m}_{\rho_{P_0}/2}(P_0)$ is an $(n-2)$-dimensional $C^{1,\alpha}$ submanifold of  $B^{n+m}_{\rho_{P_0}/2}(P_0)$ for some $\alpha = \alpha_{P_0} \in (0,1)$.}

\emph{In fact if $P_{0} \in M$ is such that one tangent function at $P_{0}$, after composing with an orthogonal rotation of $T_{P_0} M$, has the form $\op{Re}(c (x_1+ix_2)^{3/2})$ with $c \in \mathbb{C}^m \setminus \{0\}$ a constant, then all of the above conclusions hold for some $\rho_{P_{0}}>0$ and constants $C_{P_{0}}, \gamma_{P_{0}} >0$.}

\begin{rmk}
The same result also holds  under the hypotheses of Theorem~A$^{\prime}$ near any point $Z_0 \in \mathcal{B}_u$ such that one tangent function to $M$ at 
$P_{0} = (Z_{0}, u_{1}(Z_{0}))$, after composing with an orthogonal rotation of $T_{P_0} M$, has the form $\op{Re}(c (x_1+ix_2)^{3/2})$ with $c \in \mathbb{C}^m \setminus \{0\}$ a constant. 
(We omit the formal statement and proof as they correspond to Theorem~B with obvious changes.)
\end{rmk}

The estimates leading to Theorem~A give also the following structure result for the branch set of a two-valued stationary graph:

\noindent \textbf{Theorem~C.} \emph{Let $\mu \in (0,1]$ and let $u : \Omega \rightarrow \mathcal{A}_2(\mathbb{R}^m)$ be a $C^{1,\mu}$ two-valued function such that $M = {\rm graph}\,u$ is stationary in $\Omega \times \mathbb{R}^m$ in the sense of Definition~\ref{twoval minimal graph}. Let $\op{sing} M$ be the branch set of $M$.  If $n = 2$, $\op{sing} M$ is discrete.  If $n \geq 3$ then for each closed ball $B \subset \Omega$, either $\op{sing} M \cap (B \times \mathbb{R}^m) = \emptyset$ or $\op{sing} M \cap (B \times \mathbb{R}^m)$ has positive $(n-2)$-dimensional Hausdorff measure and is equal to the union of a finite number of pairwise disjoint, locally compact sets each of which is locally $(n-2)$-rectifiable (and has in particular locally finite $(n-2)$-dimensional Hausdorff measure). }

Combining the above results with the main theorems in \cite{M} and \cite{MW}, we obtain the following stratification result for the density $< 3$ region of a stable codimension 1 integral varifolds with no density $3/2$ points.

\noindent \textbf{Theorem~D.} \emph{Let $V$ be a stationary integral $n$-varifold on $B^{n+1}_1(0)$ such that any two-sided portion of 
the regular part of $V$ is stable. Suppose that $\{P \in B^{n+1}_{1}(0) \, : \, \Theta \, (\|V\|, P) = 3/2\} = \emptyset$. 
Then 
\begin{equation*}
	{\rm spt} \, \|V\| \cap B^{n+1}_1(0) \cap \{ P : \Theta^n(\|V\|, P) < 3 \} = \Omega \cup \mathcal{T} \cup \mathcal{C} \cup \mathcal{B} \cup \mathcal{S}
\end{equation*}
where: }
\begin{enumerate}
        \item[(a)] \emph{$\Omega$ is the set of points $P \in 	{\rm spt} \, \|V\| \cap B^{n+1}_1(0) \cap \{ P : \Theta^n(\|V\|, P) < 3 \}$ such that ${\rm spt} \, \|V\|$ is a 
        smoothly embedded hypersurface of $B_{1}^{n+1}(0)$ near $P$};
        \item[(b)] \cite[Theorem~C]{MW} \emph{${\mathcal T}$ is the set of points $P \in {\rm spt} \, \|V\| \cap B_{1}^{n+1}(0)$ such that one tangent cone to $V$ at $P$ is $|L^{1}_{P}| + |L^{2}_{P}|$ for a pair of transversely intersecting hyperplanes $L^{1}_{P}$, $L^{2}_{P}$. This is the unique tangent cone to $V$ at $P$, and there is $\delta_{P} >0$ such that ${\rm spt} \, \|V\| \cap B_{\delta_{P}}^{n+1}(P)$ is the union of two transversely intersecting smoothly embedded minimal disks (in fact two graphs over the two hyperplanes $L^{1}_{P}$, $L^{2}_{P}$); ${\mathcal T}$ is a smoothly embedded $(n-1)$-dimensional submanifold of $B_{1}^{n+1}(0)$}; 
        \item[(c)] \cite{M} \emph{$\mathcal{C}$ is the set of points $P \in {\rm spt} \, \|V\| \cap B_{1}^{n+1}(0)$ where one tangent cone is the sum of 5 half-hyerplanes meeting along a common axis. At each point in ${\mathcal C}$ the varifold has a unique tangent cone, and ${\mathcal C}$ is a smoothly embedded $(n-1)$-dimensional submanifold.} 
	\item[(d)] \cite[Theorem~C] {MW} \emph{$\mathcal{B}$ is the set of points $P \in ({\rm spt}  \, \|V\| \setminus \Omega) \cap B^{n+1}_1(0)$ such that one tangent cone to $V$ at $P$ is $2|L_{P}|$ for some hyperplane $L_{P}$. This is the unique tangent cone to $V$ at $P$, and there is a $\delta_{P} > 0$ such that $M_{P} = {\rm spt} \, \|V\| \cap B^{n+1}_{\delta_{P}}(P)$ is the graph of a $C^{1,1/2}$ two-valued function over a domain in 
	$L_{P}$.} 
	\item[(e)] Moreover, by Theorem~A and Theorem~C, 
	\begin{enumerate} 
		\item[(i)] \emph{for $\mathcal{H}^{n-2}$-a.e.~$P \in \mathcal{B}$, there is $\rho_{P}>0$, a function $\widetilde{u}_{P} \, : \, B_{\rho_{P}}^{n+1}(0) \cap L_{P} \to {\mathcal A}_{2}(L_{p}^{\perp})$ of class $C^{1, 1/2}$, a unique non-zero cylindrical tangent function $\varphi^{(P)} \, : \, L_{P} \to {\mathcal A}_{2}(L_{P}^{\perp})$ to $M$ at $P$ and an orthogonal rotation $q_{P}$ of $L_{P}$ such that  $(M - P) \cap B^{n+1}_{\rho_P}(0) = {\rm graph}\,\widetilde{u}_P \cap B^{n+1}_{\rho_P}(0)$
and the asymptotic expansion of $\widetilde{u}_{P}$ as in Theorem~A (with $L_{P}$ in place of $T_{P} \, M$) holds; }

		\item[(ii)] \emph{${\mathcal B}$ is countably $(n-2)$-rectifiable; in fact for each closed ball $$B \subset B^{n+1}_1(0) \setminus \left(\{ P : \Theta^n(\|V\|) \geq 3 \} \cup \mathcal{S}\right),$$ either $\mathcal{B} \cap B= \emptyset$ or $\mathcal{H}^{n-2}(\mathcal{B} \cap B) > 0$ and $\mathcal{B} \cap B$ is equal to the union of a finite number of pairwise disjoint, locally compact sets each of which is locally $(n-2)$-rectifiable (and has in particular locally finite $(n-2)$-dimensional Hausdorff measure). }
	\end{enumerate}
	
	\item[(d)] \emph{$\mathcal{S} = {\rm spt} \, \|V\| \cap B^{n+1}_1(0) \cap \{ P : \Theta^n(\|V\|, P) < 3 \} \setminus (\Omega \cup {\mathcal T} \cup {\mathcal B})$ is a relatively closed subset of
${\rm spt} \, \|V\| \cap B^{n+1}_1(0) \cap \{ P : \Theta^n(\|V\|, P) < 3 \}$ of Hausdorff dimension $\leq n-3$.}
\end{enumerate}

\section{Two-valued function preliminaries} \label{sec:preliminaries sec}

\subsection{Two-valued functions}  We let $\mathcal{A}_2(\mathbb{R}^m)$ denote the space of unordered pairs $a = \{a_1,a_2\}$ where $a_1,a_2 \in \mathbb{R}^m$, possibly repeating.  We equip $\mathcal{A}_2(\mathbb{R}^m)$ with the metric $\mathcal{G}$ defined by 
\begin{equation*}
	\mathcal{G}(a,b) = \min\left\{ \sqrt{|a_1-b_1|^2 + |a_2-b_2|^2}, \sqrt{|a_1-b_2|^2 + |a_2-b_1|^2} \right\}
\end{equation*}
for each $a = \{a_1,a_2\}, \,b = \{b_1,b_2\} \in \mathcal{A}_2(\mathbb{R}^m)$.  
For each $a = \{a_1,a_2\} \in \mathcal{A}_2(\mathbb{R}^m)$ we let 
\begin{equation*}
	|a| = \mathcal{G}(a,\{0,0\}) = \sqrt{|a_1|^2 + |a_2|^2}.
\end{equation*}

Let $\Omega \subseteq \mathbb{R}^n$ be an open subset.  A \emph{two-valued function} $u : \Omega \rightarrow \mathcal{A}_2(\mathbb{R}^m)$ is a map such that at each $X \in \Omega$, $u(X) = \{u_1(X),u_2(X)\}$ as an unordered pair of $u_1(X),u_2(X) \in \mathbb{R}^m$. 

To each two-valued function $u : \Omega \rightarrow \mathcal{A}_2(\mathbb{R}^m)$ we associate the single-valued average $u_a : \Omega \rightarrow \mathbb{R}^m$ and symmetric part $u_s : \Omega \rightarrow \mathcal{A}_2(\mathbb{R}^m)$ defined by 
\begin{equation} \label{avg and free defn}
	u_a(X) = \frac{u_1(X) + u_2(X)}{2}, \quad u_s(X) = \left\{ \pm \frac{u_1(X) - u_2(X)}{2} \right\} 
\end{equation}
for each $X \in \Omega$, where $u(X) = \{u_1(X),u_2(X)\}$.  We say the two-valued function $u$ is \emph{symmetric} if $u_a(X) = 0$ for all $X \in \Omega$.  In other words, $u$ is symmetric if and only if $u(X) = \{\pm u_1(X)\}$ for each $X \in \Omega$, where $u_1(X) \in \mathbb{R}^m$.  Clearly $u_s$ is always a symmetric two-valued function.

Since $\mathcal{A}_2(\mathbb{R}^m)$ is a metric space, we can define the space $C^0(\Omega,\mathcal{A}_2(\mathbb{R}^m))$ of continuous two-valued functions with the uniform topology in the usual way.  We define the space $C^{0,\mu}(\Omega,\mathcal{A}_2(\mathbb{R}^m))$ of H\"older continuous two-valued functions with exponent $\mu \in (0,1]$ to be the set of all $u : \Omega \rightarrow \mathcal{A}_2(\mathbb{R}^m)$ such that 
\begin{equation*}
	[u]_{\mu,\Omega} \equiv \sup_{X,Y \in \Omega, \,X \neq Y} \frac{\mathcal{G}(u(X),u(Y))}{|X-Y|^{\mu}} < \infty.  
\end{equation*}
For each $u \in C^0(\Omega, \mathcal{A}_2(\mathbb{R}^m))$ we let 
\begin{equation*}
	\|u\|_{C^0(\Omega)} = \sup_{\Omega} |u| 
\end{equation*}
and for each $\mu \in (0,1]$ and $u \in C^{0,\mu}(\Omega, \mathcal{A}_2(\mathbb{R}^m))$ we let 
\begin{equation*}
	\|u\|_{C^{0,\mu}(\Omega)} = \sup_{\Omega} |u| + [u]_{\mu,\Omega}. 
\end{equation*}

We say a two-valued function $u : \Omega \rightarrow \mathcal{A}_2(\mathbb{R}^m)$ is \emph{differentiable} at a point $Y \in \Omega$ if there exists an affine two-valued function $\ell_Y : \mathbb{R}^n \rightarrow \mathcal{A}_2(\mathbb{R}^m)$, i.e.~a two-valued function given by $\ell_Y(X) = \{A_1^Y X + b_1^Y, A_2^Y X + b_2^Y\}$ for each $X \in \mathbb{R}^n$ and for some constant $m \times n$ real-valued matrices $A_1^Y,A_2^Y$ and constants $b_1^Y,b_2^Y \in \mathbb{R}^m$, such that 
\begin{equation*}
	\lim_{X \rightarrow 0} \frac{\mathcal{G}(u(Y+X),\ell_Y(X))}{|X|} = 0.
\end{equation*}
$\ell_Y$ is unique if it exists.  We call $\ell_Y$ the \emph{affine approximation} of $u$ at $Y$.  We let $Du(Y) = \{A_1^Y,A_2^Y\}$ denote the \emph{derivative} of $u$ at $Y$.  We define the space $C^1(\Omega,\mathcal{A}_2(\mathbb{R}^m))$ as the space of all $u : \Omega \rightarrow \mathcal{A}_2(\mathbb{R}^m)$ such that $u$ is differentiable at each point of $\Omega$ and 
\begin{equation*}
	\lim_{Y \rightarrow Y_0} \,\sup_{X \in B_1(0)} \mathcal{G}(\ell_Y(X),\ell_{Y_0}(X)) = 0 
\end{equation*}
for all $Y_0 \in \Omega$, where $\ell_Y, \ell_{Y_0}$ are the affine approximations of $u$ at $Y,Y_0 \in \Omega$.  If $u \in C^1(\Omega,\mathcal{A}_2(\mathbb{R}^m))$, then $u$ and $Du$ are continuous on $\Omega$.  Given $u_k,u \in C^1(\Omega,\mathcal{A}_2(\mathbb{R}^m))$ we say $u_k \rightarrow u$ in $C^1(\Omega,\mathcal{A}_2(\mathbb{R}^m))$ if 
\begin{equation*}
	\lim_{k \rightarrow \infty} \,\sup_{Y \in \Omega} \,\sup_{X \in B_1(0)} \mathcal{G}(\ell_{k,Y}(X),\ell_{Y}(X)) = 0, 
\end{equation*}
where $\ell_{k,Y}, \ell_Y$ are the affine approximations of $u_k, u$ at $Y \in \Omega$.  If $u_k \rightarrow u$ in $C^1(\Omega,\mathcal{A}_2(\mathbb{R}^m))$, then $u_k \rightarrow u$ and $Du_k \rightarrow Du$ uniformly on $\Omega$.  For each $\mu \in (0,1]$ we define $C^{1,\mu}(\Omega,\mathcal{A}_2(\mathbb{R}^m))$ to be the space of all $u \in C^1(\Omega,\mathcal{A}_2(\mathbb{R}^m))$ such that $Du \in C^{0,\mu}(\Omega,\mathcal{A}_2(\mathbb{R}^{m \times n}))$.  For each $u \in C^1(\Omega, \mathcal{A}_2(\mathbb{R}^m))$ we let 
\begin{equation*}
	\|u\|_{C^1(\Omega)} = \sup_{\Omega} |u| + \sup_{\Omega} |Du| 
\end{equation*}
and for each $\mu \in (0,1]$ and $u \in C^{1,\mu}(\Omega, \mathcal{A}_2(\mathbb{R}^m))$ we let 
\begin{equation*}
	\|u\|_{C^{1,\mu}(\Omega)} = \sup_{\Omega} |u| + \sup_{\Omega} |Du| + [Du]_{\mu,\Omega}. 
\end{equation*}

\begin{definition} \label{singular sets defn}
Given a two-valued function $u \in C^1(\Omega,\mathcal{A}_2(\mathbb{R}^m))$ we define: 
\begin{enumerate}
	\item[(a)] $\mathcal{Z}_u = \{ X \in \Omega : u_1(X) = u_2(X) \}$ and 
	\item[(b)] $\mathcal{K}_u = \{ X \in \Omega : u_1(X) = u_2(X), \,Du_1(X) = Du_2(X) \}$,
\end{enumerate}
where $u(X) = \{u_1(X),u_2(X)\}$ and $Du(X) = \{Du_1(X),Du_2(X)\}$ for each $X \in \Omega$, and
\begin{enumerate}
	\item[(c)] the \emph{branch set} $\mathcal{B}_u$ to be the set of all points $Y \in \Omega$ such that there is no radius $0 < \delta < \op{dist}(Y,\partial \Omega)$ such that $u(X) = \{u_1(X),u_2(X)\}$ for all $X \in B_{\delta}(Y)$ for single-valued functions $u_1,u_2 \in C^1(B_{\delta}(Y),\mathbb{R}^m)$. 
\end{enumerate}
\end{definition}

Clearly $\mathcal{B}_u \subseteq \mathcal{K}_u \subseteq \mathcal{Z}_u$.  Moreover, 
\begin{align*}
	\mathcal{Z}_u &= \{ X \in \Omega : u_s(X) = \{0,0\} \}, \\
	\mathcal{K}_u &= \{ X \in \Omega : u_s(X) = \{0,0\}, \,Du_s(X) = \{0,0\} \} ,
\end{align*}
where $u_s$ is the symmetric part of $u$ as in \eqref{avg and free defn}.  

For each $1 \leq p \leq \infty$, we define $L^p(\Omega,\mathcal{A}_2(\mathbb{R}^m))$ to be the space of all Lebesgue measurable functions $u : \Omega \rightarrow \mathcal{A}_2(\mathbb{R}^m)$ taking values in the metric space $\mathcal{A}_2(\mathbb{R}^m)$ such that $\|u\|_{L^p(\Omega)} = \|\mathcal{G}(u,\{0,0\})\|_{L^p(\Omega)} < \infty$.  We equip $L^p(\Omega,\mathcal{A}_2(\mathbb{R}^m))$ with the metric $d_{L^p}(u,v) = \|\mathcal{G}(u,v)\|_{L^p(\Omega)}$ for each $u,v \in L^p(\Omega,\mathcal{A}_2(\mathbb{R}^m))$. 

Let $u \in C^1(\Omega,\mathcal{A}_2(\mathbb{R}^m))$ be a symmetric two-valued function such that $u$ is smooth in $\Omega \setminus \mathcal{K}_u$ in the sense that in each open ball $B \subset \Omega \setminus \mathcal{K}_u$, $u = \{ \pm u_1\}$ for some smooth function $u_1 : B \rightarrow \mathbb{R}^m$.  Following~\cite{SimWic16}, we say $u \in W^{2,2}_{\rm loc}(\Omega,\mathcal{A}_2(\mathbb{R}^m))$ if $D^2 u \in L^2(\Omega' \setminus \mathcal{K}_u,\mathcal{A}_2(\mathbb{R}^m))$ for every open set $\Omega'$ with $\overline{\Omega'} \subset \Omega$ and we extend $D^2 u = \{0,0\}$ on $\mathcal{K}_u$.  By approximating the derivatives of $u$ as in~\cite{SimWic16}, one can show that $Du$ is in $W^{1,2}_{\rm loc}(\Omega,\mathcal{A}_2(\mathbb{R}^m))$, the space of Sobolev two-valued functions as defined by Almgren in~\cite{Almgren}.  Equivalently, $W^{1,2}_{\rm loc}(\Omega,\mathcal{A}_2(\mathbb{R}^m))$ is the space of Sobolev functions taking values in the metric space $\mathcal{A}_2(\mathbb{R}^m)$ as defined by Ambrosio~\cite{Ambrosio} and Reshetnyak~\cite{Resh1}\cite{Resh2}, see~\cite{DeLSpa11}.

\subsection{Two-valued harmonic functions} \label{sec:prelims harmonic subsec}  $C^{1,\mu}$ two-valued harmonic functions were introduced in~\cite{SimWic16} as approximations of $C^{1,\mu}$ two-valued functions whose graphs are area-stationary varifolds (as in Definition~\ref{twoval minimal graph} below).  We will use $C^{1,\mu}$ two-valued harmonic functions in a similar manner.  We state the definition and basic facts about $C^{1,\mu}$ two-valued harmonic functions below and refer the reader to~\cite{SimWic16} and~\cite{KrumWic1} for more thorough discussions. 

\begin{definition} \label{twoval harmonic}
Let $\mu \in (0,1]$ and $\Omega \subseteq \mathbb{R}^n$ be open.  We say $\varphi \in C^{1,\mu}(\Omega,\mathcal{A}_2(\mathbb{R}^m))$ is \emph{locally harmonic} in $\Omega \setminus \mathcal{B}_{\varphi}$ if for every open ball $B \subseteq \Omega \setminus \mathcal{B}_{\varphi}$ there exists a pair of single-valued harmonic (hence real analytic) functions $\varphi_1,\varphi_2 : B \rightarrow \mathbb{R}^m$ such that $\varphi(X) = \{\varphi_1(X),\varphi_2(X)\}$ for all $X \in B$.  Clearly if such $\varphi_1,\varphi_2$ exist, they are unique. 
\end{definition}

Suppose $\varphi \in C^{1,\mu}(\Omega,\mathcal{A}_2(\mathbb{R}^m))$ is a non-zero, symmetric two-valued function which is locally harmonic in $\Omega \setminus \mathcal{B}_{\varphi}$.  For each $Y \in \Omega$, the \emph{frequency function} 
\begin{equation} \label{freqfn harmonic defn}
	N_{\varphi,Y}(\rho) = \frac{\rho^{2-n} \int_{B_{\rho}(Y)} |D\varphi|^2}{\rho^{1-n} \int_{\partial B_{\rho}(Y)} |\varphi|^2}
\end{equation}
is monotone non-decreasing as a function of $\rho \in (0,\op{dist}(Y,\partial \Omega))$, see~\cite[Lemma 2.2]{SimWic16}.  Thus we can define the \emph{frequency} $\mathcal{N}_{\varphi}(Y)$ of $\varphi$ at each point $Y \in \Omega$ by 
\begin{equation} \label{freq harmonic defn}
	\mathcal{N}_{\varphi}(Y) = \lim_{\rho \downarrow 0} N_{\varphi,Y}(\rho). 
\end{equation}

Let $\alpha \geq 0$ and $\varphi \in C^{1,\mu}(\mathbb{R}^n,\mathcal{A}_2(\mathbb{R}^m))$ be a non-zero, symmetric, harmonic  two-valued function.  By~\cite[Remark 2.4(a)]{SimWic16}, $N_{\varphi,0}(\rho) = \alpha$ for all $\rho > 0$ if and only if $\varphi$ is homogeneous degree $\alpha$, i.e.~$\varphi(\lambda X) = \{ \pm \lambda^{\alpha} \varphi_1(X) \}$ for all $X \in \mathbb{R}^n$ and all $\lambda > 0$ where $\varphi(X) = \{ \pm \varphi_1(X) \}$.  Suppose that $\varphi$ is homogeneous degree $\alpha$.  By~\cite[Remark 2.4(b)]{SimWic16}, the set 
\begin{equation*}
	S(\varphi) = \{ X \in \mathbb{R}^n : \mathcal{N}_{\varphi}(X) = \mathcal{N}_{\varphi}(0) = \alpha \}
\end{equation*}
is a linear subspace of $\mathbb{R}^n$ with the property that $\varphi(Z+X) = \varphi(X)$ for all $X \in \mathbb{R}^n$ and $Z \in S(\varphi)$.  Assuming that $0 \in \mathcal{K}_{\varphi}$, $\dim S(\varphi) \leq n-2$ with equality if and only if $\alpha = k/2$ for some integer $k \geq 3$ and after an orthogonal change of coordinates $\varphi(X) = \op{Re}(c (x_1+ix_2)^{k/2})$ for some $c \in \mathbb{C}^m \setminus \{0\}$.  We call such a function $\varphi$ \emph{cylindrical}.  Furthermore, by~\cite[Lemma 4.1]{SimWic16}, assuming that $0 \in \mathcal{K}_{\varphi}$, $\alpha \geq 3/2$ with equality if and only if after an orthogonal change of coordinates $\varphi(X) = \op{Re}(c (x_1+ix_2)^{3/2})$ for some $c \in \mathbb{C}^m \setminus \{0\}$. 

Using the monotonicity formula for frequency functions, it is shown in~\cite{SimWic16} that whenever $\varphi \in C^{1,\mu}(\Omega,\mathcal{A}_2(\mathbb{R}))$ is symmetric and locally harmonic in $\Omega \setminus \mathcal{B}_{\varphi}$: 
\begin{enumerate}
	\item[(i)] \cite[Lemma 4.1]{SimWic16}  either $\varphi \equiv \{0,0\}$ on $\Omega$ or $\dim_{\mathcal{H}}(\mathcal{K}_{\varphi}) \leq n-2$, and moreover either $\mathcal{B}_{\varphi} = \emptyset$ or $\dim_{\mathcal{H}}(\mathcal{B}_{\varphi}) = n-2$ and $\mathcal{H}^{n-2}(\mathcal{B}_{\varphi}) > 0$; 
	\item[(ii)] \cite[Lemmas 2.1, 3.2 and 4.1]{SimWic16} $\varphi \in C^{1,1/2}(\Omega,\mathcal{A}_2(\mathbb{R}^m)) \cap W^{2,2}_{\rm loc}(\Omega,\mathcal{A}_2(\mathbb{R}^m))$ and $\varphi$ satisfies the estimates  
	\begin{align} \label{harmonic schauder}
		&\sup_{B_{\rho/2}(Y)} |\varphi| + \rho \sup_{B_{\rho/2}(Y)} |D\varphi| + \rho^{3/2} [D\varphi]_{1/2,B_{\rho/2}(Y)} 
			\\&\hspace{15mm} + \rho^{2-n/2} \|D^2 \varphi\|_{L^2(B_{\rho/2}(Y) \setminus \mathcal{K}_{\varphi})} 
		\leq C \rho^{-n/2} \|\varphi\|_{L^2(B_{\rho}(Y))} \nonumber 
	\end{align}
	for each open ball $B_{\rho}(Y)$ with $\overline{B_{\rho}(Y)} \subset \Omega$, where $C = C(n,m) \in (0,\infty)$ is a constant. 
\end{enumerate}

\subsection{Two-valued stationary graphs}  Let $\mu \in (0,1]$ and $\Omega \subseteq \mathbb{R}^n$ be open.  For each two-valued function $u \in C^{1,\mu}(\Omega,\mathcal{A}_2(\mathbb{R}^m))$, the graph of $u$ is the set 
\begin{equation*}
	M = {\rm graph}\,u = \{ (X,Y) \in \Omega \times \mathbb{R}^m : Y = u_1(X) \text{ or } Y = u_2(X) \} ,
\end{equation*}
where $u(X) = \{u_1(X),u_2(X)\}$ for each $X \in \Omega$.  We can associate to $M$ an $n$-dimensional integral varifold $V = (M, \theta)$ where the multiplicity function $\theta : M \rightarrow \mathbb{Z}_+$ is defined by $\theta(X,u_1(X)) = 2$ if $u_1(X) = u_2(X)$ and $\theta(X,u_1(X)) = \theta(X,u_2(X)) = 1$ if $u_1(X) \neq u_2(X)$.  

\begin{definition} \label{twoval minimal graph}
Let $M$ be the graph of a two-valued function $u \in C^{1,\mu}(\Omega,\mathcal{A}_2(\mathbb{R}^m))$.  We say that $M$ is \emph{stationary} in $\Omega \times {\mathbb R}^{m}$ if the varifold $V = (M, \theta)$ is stationary in $\Omega \times {\mathbb R}^{m}$ with respect to the mass functional; that is to say,  
\begin{equation*}
	\int_M \op{div}_M \zeta \,\theta \,d\mathcal{H}^n = 0
\end{equation*}
for all $\zeta \in C^1_c(\Omega \times \mathbb{R}^m,\mathbb{R}^{n+m})$, where 
$\op{div}_M \zeta$ denotes the divergence of $\zeta$ computed with respect to the tangent plane to $M$ at $\mathcal{H}^n$-a.e.~point of $M$.
\end{definition}

\begin{definition}
Let $M$ be a $C^{1,\mu}$ two-valued stationary graph as in Definition~\ref{twoval minimal graph} and let $\theta$ be the multiplicity function associated with $M$.  We define: 
\begin{enumerate}
	\item[(a)] $\mathcal{Z}_M = \{P \in M : \theta(P) = 2 \}$;
	\item[(b)] $\mathcal{K}_M$ to be the set of points $P \in M$ such that $\theta(P) = 2$ and $M$ has a multiplicity two tangent plane at $P$; 
	\item[(c)] the \emph{singular set} $\op{sing} M$ of $M$ to be the set of all points $P \in M$ such that there is no radius $\delta > 0$ such that $M \cap B^{n+m}_{\delta}(P)$ is a finite union of smoothly embedded submanifolds.
\end{enumerate}
\end{definition} 

Notice that $\op{sing} M \subseteq \mathcal{K}_M \subseteq \mathcal{Z}_M$ and in particular each point of $\op{sing} M$ is a branch point, i.e.~a singular point at which $M$ has a multiplicity two tangent plane.  
The respective orthogonal projections of the sets $\mathcal{Z}_M$, $\mathcal{K}_M$, $\op{sing} M$ onto $\mathbb{R}^n \times \{0\}$ are the sets $\mathcal{Z}_u$, $\mathcal{K}_u$, $\mathcal{B}_u$ as in Definition~\ref{singular sets defn}.   

Since $M = {\rm graph}\,u$ is stationary, $u$ is a two-valued solution to the minimal surface system on $\Omega \setminus \mathcal{B}_u$ in the sense that for every open ball $B \subset \Omega \setminus \mathcal{B}_u$, $u(X) = \{u_1(X),u_2(X)\}$ for real analytic single-valued solutions $u_1,u_2 : B \rightarrow \mathbb{R}^m$ to the minimal surface system
\begin{align} 
	\label{mss1a} D_i (G^{ij}(Du_l)) &= 0 \text{ in } B \text{ for } j = 1,2,\ldots,n, \,l = 1,2, \\
	\label{mss1b} D_i (G^{ij}(Du_l) \,D_j u_l^{\kappa}) &= 0 \text{ in } B \text{ for } \kappa = 1,\ldots,m, \, l = 1,2, 
\end{align}
where $G^{ij}(p) = \sqrt{\det g(p)} \,g^{ij}(p)$ letting $g(p) = (g_{ij}(p))$ be the $n \times n$ matrix with entries $g_{ij}(p) = \delta_{ij} + p_i p_j$ where $\delta_{ij}$ denotes the Kronecker delta and $(g^{ij}(p))$ is be the inverse matrix of $g(p)$.  
By taking differences in \eqref{mss1b}, we show that $w = u_s$ satisfies 
\begin{equation} \label{mss3} 
	\Delta w^{\kappa} = -D_i (b^{ij}_{\kappa\lambda}(Du_a,Du_s) \,D_j w^{\lambda}) = f_{\kappa} 
		\text{ in } \Omega \setminus \mathcal{B}_u \text{ for } \kappa = 1,\ldots,m , 
\end{equation}
where 
\begin{equation} \label{mss b defn} 
	b_{\kappa \lambda}^{ij}(p,q) = \frac{G^{ij}(p+q) + G^{ij}(p-q)}{2} \,\delta_{\kappa \lambda} - \delta_{ij} \delta_{\kappa \lambda} 
		+ \frac{1}{2} \int_{-1}^1 (D_{p^{\lambda}_j} G^{ik})(p+tq) \,p^{\kappa}_k \,dt .
\end{equation}
for all $m \times n$ real-valued matrices $p,q$.  Since $b^{ij}_{\kappa\lambda}(p,q) = b^{ij}_{\kappa\lambda}(p,-q)$ for all $m \times n$ real-valued matrices $p,q$, $b^{ij}_{\kappa\lambda}(Du_a, Du_s)$ is a well-defined single-valued function.  Throughout we shall interpret \eqref{mss3} as meaning that for each open ball $B \subset \Omega \setminus \mathcal{B}_w$, $w = \{\pm w_1\}$ in $B$ for some smooth single-valued function $w_1 : B \rightarrow \mathbb{R}^m$ and $\Delta w_1^{\kappa} = f_{1,\kappa}$ in $B$ where $f_{1,\kappa} = -D_i (b^{ij}_{\kappa\lambda}(Du_a,Du_s) \,D_j w_1^{\lambda})$.  By \eqref{mss b defn}, 
\begin{equation} \label{mss b est} 
	|b_{\kappa \lambda}^{ij}(p,q)| \leq C (|p|^2 + |q|^2), \quad 
	|D_p b_{\kappa \lambda}^{ij}(p,q)| \leq C |p|, \quad 
	|D_q b_{\kappa \lambda}^{ij}(p,q)| \leq C |q|,
\end{equation}
for all $m \times n$ real-valued matrices $p,q$ with $|p|,|q| \leq 1$ and some constant $C = C(n,m) \in (0,\infty)$.  Thus 
\begin{equation} \label{mss3 f est} 
	|f_{\kappa}| \leq C (|Du_a|\,|D^2 u_a| + |Du_s|\,|D^2 u_s|) |Dw| + C (|Du_a|^2 + |Du_s|^2) |D^2 w| 
\end{equation}
for some constant $C = C(n,m) \in (0,\infty)$.

By localizing $M$ at a point of $\mathcal{K}_M$, let us assume that $M$ is the graph of a two-valued function $u \in C^{1,\mu}(B_1(0),\mathcal{A}_2(\mathbb{R}))$ with $\|u\|_{C^{1,\mu}(B_1(0))} \leq \varepsilon(n,m)$.  Then: 
\begin{enumerate}
	\item[(i)] \cite[Theorem 8.10]{SimWic16}  Either $\mathcal{K}_M = M$ or $\dim_{\mathcal{H}}(\mathcal{K}_M) \leq n-2$, and moreover either $\op{sing} M = \emptyset$ or $\dim_{\mathcal{H}}(\op{sing} M) = n-2$ and $\mathcal{H}^{n-2}(\op{sing} M) > 0$.
\end{enumerate}
In particular, the multiplicity function $\theta$ associated with $M$ satisfies either $\theta(P) = 2$ for all $P \in M$, i.e.~$M$ is a single-valued graph with multiplicity two, or $\theta(P) = 1$ for $\mathcal{H}^{n-1}$-a.e.~$P \in M$.  Furthermore: 
\begin{enumerate}
	\item[(ii)] \cite[Lemma 3.2, Lemma 5.14, Theorem 7.1 and Theorem 7.4]{SimWic16}  $u_a \in C^{1,1}(B_1(0),\mathcal{A}_2(\mathbb{R}^m))$ and $u_s \in C^{1,1/2}(B_1(0),\mathcal{A}_2(\mathbb{R}^m)) \cap W^{2,2}_{\rm loc}(B_1(0),\mathcal{A}_2(\mathbb{R}^m))$ and satisfy the estimates 
\begin{align} 
		\label{minimal avg schauder} &\sup_{B_{\rho/2}(Y)} |u_a| + \rho \sup_{B_{\rho/2}(Y)} |Du_a| + \rho^2 \sup_{B_{\rho/2}(Y)} |D^2 u_a| 
			\leq C \rho^{-n/2} \|u_a - u_a(Y)\|_{L^2(B_{\rho}(Y))}, \\ 
		\label{minimal sym schauder} &\sup_{B_{\rho/2}(Y)} |u_s| + \rho \sup_{B_{\rho/2}(Y)} |Du_s| + \rho^{3/2} [Du_s]_{1/2,B_{\rho/2}(Y)}  
			\\&\hspace{30mm} + \rho^{2-n/2} \|D^2 u_s\|_{L^2(B_{\rho/2}(Y) \setminus \mathcal{K}_u)} \leq C \rho^{-n/2} \|u_s\|_{L^2(B_{\rho}(Y))} \nonumber 
	\end{align}
	for each open ball $B_{\rho}(Y)$ with $\overline{B_{\rho}(Y)} \subset B_1(0)$ and some constant $C = C(n,m,\mu) \in (0,\infty)$ and 
	\begin{equation} \label{minimal sym decay} 
		|u_s(X)| + d(X) |Du_s(X)| + d(X)^2 |D^2 u_s(X)| \leq C d(X)^{3/2} \|u_s\|_{L^2(B_1(0))} 
	\end{equation}
	for each $X \in B_{1/2}(0)$ with $d(X) = \op{dist}(X,\mathcal{K}_u) \leq 1/4$ and some constant $C = C(m,n,\mu) \in (0,\infty)$. 
\end{enumerate} 
The approach of~\cite{SimWic16} involves approximating two-valued stationary graphs by the graphs of $C^{1,\mu}$ two-valued harmonic functions and proving monotonicity formulas for various frequency functions associated with a two-valued stationary graph.  We will describe these monotonicity formulas in Section~\ref{sec:frequency sec}.  First let us discuss in the next section how we will represent a two-valued stationary graph $M$ as the graph of a two-valued function $\widetilde{u}_P$ over the tangent plane at each point $P \in \mathcal{K}_M$ (including branch points), as we will define frequency functions and blow-ups in terms of the two-valued functions $\widetilde{u}_P$.

\section{Local graphical representation relative to tangent planes at branch points} \label{sec:rotation sec}

Given an $n$-dimensional linear subspace $\Pi \subseteq \mathbb{R}^{n+m}$, $X_0 \in \Pi$, and $\rho > 0$, we define the open ball $\mathbf{B}_{\rho}(X_0,\Pi)$ and open cylinder $\mathbf{C}_{\rho}(X_0,\Pi)$ by 
\begin{align*}
	\mathbf{B}_{\rho}(X_0,\Pi) &= \{ X \in \Pi : |X - X_0| < \rho \} , \\
	\mathbf{C}_{\rho}(X_0,\Pi) &= \{ X + Y : X \in \Pi, \,|X - X_0| < \rho, \,Y \in \Pi^{\perp} \} ,
\end{align*}
where $\Pi^{\perp}$ is the orthogonal complement of $\Pi$.

Let $\Omega \subseteq \mathbb{R}^n$ be open and $M$ be the stationary graph of a two-valued function $u \in C^{1,1/2}(\Omega,\mathcal{A}_2(\mathbb{R}^m))$.  For each $P \in \mathcal{K}_M$, $M$ has a unique tangent plane, which we denote by $T_P M$.  We let $\pi_P : \mathbb{R}^{n+m} \rightarrow T_P M$ denote the orthogonal projection map onto $T_P M$ and $T_P M^{\perp}$ denote the orthogonal complement of $T_P M$ in $\mathbb{R}^{n+m}$.  We let $\pi_0 : \mathbb{R}^{n+m} \rightarrow \mathbb{R}^n$ denote the orthogonal projection map onto $\mathbb{R}^n \cong \mathbb{R}^n \times \{0\}$ and $\pi_0^{\perp} : \mathbb{R}^{n+m} \rightarrow \mathbb{R}^m$ denote the orthogonal projection onto $\mathbb{R}^m \cong \{0\} \times \mathbb{R}^m$.  Notice that for each $P \in \mathcal{K}_M$, 
\begin{equation} \label{tilt tangent planes}
	\|\pi_P - \pi_0\| \leq C(n,m) \,|Du(\pi_0 P)| .
\end{equation}

We define 
$$\mathcal{A}_2(T_P M^{\perp}) = \{ a = \{a_1,a_2\} \in \mathcal{A}_2(\mathbb{R}^{n+m}) : a_1,a_2 \in T_P M^{\perp} \}.$$  
Given any open set $\widetilde{\Omega} \subseteq T_P M$ and a  function $\widetilde{u} : \widetilde{\Omega} \rightarrow \mathcal{A}_2(T_P M^{\perp})$, the graph of $\widetilde{u}$ is 
\begin{equation*}
	{\rm graph}\, \widetilde{u} = \{ X + Y : X \in \widetilde{\Omega}, \,Y = \widetilde{u}_{1}(X) \text{ or } Y = \widetilde{u}_{2}(X) \} , 
\end{equation*}
where $\widetilde{u}(X) = \{\widetilde{u}_{1}(X),\widetilde{u}_{2}(X)\}$ for each $X \in \widetilde{\Omega}$. 

\begin{definition} \label{tildeu defn}
Given a two-valued stationary graph $M$ (as in Definition~\ref{twoval minimal graph}), $P \in \mathcal{K}_M$, and $\rho > 0$, we let $\widetilde{u}_P \in C^{1,1/2}(\mathbf{B}_{\rho}(0,T_P M), \mathcal{A}_2(T_P M^{\perp}))$ be the two-valued function such that 
\begin{equation} \label{tildeu graph} 
	{\rm graph}\, \widetilde{u}_P = (M - P) \cap \mathbf{C}_{\rho}(0,T_P M) 
\end{equation}
assuming $\rho$ is small enough that such a two-valued function $\widetilde{u}_P$ exists. 
\end{definition}
\begin{rmk}\label{local-rep}
Notice that, since $M$ is a $C^{1,1/2}$ two-valued graph, for each $P \in \mathcal{K}_{M}$ there always exists $\rho > 0$ such that $\widetilde{u}_P$ as in Definition~\ref{tildeu defn} exists and that for each such $\rho$, $\widetilde{u}_P$ is uniquely defined by \eqref{tildeu graph}.  Moreover, for every $\gamma \in (0,1)$ there is a $\varepsilon = \varepsilon(n,m,\gamma) > 0$ such that if $\|Du\|_{C^0(\Omega)} \leq \varepsilon$, then for each $P \in \mathcal{K}_M$ and $\rho_0 > 0$ such that $B^{n+m}_{\rho_0}(P) \subseteq \Omega \times \mathbb{R}^n$ there exists a two-valued function $\widetilde{u}_P \in C^{1,1/2}(\mathbf{B}_{\gamma \rho_0}(0,T_P M), \mathcal{A}_2(T_P M^{\perp}))$ such that \eqref{tildeu graph} holds true with $\rho = \gamma \rho_0$ and 
\begin{gather*}
	\rho_0^{-1} \sup_{\mathbf{B}_{\gamma \rho_0}(0,T_P M)} |\widetilde{u}_P| + \sup_{\mathbf{B}_{\gamma \rho_0}(0,T_P M)} |D\widetilde{u}_P| 
		\leq C(n,m) \sup_{B_{\rho_0}(\pi_0 P)} |Du| \leq C(n,m) \,\varepsilon, \\
	\rho_0^{1/2} [D\widetilde{u}_P]_{1/2,\mathbf{B}_{\gamma \rho_0}(0,T_P M)} 
		\leq C(n,m,\gamma) \left( \sup_{B_{\rho_0}(\pi_0 P)} |Du| + \rho_0^{1/2} [Du]_{1/2,B_{\rho_0}(\pi_0 P)} \right) . 
\end{gather*}
\end{rmk}
For each $X \in \mathbf{B}_{\rho}(0,T_P M)$ we write $\widetilde{u}_P(X) = \{\widetilde{u}_{P,1}(X),\widetilde{u}_{P,2}(X)\}$ where $\widetilde{u}_{P,1}(X),\widetilde{u}_{P,2}(X) \in T_P M^{\perp}$.  We let $$\widetilde{u}_{P,a}(X) = (\widetilde{u}_{P,1}(X) + \widetilde{u}_{P,2}(X))/2$$ denote the average of $\widetilde{u}_P$ and $$\widetilde{u}_{P,s}(X) = \{ \pm (\widetilde{u}_{P,1}(X) - \widetilde{u}_{P,2}(X))/2 \}$$ denote the symmetric part of $\widetilde{u}_P$ (as in \eqref{avg and free defn}). 

In order to compare $\widetilde{u}_P$ and $u$, which might have domains in different planes, we define the following function $\widehat{u}_P$.

\begin{definition} \label{hatu defn}
Let $M$ be a two-valued stationary graph (as in Definition~\ref{twoval minimal graph}), $P \in \mathcal{K}_M$, and $\rho > 0$.  Let $Q_P$ be a rotation of $\mathbb{R}^{n+m}$ such that 
\begin{equation} \label{hatu rotation}
	Q_P (T_P M) = \mathbb{R}^n \times \{0\}, \quad \|Q_P - I\| \leq C(n,m) \,|Du(\pi_0 P)| .
\end{equation}
We associate $Q_P$ with a two-valued function $\widehat{u}_P \in C^{1,1/2}(B_{\rho}(\pi_0 P), \mathcal{A}_2(\mathbb{R}^m))$ such that 
\begin{equation} \label{hatu graph}
	{\rm graph}\, \widehat{u}_P = P + (Q_P(M - P)) \cap B_{\rho}(0) \times \mathbb{R}^m 
\end{equation}
assuming $\rho$ is small enough that such a two-valued function $\widehat{u}_P$ exists. 
\end{definition}

Note that the rotation $Q_P$ is not uniquely defined by \eqref{hatu rotation} and thus we must associate $\widehat{u}_P$ with a particular rotation $Q_P$.  If $\rho > 0$ is small enough that $\widetilde{u}_P$ as in Definition~\ref{tildeu defn} exists, then $\widehat{u}_P$ also exists.  
For each $X \in B_{\rho}(\pi_0 P)$, we express $\widehat{u}_P$ as $\widehat{u}_P(X) = \{\widehat{u}_{P,1}(X),\widehat{u}_{P,2}(X)\}$ where $\widehat{u}_{P,1}(X),\widehat{u}_{P,2}(X) \in \mathbb{R}^m$.  We let $\widehat{u}_{P,a}(X) = (\widehat{u}_{P,1}(X) + \widehat{u}_{P,2}(X))/2$ denote the average of $\widehat{u}_P$ and $\widehat{u}_{P,s}(X) = \{ \pm (\widehat{u}_{P,1}(X) - \widehat{u}_{P,2}(X))/2 \}$ denote the symmetric part of $\widehat{u}_P$ (as in \eqref{avg and free defn}). 

We have the following result, Lemma~\ref{rotate lemma}, which we will use to compare the symmetric parts of $\widehat{u}_P$ and $u$.  (Note that we state Lemma~\ref{rotate lemma} in terms of general two-valued functions $u, \widehat{u}$.  We will often apply Lemma~\ref{rotate lemma} with $\widehat{u} = \widehat{u}_P$, but in a few places we also apply Lemma~\ref{rotate lemma} with $\widehat{u}_P,u$ in place of $u,\widehat{u}$.) 

\begin{lemma} \label{rotate lemma}
There is $\varepsilon = \varepsilon(n,m) \in (0,1)$ such that the following holds true.  Let $u \in C^{1,1/2}(B_1(0),$ $\mathcal{A}_2(\mathbb{R}^m))$ such that $M = {\rm graph}\,u$ is stationary in $B_1(0) \times \mathbb{R}^m$ and $\|u\|_{C^{1,1/2}(B_1(0))} \leq \varepsilon$.  Let $P \in \mathcal{K}_M \cap B_{1/2}(0) \times \mathbb{R}^m$, $Q$ be a rotation of $\mathbb{R}^{n+m}$ such that $\|Q-I\| \leq \varepsilon$ and $\widehat{u} \in C^{1,1/2}(B_{1/4}(\pi_0 P),$ $\mathcal{A}_2(\mathbb{R}^m))$ such that 
\begin{equation} \label{rotate graph}
	{\rm graph}\, \widehat{u} = P + (Q (M - P)) \cap B_{1/4}(0) \times \mathbb{R}^m. 
\end{equation}
Then 
\begin{equation} \label{rotate concl} 
	\sup_{B_{1/4}(\pi_0 P)} \mathcal{G}(\widehat{u}_s, u_s) \leq C(n,m) \,\|Q - I\| \,\|u_s\|_{L^2(B_1(0))}, 
\end{equation}
where $u_s$ and $\widehat{u}_s$ are the symmetric parts of $u$ and $\widehat{u}$ (as in \eqref{avg and free defn}).
\end{lemma}

\begin{proof}
Assume that $\mathcal{K}_M \neq M$, as otherwise $u_s, \widehat{u}_s$ are identically zero and thus \eqref{rotate concl} holds true.  By translating we may assume that $P = 0$.  This translates the domain of $u$, so we must now restrict $u$ to $B_{1/2}(0)$ and show that 
\begin{equation*} 
	\sup_{B_{1/4}(0)} \mathcal{G}(\widehat{u}_s, u_s) \leq C(n,m) \,\|Q - I\| \,\|u_s\|_{L^2(B_{1/2}(0))}.
\end{equation*}

Express $Q$ as a block matrix 
\begin{equation*} 
	Q = \left(\begin{matrix} Q_{11} & Q_{12} \\ Q_{21} & Q_{22} \end{matrix}\right)
\end{equation*}
where $Q_{11}$, $Q_{12}$, $Q_{21}$ and $Q_{22}$ are matrices of dimensions $n \times n$, $n \times m$, $m \times n$, and $m \times m$ respectively.  Let $X \in B_{1/4}(0) \setminus \mathcal{Z}_{\widehat{u}}$.  Set $d = d(X) = \op{dist}(X,\mathcal{Z}_{\widehat{u}})$ and express $\widehat{u} = \{\widehat{u}_1,\widehat{u}_2\}$ in $B_{d}(X)$ for single-valued functions $\widehat{u}_1, \widehat{u}_2 \in C^{1,1/2}(B_{d}(X);\mathbb{R}^m)$.  For each $l = 1,2$ define $\xi_l : B_{d}(X) \rightarrow \mathbb{R}^n$ by 
\begin{equation} \label{rotate eqn2} 
	\xi_l(Y) = \pi_0 Q^{-1} (Y,\widehat{u}_l(Y)) 
\end{equation}
for all $Y \in B_{d}(X)$.  
For $l = 1,2$ let $u_l : \xi_l(B_d(X)) \rightarrow \mathbb{R}^m$ be functions such that  
\begin{equation} \label{rotate eqn4} 
	\left(\begin{matrix} Y \\ \widehat{u}_l(Y) \end{matrix}\right) 
		= Q \left(\begin{matrix} \xi_l(Y) \\ u_l(\xi_l(Y)) \end{matrix}\right) 
		= \left(\begin{matrix} Q_{11} \,\xi_l(Y) + Q_{12} \,u_l(\xi_l(Y)) \\ 
			Q_{21} \,\xi_l(Y) + Q_{22} \,u_l(\xi_l(Y)) \end{matrix}\right) 
\end{equation}
for each $Y \in B_d(X)$.  By taking the difference of \eqref{rotate eqn4} for $l = 1,2$, 
\begin{align}
	\label{rotate eqn5} 0 &= Q_{11} \,\frac{\xi_1(Y) - \xi_2(Y)}{2} + Q_{12} \,\frac{u_1(\xi_1(Y)) - u_2(\xi_2(Y))}{2}, \\
	\label{rotate eqn6} \frac{\widehat{u}_1(Y) - \widehat{u}_2(Y)}{2} &= Q_{21} \,\frac{\xi_1(Y) - \xi_2(Y)}{2} 
		+ Q_{22} \,\frac{u_1(\xi_1(Y)) - u_2(\xi_2(Y))}{2} .
\end{align}
By solving for $(\xi_1(Y) - \xi_2(Y))/2$ in \eqref{rotate eqn5}, 
\begin{equation} \label{rotate eqn7} 
	\frac{\xi_1(Y) - \xi_2(Y)}{2} = -Q_{11}^{-1} \,Q_{12} \,\frac{u_1(\xi_1(Y)) - u_2(\xi_2(Y))}{2}, 
\end{equation}
which substituted into \eqref{rotate eqn6} gives us 
\begin{equation} \label{rotate eqn8} 
	\frac{\widehat{u}_1(Y) - \widehat{u}_2(Y)}{2} = (Q_{22} - Q_{21} \,Q_{11}^{-1} \,Q_{12}) \,\frac{u_1(\xi_1(Y)) - u_2(\xi_2(Y))}{2}.  
\end{equation}
By \eqref{rotate eqn7} and \eqref{rotate eqn8}, 
\begin{equation*}
	|\xi_1(X) - \xi_2(X)| \leq C(n,m) \,\varepsilon \,|\widehat{u}_s(X)| \leq C(n,m) \,\varepsilon \,d^{3/2} < d/8
\end{equation*}
provided $\varepsilon$ is sufficiently small.  Hence \eqref{rotate eqn4} holds true for all $Y \in B_{d/4}(\xi_1(X))$ and all $l = 1,2$.  Moreover, ${\rm graph}\,\widetilde{u} \cap (B_{d}(X) \times \mathbb{R}^m)$ is the union of the mutually disjoint graphs of the single-valued functions $\widetilde{u}_1,\widetilde{u}_2$.  Rotating by $Q$ gives us that ${\rm graph}\,u \cap (B_{d/4}(\xi_1(X)) \times \mathbb{R}^m)$ is the union of the mutually disjoint graphs of the single-valued functions $u_1,u_2$ and thus $u(Y) = \{u_1(Y),u_2(Y)\}$ for all $Y \in B_{d/4}(\xi_1(X))$. 

By \eqref{rotate eqn7} (with $Y = X$) and $\|Du\|_{C^0(B_{1/2}(0))} \leq \varepsilon$, 
\begin{align} \label{rotate eqn9} 
	&\left| \frac{u_1(\xi_1(X)) - u_2(\xi_2(X))}{2} - \frac{u_1(\xi_a(X)) - u_2(\xi_a(X))}{2} \right|
	\\&\hspace{15mm} \leq C(n,m) \,\varepsilon \,\frac{|\xi_1(X) - \xi_2(X)|}{2} 
	\leq C(n,m) \,\varepsilon \,\|Q - I\| \,\frac{|u_1(\xi_1(X)) - u_2(\xi_2(X))|}{2}, \nonumber 
\end{align}
where $\xi_a(X) = (\xi_1(X) + \xi_2(X))/2$.  Thus provided $\varepsilon$ is sufficiently small, by \eqref{rotate eqn8}, \eqref{rotate eqn9} and $\|Q - I\| \leq \varepsilon$, 
\begin{align} \label{rotate eqn10} 
	\mathcal{G}(\widehat{u}_s(X), u_s(\xi_a(X))) 
	&\leq \left| \frac{\widehat{u}_1(X) - \widehat{u}_2(X)}{2} - \frac{u_1(\xi_a(X)) - u_2(\xi_a(X))}{2} \right| 
	\\&\leq C(n,m) \,\|Q - I\| \,|u_s(\xi_a(X))| . \nonumber 
\end{align}

By summing \eqref{rotate eqn4} for $l = 1,2$, 
\begin{align}
	\label{rotate eqn11} X &= Q_{11} \,\xi_a(X) + Q_{12} \,\frac{u_1(\xi_1(X)) + u_2(\xi_2(X))}{2}, \\
	\label{rotate eqn12} \widehat{u}_a(X) &= Q_{21} \,\xi_a(X) + Q_{22} \,\frac{u_1(\xi_1(X)) + u_2(\xi_2(X))}{2}. 
\end{align}
Assuming $\varepsilon$ is sufficiently small, solving for $\xi_a(X)$ using \eqref{rotate eqn11} and \eqref{rotate eqn12} gives us 

\begin{equation} \label{rotate eqn13} 
	\xi_a(X) = (Q_{11} - Q_{12} \,Q_{22}^{-1} \,Q_{21})^{-1} (X - Q_{12} \,Q_{22}^{-1} \,\widehat{u}_a(X)) . 
\end{equation}
In particular, \eqref{rotate eqn13} expresses $\xi_a(X)$ as a $C^{1,1}$-function of $X \in B_{1/4}(0)$ for which 
\begin{equation} \label{rotate eqn14} 
	|\xi_a(X) - X| + |D\xi_a(X) - I| \leq C(n,m) \,\|Q - I\| .
\end{equation}
Using \eqref{rotate eqn14}, 
\begin{equation*} 
	\mathcal{G}(u_s(\xi_a(X)), u_s(X)) \leq |\xi_a(X) - X| \sup_{B_{3/8}(0)} |Du_s| \leq C(n,m) \,\|Q - I\| \sup_{B_{3/8}(0)} |Du_s| ,
\end{equation*}
which together with \eqref{rotate eqn10} gives us 
\begin{equation*} 
	\mathcal{G}(\widehat{u}_s(X), u_s(X)) \leq C(n,m) \,\|Q - I\| \left( \sup_{B_{3/8}(0)} |u_s| + \sup_{B_{3/8}(0)} |Du_s| \right) .
\end{equation*}
By the estimates \eqref{minimal sym schauder}, we obtain 
\begin{equation} \label{rotate eqn15} 
	\mathcal{G}(\widehat{u}_s(X), u_s(X)) \leq C(n,m) \,\|Q - I\| \,\|u_s\|_{L^2(B_{1/2}(0))} 
\end{equation}
for all $X \in B_{1/4}(0) \setminus \mathcal{Z}_u$.  By continuity, \eqref{rotate eqn15} holds true for all $X \in B_{1/4}(0)$.
\end{proof}

\section{Monotonicity of frequency functions and doubling conditions} \label{sec:frequency sec}

Here we present the basic theory of frequency functions of  two-valued stationary graphs as developed in~\cite{SimWic16}.  Let $M$ be a two-valued stationary graph (as in Definition~\ref{twoval minimal graph}) and assume that $\mathcal{K}_M \neq M$.  Consider any point $P \in \mathcal{K}_M$.  Let $\rho_0 > 0$ and $\widetilde{u}_P \in C^{1,1/2}(\mathbf{B}_{\rho_0}(0,T_P M),\mathcal{A}_2(T_P M^{\perp}))$ such that 
\begin{equation*}
	{\rm graph}\, \widetilde{u}_P = (M - P) \cap \mathbf{C}_{\rho_0}(0,T_P M) 
\end{equation*}
as in Definition~\ref{tildeu defn}.  Let $\widetilde{u}_{P,a}$ and $\widetilde{u}_{P,s}$ be the average and symmetric part of $\widetilde{u}_P$ (as in \eqref{avg and free defn}).   We will define the frequency function, frequency, and tangent functions of $M$ at $P$ in terms of $\widetilde{u}_{P,s}$.  

\begin{definition} \label{freqfn defn}
Let $M$ be a two-valued stationary graph (as in Definition~\ref{twoval minimal graph}) such that $\mathcal{K}_M \neq M$.  Let $P \in \mathcal{K}_M$, $\widetilde{u}_P \in C^{1,1/2}(\mathbf{B}_{\rho_0}(0,T_P M),\mathcal{A}_2(T_P M^{\perp}))$ be as in Definition~\ref{tildeu defn}, and $\widetilde{u}_{P,s}$ be the symmetric part of $\widetilde{u}_P$ (as in \eqref{avg and free defn}).  The \emph{frequency function} $N_{M,P}$ associated with $M$ and $P$ is defined by 
\begin{equation*}
	N_{M,P}(\rho) = \frac{D_{M,P}(\rho)}{H_{M,P}(\rho)}
\end{equation*}
for all $\rho \in (0,\rho_0)$ with $H_{M,P}(\rho) \neq 0$, where 
\begin{equation*}
	D_{M,P}(\rho) = \rho^{2-n} \int_{B_{\rho}(0)} |D\widetilde{u}_{P,s}|^2 \text{ and } 
	H_{M,P}(\rho) = \rho^{1-n} \int_{\partial B_{\rho}(0)} |\widetilde{u}_{P,s}|^2
\end{equation*}
for all $\rho \in (0,\rho_0)$.  The \emph{frequency} $\mathcal{N}_{M}(p)$ of $M$ at $P$ is defined by 
\begin{equation*}
	\mathcal{N}_M(P) = \lim_{\rho \downarrow 0} N_{M,P}(\rho)
\end{equation*}
whenever the limit exists.
\end{definition}

\begin{definition} \label{tan fn defn}
Let $M$ be a two-valued stationary graph (as in Definition~\ref{twoval minimal graph}) such that $\mathcal{K}_M \neq M$.  Let $P \in \mathcal{K}_M$, $\widetilde{u}_P \in C^{1,1/2}(\mathbf{B}_{\rho_0}(0,T_P M),\mathcal{A}_2(T_P M^{\perp}))$ be as in Definition~\ref{tildeu defn}, and $\widetilde{u}_{P,s}$ be the symmetric part of $\widetilde{u}_P$ (as in \eqref{avg and free defn}).  For each $\rho \in (0,\rho_0]$, we define $\widetilde{u}_{P,s,\rho} \in C^{1,1/2}(\mathbf{B}_{\rho_0/\rho}(0,T_P M),$ $\mathcal{A}_2(T_P M^{\perp}))$ by 
\begin{equation} \label{tan fn scaling}
	\widetilde{u}_{P,s,\rho}(X) = \frac{\widetilde{u}_{P,s}(\rho X)}{\rho^{-n/2} \|\widetilde{u}_{P,s}\|_{L^2(\mathbf{B}_{\rho}(0,T_P M))}} 
\end{equation}
for each $X \in \mathbf{B}_{\rho_0/\rho}(0,T_P M)$.  We say $\varphi \in C^{1,1/2}(T_P M, \mathcal{A}_2(T_P M^{\perp}))$ is a \emph{tangent function} to $M$ at $P$ if there exists a sequence of positive numbers $\rho_j$ with $\rho_j \rightarrow 0^+$ such that 
\begin{equation*}
	\widetilde{u}_{P,s,\rho_j} \rightarrow \varphi
\end{equation*}
in the $C^1$-topology on compact subsets of $T_P M$.
\end{definition}

We claim that $N_{M,P}$ satisfies a monotonicity formula and consequently the frequency $\mathcal{N}_M(P)$ exists and at least one tangent function $\varphi$ of $M$ at $P$ exist.  (Note that we do not assert that $\varphi$ is unique, i.e.\  independent of the sequence $\rho_j$.)

\begin{theorem} \label{freq_mono_thm}
Let $M$ be a $C^{1,1/2}$ two-valued stationary graph (as in Definition~\ref{twoval minimal graph}) and assume that $\mathcal{K}_M \neq M$.  For every $P \in \mathcal{K}_M$ there exists $\rho_0 > 0$ (depending on $M$ and $P$) such that: 
\begin{enumerate}
	\item[(i)] there exists $\widetilde{u}_P \in C^{1,1/2}(\mathbf{B}_{\rho_0}(0,T_P M), \mathcal{A}_2(T_P M^{\perp}))$ such that 
\begin{equation*}
	M \cap \mathbf{C}_{\rho_0}(0,T_P M) = {\rm graph}\, \widetilde{u}_P;
\end{equation*}
	\item[(ii)] $H_{M,0}(\rho) \neq 0$ for all $\rho \in (0,\rho_0]$ and $e^{C \rho^{1/2}} N_{M,0}(\rho)$ is monotone non-decreasing as a function of $\rho \in (0,\rho_0]$, where $C \in (0,\infty)$ is a constant (depending on $M$ and $P$ and independent of $\rho$); 
	\item[(iii)] the frequency $\mathcal{N}_M(P) = \lim_{\rho \downarrow 0} N_{M,P}(\rho)$ exists and $\mathcal{N}_M(P) \geq 3/2$; 
	\item[(iv)] for each sequence of positive numbers $\rho_j$ with $\rho_j \rightarrow 0^+$ there exists a subsequence $(\rho_{j'}) \subset (\rho_j)$ and a non-zero, symmetric, homogeneous degree $\mathcal{N}_M(P)$, harmonic two-valued function $\varphi \in C^{1,1/2}(T_P M,$ $\mathcal{A}_2(T_P M^{\perp}))$ such that $\widetilde{u}_{P,s,\rho_{j'}} \rightarrow \varphi$ in the $C^1$-topology on compact subsets of $\mathbb{R}^n$.
\end{enumerate}
\end{theorem}

To prove Theorem~\ref{freq_mono_thm}, we may translate and rotate so that $P = 0$ and $T_P M = \mathbb{R}^n \times \{0\}$ and thereby assume that $M = {\rm graph}\,u$ where $u \in C^{1,1/2}(B_1(0),\mathcal{A}_2(\mathbb{R}^m))$ such that $u(0) = \{0,0\}$ and $Du(0) = \{0,0\}$.  We have the following monotonicity formula~\cite[Lemma 6.6]{SimWic16}: 

\begin{lemma} \label{freq_mono_lemma}
For every $\gamma \in (0,\infty)$ there exists $\varepsilon = \varepsilon(n,m,\gamma) > 0$ such that if $u \in C^{1,1/2}(B_1(0),\mathcal{A}_2(\mathbb{R}^m))$ is a two-valued function such that $M = {\rm graph}\,u$ is stationary in $B_1(0) \times \mathbb{R}^m$, $\mathcal{K}_M \neq M$, $u(0) = \{0,0\}$, $Du(0) = \{0,0\}$, $[Du]_{1/2,B_1(0)} \leq \varepsilon$ and 
\begin{equation} \label{freq_mono_doubling_hyp}
	\int_{B_{\rho}(0)} |u_s|^2 \leq \gamma \int_{B_{\rho/2}(0)} |u_s|^2 
\end{equation}
for all $\rho \in (0,1]$ (where $u_s$ is as in \eqref{avg and free defn}), then: 
\begin{enumerate}
	\item[(i)] $H_{M,0}(\rho) \neq 0$ for all $\rho \in (0,1/2]$ and $e^{C \rho^{1/2}} N_{M,0}(\rho)$ is monotone non-decreasing as a function of $\rho \in (0,1/2]$, where $C = C_1(n,m,\gamma) \,[Du]_{1/2,B_1(0)}$ is a constant (independent of $\rho$); 
	\item[(ii)] the frequency $\mathcal{N}_M(0) = \lim_{\rho \downarrow 0} N_{M,0}(\rho)$ exists and $\mathcal{N}_M(0) \geq 3/2$; 
	\item[(iii)] for each sequence of positive numbers $\rho_j$ with $\rho_j \rightarrow 0^+$ there exists a subsequence $(\rho_{j'}) \subset (\rho_j)$ and a non-zero, symmetric, homogeneous degree $\mathcal{N}_M(0)$, harmonic two-valued function $\varphi \in C^{1,1/2}(\mathbb{R}^n,$ $\mathcal{A}_2(\mathbb{R}^m))$ such that 
\begin{equation*}
	\frac{u_s(\rho_{j'} X)}{\rho_{j'}^{-n/2} \|u_s\|_{L^2(B_{\rho_{j'}}(0))}} \rightarrow \varphi
\end{equation*}
in the $C^1$-topology on compact subsets of $\mathbb{R}^n$.
\end{enumerate}
\end{lemma}

\begin{proof} 
See~\cite[Lemma 6.6]{SimWic16} (with $w = u_s$).  We modify their argument slightly, using $\|u\|_{C^{1,1/2}(B_1(0))} \leq \varepsilon$, \eqref{mss b est}  and \eqref{minimal sym schauder} in place of 6.2, 6.3, and 6.4 of~\cite{SimWic16} and using the doubling condition \eqref{freq_mono_doubling_hyp} in place of (11) of~\cite[Lemma 6.6]{SimWic16}. 
\end{proof}

To apply Lemma~\ref{freq_mono_lemma}, we need to prove the doubling condition \eqref{freq_mono_doubling_hyp}.  This can be accomplished using~\cite[Lemma 8.6]{SimWic16}.  

\begin{lemma} \label{doubling condition lemma} 
There exists $\varepsilon = \varepsilon(n,m) > 0$ such that if $u \in C^{1,1/2}(B_1(0),\mathcal{A}_2(\mathbb{R}^m))$ is a two-valued function such that $M = {\rm graph}\,u$ is stationary in $B_1(0) \times \mathbb{R}^m$, $u(0) = \{0,0\}$, $Du(0) = \{0,0\}$ and $[Du]_{1/2,B_1(0)} \leq \varepsilon$, then there exist a constant $\gamma \in (0,\infty)$ (depending on $M$ and $p$) such that 
\begin{equation} \label{doubling condition concl}
	\int_{B_{\rho}(0)} |u_s|^2 \leq \gamma \int_{B_{\rho/2}(0)} |u_s|^2 
\end{equation}
for all $\rho \in (0,1/4]$. 
\end{lemma}

\begin{proof}
We will use the results of~\cite[Chapter 8]{SimWic16}, to which we refer the reader for further details.   Without loss of generality assume that $u_s$ is not identically zero.  By taking differences in \eqref{mss1a} and \eqref{mss1b} like in~\cite[Chapters 5 and 8]{SimWic16}, we can show that $u_s$ is a solution to the weakly coupled differential system 
\begin{equation} \label{mss4} 
	D_i (A^{ij} \,D_j u_s^{\kappa}) + E^l_{\kappa\lambda} \,D_l u_s^{\lambda} = 0 
		\text{ in } B_{3/4}(0) \setminus \mathcal{B}_u \text{ for } \kappa = 1,\ldots,m 
\end{equation}
(as in 8.1 of~\cite{SimWic16} with $B_{3/4}(0)$ in place of $B_{1/2}(0)$), where $A^{ij}$ and $E^l_{\lambda}$ are real-valued functions such that 
\begin{equation} \label{mss4 A E est} 
	A^{ij}(0) = \delta_{ij}, \quad \op{Lip} A^{ij} \leq C \varepsilon, \quad |E^l_{\kappa\lambda}| \leq C \varepsilon 
\end{equation}
on $B_{3/4}(0) \setminus \mathcal{B}_u$ for some constant $C = C(n,m) \in (0,\infty)$.  Here $|E^l_{\kappa\lambda}| \leq C \varepsilon$ follows using Lemmas~7.1 and 7.4 and the construction of $E^l_{\kappa\lambda}$ in~\cite{SimWic16}.  Following an argument originally from~\cite{AKS}, we construct a bi-Lipschitz change of coordinates $y = \Gamma(x)$ and metric $\sum_{i,j=1}^n \widehat{g}_{ij}(y) \,dy_i \,dy_j$ (as in~\cite{SimWic16}) which by~\eqref{mss4 A E est} satisfy 

\begin{equation}
	\label{doubling Gamma hatg est} \sup_{X \in B_{1/2}(0)} (|\Gamma(X) - X| + |D\Gamma(X) - I|) \leq C \varepsilon , \quad
		\sup_{X \in B_{1/2}(0)} (|\widehat{g}_{ij} - \delta_{ij}| + |D_r \widehat{g}_{ij}|) \leq C \varepsilon , 
\end{equation}
where $r = |y|$ and $C = C(n,m) \in (0,\infty)$ is a constant.  Set $\widehat{v} = u_s \circ \Gamma : B_{1/2}(0) \rightarrow \mathcal{A}_2(\mathbb{R}^m)$.  We define the frequency function $\widehat{N}_{\widehat{v}}(\rho) = \mathcal{D}(\rho)/\mathcal{H}(\rho)$
for each $\rho \in (0,1/2]$ where 
\begin{equation*}
	\mathcal{D}(\rho) = \rho^{2-n} \int_{\partial B_{\rho}(0)} \widehat{v} \cdot D_r \widehat{v} \,\sqrt{\widehat{g}} \text{ and } 
	\mathcal{H}(\rho) = \rho^{1-n} \int_{\partial B_{\rho}(0)} |\widehat{v}|^2 \,\sqrt{\widehat{g}} 
\end{equation*}
for each $\rho \in (0,1/2]$, where $\widehat{g} = \det(\widehat{g}_{ij})$.  Arguing as in~\cite[Lemma 8.6]{SimWic16}, which itself uses~\cite{GL87}, and assuming $\varepsilon$ is sufficiently small, $\mathcal{H}(\rho) \neq 0$ for all $\rho \in (0,1/2]$ and $e^{C \varepsilon \rho} \widehat{N}_{\widehat{v}}(\rho)$ is non-decreasing as a function of $\rho \in (0,1/2]$.  Moreover,
\begin{equation} \label{doubling condition eqn1}
	\frac{\rho \,\mathcal{H}'(\rho)}{2 \,\mathcal{H}(\rho)} \leq (1 + C(n,m) \varepsilon \rho) \,\widehat{N}_{\widehat{v}}(1/2) 
		\leq 2 \,\widehat{N}_{\widehat{v}}(1/2)
\end{equation}
for all $\rho \in (0,1/2]$.  By integrating both sides of \eqref{doubling condition eqn1} over $[\sigma,\rho]$ we obtain 
\begin{equation*}
	\Big(\frac{\sigma}{\rho}\Big)^{4 \,\widehat{N}_{\widehat{v}}(1/2)} \mathcal{H}(\rho) \leq \mathcal{H}(\sigma)
\end{equation*}
for all $0 < \sigma < \rho \leq 1/2$, which by integrating again gives us 
\begin{equation*}
	\Big(\frac{\sigma}{\rho}\Big)^{4 \widehat{N}_{\widehat{v}}(1/2)} \rho^{-n} \int_{B_{\rho}(0)} |\widehat{v}|^2 \,\sqrt{\widehat{g}}
		\leq \sigma^{-n} \int_{B_{\sigma}(0)} |\widehat{v}|^2 \,\sqrt{\widehat{g}}
\end{equation*}
for all $0 < \sigma < \rho \leq 1/2$.  Since $\widehat{v} = u_s \circ \Gamma$ and $\Gamma$ and $(\widehat{g}_{ij})$ satisfy \eqref{doubling Gamma hatg est}, 
\begin{equation} \label{doubling condition eqn3}
	\Big(\frac{\sigma}{\rho}\Big)^{4 \widehat{N}_{\widehat{v}}(1/2)} \rho^{-n} \int_{B_{\rho/2}(0)} |u_s|^2 
		\leq (1 + C \varepsilon) \,\sigma^{-n} \int_{B_{2\sigma}(0)} |u_s|^2
\end{equation}
for all $0 < \sigma < \rho \leq 1/2$ and some constant $C = C(n,m) \in (0,\infty)$ provided $\varepsilon$ is sufficiently small.  In particular, by \eqref{doubling condition eqn3} with $\rho/4,2\rho$ in place of $\sigma,\rho$, we obtain \eqref{doubling condition concl} with $\gamma = 2^{3 n + 12 \widehat{N}_{\widehat{v}}(1/2) + 1}$. 
\end{proof}

\begin{proof}[Proof of Theorem~\ref{freq_mono_thm}]
Without loss of generality, assume that $P = 0$ and $T_P M = \mathbb{R}^n \times \{0\}$.
Apply Lemma~\ref{freq_mono_lemma}, after using Lemma~\ref{doubling condition lemma} to establish the doubling condition \eqref{freq_mono_doubling_hyp}.  
\end{proof}

Lemma~\ref{doubling condition lemma} only establishes the doubling condition at a single point $P \in \mathcal{K}_M$.  Often we will want a doubling condition with a uniform constant $\gamma$ for each point $P \in \mathcal{K}_M$ in an open neighborhood.  
For our purposes, we obtain such a doubling condition via the following lemma.

\begin{lemma} \label{doubling condition lemma3}
For every $\alpha \in (0,\infty)$ there exists an $\varepsilon = \varepsilon(n,m,\alpha) \in (0,1)$ such that if $\varphi \in C^{1,1/2}(\mathbb{R}^n,\mathcal{A}_2(\mathbb{R}^m))$ is a non-zero, symmetric, homogeneous, harmonic two-valued function with $\mathcal{N}_{\varphi}(0) \leq \alpha$ and if $u \in C^{1,1/2}(B_1(0),\mathcal{A}_2(\mathbb{R}^m))$ is a two-valued function such that $M = {\rm graph}\,u$ is stationary in $B_1(0) \times \mathbb{R}^m$, $\|u\|_{C^{1,1/2}(B_1(0))} \leq \varepsilon$ and 
\begin{equation} \label{doubling condition3 hyp}
	\int_{B_1(0)} \mathcal{G}\left( \frac{u_s}{\|u_s\|_{L^2(B_1(0))}}, \varphi \right)^2 < \varepsilon^2, 
\end{equation}
then 
\begin{equation} \label{doubling condition3 concl}
	\int_{\mathbf{B}_{\rho}(0,T_P M)} |\widetilde{u}_{P,s}|^2 \leq 2^{3 n + 12 \alpha + 13} \int_{\mathbf{B}_{\rho/2}(0,T_P M)} |\widetilde{u}_{P,s}|^2 
\end{equation}
for all $P \in \mathcal{K}_M \cap B_{1/2}(0) \times \mathbb{R}^m$ and $\rho \in (0,1/16]$, where we let $\widetilde{u}_P \in C^{1,1/2}(\mathbf{B}_{1/4}(0,T_P M), \mathcal{A}_2(T_P M^{\perp}))$ such that $M \cap \mathbf{C}_{1/4}(0,T_P M) = {\rm graph}\,\widetilde{u}_P$ (as in Definition~\ref{tildeu defn}) and $\widetilde{u}_{P,s}$ is the symmetric part of $\widetilde{u}_P$ (as in \eqref{avg and free defn}).  Consequently, there exists a constant $C = C_1(n,m,N_{\varphi,0}(1)) \,\|u\|_{C^{1,1/2}(B_1(0))}$ such that $e^{C \rho^{1/2}} N_{M,P}(\rho)$ is monotone non-decreasing in $\rho \in (0,1/32]$ for all $P \in \mathcal{K}_M \cap B_{1/2}(0) \times \mathbb{R}^m$. 
\end{lemma}

First let us prove the following simpler lemma, which is based on Lemma~\ref{doubling condition lemma}. 

\begin{lemma} \label{doubling condition lemma2}
For every $\alpha \in (0,\infty)$ there exists an $\varepsilon = \varepsilon(n,m,\alpha) \in (0,1)$ such that if $\varphi \in C^{1,1/2}(\overline{B_1(0)},\mathcal{A}_2(\mathbb{R}^m))$ is a non-zero, symmetric, harmonic two-valued function with $N_{\varphi,0}(1) \leq \alpha$ and if $u \in C^{1,1/2}(B_1(0),\mathcal{A}_2(\mathbb{R}^m))$ is a two-valued function such that $M = {\rm graph}\,u$ is stationary in $B_1(0) \times \mathbb{R}^m$, $\mathcal{K}_M \neq M$, and 
\begin{gather} 
	\label{doubling condition2 hyp1} u(0) = \{0,0\}, \quad Du(0) = \{0,0\}, \quad [Du]_{1/2,B_1(0)} \leq \varepsilon , \\ 
	\label{doubling condition2 hyp2} \int_{B_1(0)} \mathcal{G}\left( \frac{u_s}{\|u_s\|_{L^2(B_1(0))}}, \varphi \right)^2 < \varepsilon^2, 
\end{gather}
then 
\begin{equation} \label{doubling condition2 concl}
	\int_{B_{\rho}(0)} |u_s|^2 \leq 2^{3 n + 12 \alpha + 13} \int_{B_{\rho/2}(0)} |u_s|^2 
\end{equation}
for all $\rho \in (0,1/4]$. 
\end{lemma}

\begin{proof}
By Lemma~\ref{doubling condition lemma}, it suffices to show that $\widehat{N}_{\widehat{v}}(1/2) < \alpha + 1$, where $\widehat{v}$ and $\widehat{N}_{\widehat{v}}$ are as in the proof of Lemma~\ref{doubling condition lemma}.  Suppose to the contrary that for each $k = 1,2,3,\ldots$ there exists a non-zero, symmetric, harmonic two-valued function $\varphi_k \in C^{1,1/2}(B_1(0),\mathcal{A}_2(\mathbb{R}^m))$, a two-valued function $u_k \in C^{1,1/2}(B_1(0),\mathcal{A}_2(\mathbb{R}^m))$ and $P_k \in \mathcal{K}_{M_k}$ such that $M_k = {\rm graph}\,u_k$ is stationary in $B_1(0) \times \mathbb{R}^m$, $\mathcal{K}_{M_k} \neq M_k$, \eqref{doubling condition2 hyp1} and \eqref{doubling condition2 hyp2} hold true with $\varepsilon = 1/k$, $\varphi=\varphi_k$ and $u=u_k$ and 
\begin{equation*} 
	N_{\varphi_k,0}(1) \leq \alpha, \quad \widehat{N}_{\widehat{v}_k}(1/2) \geq \alpha + 1, 
\end{equation*}
where $u_{k,s}(X)$ is the symmetric part of $u_k$ (as in \eqref{avg and free defn}) and $\widehat{v}_k$ and $\widehat{N}_{\widehat{v}_k}(1/2)$ are as in the proof of Lemma~\ref{doubling condition lemma} with $u_k$ in place of $u$.  By \eqref{doubling condition2 hyp2}, $1/2 \leq \|\varphi_k\|_{L^2(B_1(0))} \leq 2$ for all sufficiently large $k$.  Thus by $N_{\varphi_k,0}(1) \leq \alpha$ and \cite[Lemma 3.1(b)]{KrumWic1}, after passing to a subsequence $\varphi_k \rightarrow \varphi$ in $C^1(B_{3/4}(0),\mathcal{A}_2(\mathbb{R}^m))$ for some non-zero, symmetric, harmonic two-valued function $\varphi \in C^{1,1/2}(B_{3/4}(0),\mathcal{A}_2(\mathbb{R}^m))$.  By \eqref{doubling condition2 hyp1}, \eqref{doubling condition2 hyp2}, and the estimates \eqref{minimal sym schauder}, $u_{k,s}/\|u_{k,s}\|_{L^2(B_1(0))} \rightarrow \varphi$ in $C^1(B_{3/4}(0),\mathcal{A}_2(\mathbb{R}^m))$.  This together with \eqref{doubling Gamma hatg est} and \eqref{doubling condition2 hyp1} implies that $\widehat{N}_{\widehat{v}_k}(1/2) \rightarrow N_{\varphi,0}(1/2)$.  In particular, by $N_{\varphi_k,0}(1) \leq \alpha$, $\widehat{N}_{\widehat{v}_k}(1/2) \geq \alpha + 1$ and the monotonicity of $N_{\varphi_k,0}$, 
\begin{equation*} 
	N_{\varphi,0}(1/2) = \lim_{k \rightarrow \infty} N_{\varphi_k,0}(1/2) \leq \limsup_{k \rightarrow \infty} N_{\varphi_k,0}(1) \leq \alpha < \alpha + 1 
		\leq \lim_{k \rightarrow \infty} \widehat{N}_{\widehat{v}_k}(1/2) = N_{\varphi,0}(1/2) 
\end{equation*}
which is impossible. 
\end{proof}

\begin{proof}[Proof of Lemma~\ref{doubling condition lemma3}] 
Fix $P \in \mathcal{K}_M \cap B_{1/2}(0) \times \mathbb{R}^m$.  Let $Q_P$ be a rotation of $\mathbb{R}^{n+m}$ satisfying \eqref{hatu rotation} and let $\widehat{u}_P \in C^{1,1/2}(B_{1/4}(\pi_0 P), \mathcal{A}_2(\mathbb{R}^m))$ satisfy \eqref{hatu graph} with $\rho = 1/4$.
By Definitions~\ref{tildeu defn} and \ref{hatu defn}, \eqref{doubling condition3 concl} is equivalent to 
\begin{equation} \label{doubling condition3 concl2}
	\int_{B_{\rho}(\pi_0 P)} |\widehat{u}_{P,s}|^2 \leq 2^{3 n + 12 \alpha + 13} \int_{B_{\rho/2}(\pi_0 P)} |\widehat{u}_{P,s}|^2 
\end{equation}
for all $\rho \in (0,1/16]$.  By the triangle inequality, Lemma~\ref{rotate lemma} (with $u, \widehat{u}_P$ in place of $u, \widehat{u}$) and \eqref{doubling condition3 hyp}, 
\begin{align} \label{doubling condition3 eqn1}
	&\int_{B_{1/4}(\pi_0 P)} \mathcal{G}\left( \frac{\widehat{u}_{P,s}}{\|u_s\|_{L^2(B_1(0))}}, \varphi \right)^2
		\\&\hspace{5mm} \leq 2 \int_{B_{1/4}(\pi_0 P)} \mathcal{G}\left( \frac{u_s}{\|u_s\|_{L^2(B_1(0))}}, \frac{\widehat{u}_{P,s}}{\|u_s\|_{L^2(B_1(0))}} 
			\right)^2 + 2 \int_{B_{1/4}(\pi_0 P)} \mathcal{G}\left( \frac{u_s}{\|u_s\|_{L^2(B_1(0))}}, \varphi \right)^2 \nonumber 
		\\&\hspace{5mm} \leq C(n,m) \,|Du(\pi_0 P)|^2 + 2 \int_{B_1(0)} \mathcal{G}\left( \frac{u_s}{\|u_s\|_{L^2(B_1(0))}}, \varphi \right)^2 
		\leq C(n,m) \,\varepsilon^2 . \nonumber
\end{align}
By \eqref{doubling condition3 hyp}, $| \|\varphi\|_{L^2(B_1(0))} - 1| \leq \varepsilon$.  Hence by \eqref{doubling condition3 eqn1} and the triangle inequality, 
\begin{equation*}
	\left| \frac{\|\widehat{u}_{P,s}\|_{L^2(B_{1/4}(\pi_0 P))}}{\|u_s\|_{L^2(B_1(0))}} - 1 \right| \leq C(n,m) \,\varepsilon
\end{equation*}
and thus by again applying \eqref{doubling condition3 eqn1} and the triangle inequality 
\begin{equation} \label{doubling condition3 eqn2}
	\int_{B_{1/4}(\pi_0 P)} \mathcal{G}\left( \frac{\widehat{u}_{P,s}}{\|\widehat{u}_{P,s}\|_{L^2(B_{1/4}(\pi_0 P))}}, \varphi \right)^2 \leq C(n,m) \,\varepsilon^2 .
\end{equation}
By applying Lemma~\ref{doubling condition lemma2} using \eqref{doubling condition3 eqn2} and 
\begin{equation*}
	N_{\varphi,\pi_0 P}(1/4) \leq \lim_{\rho \rightarrow \infty} N_{\varphi,\pi_0 P}(\rho) 
		= \lim_{\rho \rightarrow \infty} N_{\varphi,\pi_0 P/\rho}(1) = N_{\varphi,0}(1) =  \mathcal{N}_{\varphi}(0) \leq \alpha, 
\end{equation*}
we obtain \eqref{doubling condition3 concl2}. 
\end{proof}

A simple consequence of Lemma~\ref{doubling condition lemma3} is the following: 

\begin{theorem} \label{freq_mono_thm2}
Let $M$ be a $C^{1,1/2}$ two-valued stationary graph (as in Definition~\ref{twoval minimal graph}) and assume that $\mathcal{K}_M \neq M$.  For every $P_0 \in \mathcal{K}_M$ there exists $\rho_0 > 0$ (depending on $M$ and $P_0$) such that for every $P \in \mathcal{K}_M \cap B^{n+m}_{\rho_0}(P_0)$:
\begin{enumerate}
	\item[(i)] there exists $\widetilde{u}_P \in C^{1,1/2}(\mathbf{B}_{\rho_0}(0,T_P M), \mathcal{A}_2(T_P M^{\perp}))$ such that 
\begin{equation*}
	M \cap \mathbf{C}_{\rho_0}(0,T_P M) = {\rm graph}\, \widetilde{u}_P;
\end{equation*}
	\item[(ii)] for every $\rho \in (0,\rho_0]$, 
\begin{equation} \label{freq_mono_doubling}
	\int_{B_{\rho}(0)} |\widetilde{u}_{P,s}|^2 \leq 2^{3n + 12 \mathcal{N}_M(P_0) + 13} \int_{B_{\rho/2}(0)} |\widetilde{u}_{P,s}|^2 ,
\end{equation}
where $\widetilde{u}_{P,s}$ is the symmetric part of $\widetilde{u}_P$ (as in \eqref{avg and free defn});  
	\item[(iii)] $H_{M,P}(\rho) \neq 0$ for all $\rho \in (0,\rho_0/2]$ and $e^{C \rho^{1/2}} N_{M,0}(\rho)$ is monotone non-decreasing as a function of $\rho \in (0,\rho_0/2]$, where $C = C(n,m,\mathcal{N}_M(P_0)) \in (0,\infty)$ is a constant (depending on $M$ and $P_0$ and independent of $P$ and $\rho$).
\end{enumerate}
\end{theorem}

\begin{proof}
Translate and rotate so that $P_0 = 0$ and $T_{P_0} M = \mathbb{R}^n \times \{0\}$.  We can choose $\rho_0$ so that (i) holds true and moreover for a given $\varepsilon \in (0,1)$ we have $[D\widetilde{u}_{P,s}]_{1/2,\mathbf{B}_{\rho_0}(0,T_P M)} \leq \varepsilon$.  By applying Lemma~\ref{doubling condition lemma3} with $\varphi$ as a tangent function to $M$ at $P_0$, we can choose $\rho_0$ so that (ii) holds true.  (iii) follows by applying Lemma~\ref{freq_mono_lemma} using the doubling condition from (ii). 
\end{proof}

\section{Further standard consequences of the monotonicity of frequency functions} \label{sec:frequency sec2}

The monotonicity formula for frequency functions of two-valued stationary graphs have several standard consequences which play an important role in studying regularity.  The first is a decay rate for $H_{M,P}(\rho)$ and the $L^2$-norms of $\widetilde{u}_{P,s}$.  For simplicity and without loss of generality we assume that $P = 0$ and $T_P M = \mathbb{R}^n \times \{0\}$. 

\begin{lemma}
For each $\gamma \in (0,\infty)$ there exists $\varepsilon = \varepsilon(n,m,\gamma) > 0$ such that the following holds true.  Let $u \in C^{1,1/2}(B_1(0),\mathcal{A}_2(\mathbb{R}^m))$ be a two-valued function such that $M = {\rm graph}\,u$ is stationary in $B_1(0) \times \mathbb{R}^m$, $\mathcal{K}_M \neq M$, $u(0) = \{0,0\}$, $Du(0) = \{0,0\}$ and $[Du]_{1/2,B_1(0)} \leq \varepsilon$.  Suppose that there exists a constant $\gamma \in (0,\infty)$ such that 
\begin{equation*} 
	\int_{B_{\rho}(0)} |u_s|^2 \leq \gamma \int_{B_{\rho/2}(0)} |u_s|^2
\end{equation*}
for all $\rho \in (0,1]$.  Then there exists constants $C = C(n,m,\gamma) \in [1,\infty)$ such that 
\begin{align} \label{H_decay}
	&\left( \frac{\sigma}{\rho} \right)^{2 N_{M,0}(\rho) + C \varepsilon \rho^{1/2}} \rho^{1-n} \int_{\partial B_{\rho}(0)} |u_s|^2 
		\\&\hspace{20mm} \leq \sigma^{1-n} \int_{\partial B_{\sigma}(0)} |u_s|^2 \leq 
			e^{C \varepsilon \rho^{1/2}} \left( \frac{\sigma}{\rho} \right)^{2\mathcal{N}_M(0)} \rho^{1-n} \int_{\partial B_{\rho}(0)} |u_s|^2 \nonumber 
\end{align}
for each $0 < \sigma < \rho \leq 1/2$ and 
\begin{align} \label{L2_decay}
	&\left( \frac{\sigma}{\rho} \right)^{2N_{M,0}(\rho) + C \varepsilon \rho^{1/2}} \rho^{-n} \int_{B_{\rho}(0)} |u_s|^2 
		\\&\hspace{20mm} \leq \sigma^{-n} \int_{B_{\sigma}(0)} |u_s|^2 \leq 
			e^{C \varepsilon \rho^{1/2}} \left( \frac{\sigma}{\rho} \right)^{2\mathcal{N}_M(0)} \rho^{-n} \int_{B_{\rho}(0)} |u_s|^2. \nonumber
\end{align}
each $0 < \sigma < \rho \leq 1/2$. 
\end{lemma}
\begin{proof}
Arguing as in~\cite[Lemma 6.6]{SimWic16} (modifying the argument as in the proof of Lemma~\ref{freq_mono_lemma}), $1/C \leq N_{M,0}(\rho) \leq C$ and 
\begin{equation*}
	H'_{M,0}(\rho) = 2 \rho^{1-n} \int_{\partial B_{\rho}(0)} w^{\kappa} D_R w^{\kappa} = \frac{2}{\rho} D_{M,0}(\rho) + E \quad\text{where}\quad 
	|E| \leq C \varepsilon \rho^{-1/2} D_{M,0}(\rho) 
\end{equation*}
for all $0 < \rho < 1/2$ and some constant $C = C(n,m,\gamma) \in [1,\infty)$.  It follows that 
\begin{equation*} 
	(1 - C \varepsilon \tau^{1/2}) \,N_{M,0}(\tau) \leq \frac{\tau \,H'_{M,0}(\tau)}{2 \,H_{M,0}(\tau)} \leq (1 + C \varepsilon \tau^{1/2}) \,N_{M,0}(\tau)
\end{equation*} 
for all $0 < \tau < 1/2$ and some constant $C = C(n,m,\gamma) \in (0,\infty)$.  By applying the monotonicity formula for $N_{M,0}$ from Lemma~\ref{freq_mono_lemma} and $1/C \leq N_{M,0}(\tau) \leq C$ for all $\tau \in (0,1/2)$ and some constant $C = C(n,m,\gamma) \in [1,\infty)$, 
\begin{equation} \label{H decay eqn}
	\mathcal{N}_M(0) - C \varepsilon \tau^{1/2} \leq \frac{\tau H'_{M,0}(\tau)}{2 H_{M,0}(\tau)} \leq N_{M,0}(\rho) + C \varepsilon \rho^{1/2}
\end{equation} 
for all $0 < \tau \leq \rho < 1/2$ and some constant $C = C(n,m,\gamma) \in (0,\infty)$.  By integrating \eqref{H decay eqn} over $\tau \in [\sigma,\rho]$ we obtain \eqref{H_decay}.  Then integrating \eqref{H_decay} gives us \eqref{L2_decay}.
\end{proof}

Another important consequence is the upper semi-continuity of frequency.  We have two types of semi-continuity results: the first, Lemma~\ref{semicont freq lemma1}, concerns a sequence of two-valued stationary graphs $M_k$ converging to another two-valued stationary graph $M$, and the second, Lemma~\ref{semicont freq lemma2},  concerns a sequence two-valued stationary graphs $M_k$ which converge to a plane and blow-up to a two-valued harmonic function.  An important special case of Lemma~\ref{semicont freq lemma1} is when $M_k = M$, in which case we conclude that $\mathcal{N}_M : \mathcal{K}_M \rightarrow [3/2,\infty)$ is upper semi-continuous for every $C^{1,1/2}$ two-valued stationary graph $M$. 

\begin{lemma} \label{semicont freq lemma1}
Let $\Omega \subset \mathbb{R}^n$ be a bounded open set and $u_k,u \in C^{1,1/2}(\Omega,\mathcal{A}_2(\mathbb{R}^m))$ such that $M_k = {\rm graph}\,u_k$ and $M = {\rm graph}\,u$ are stationary in $\Omega \times \mathbb{R}^m.$ 
Suppose that $\mathcal{K}_M \neq M$, $u_k \rightarrow u$ in $C^1(\Omega,\mathcal{A}_2(\mathbb{R}^m))$ and that 
for some $\mu \in (0,1]$, 
\begin{equation} \label{semicont freq1 hyp}
	\sup_{k} [Du_k]_{\mu,\Omega} < \infty .
\end{equation}
Let $P_k \in \mathcal{K}_{M_k}$ and $P \in \mathcal{K}_M$ such that $P_k \rightarrow P$.  Then 
\begin{equation} \label{semicont freq1 concl}
	\mathcal{N}_M(P) \geq \limsup_{k \rightarrow \infty} \mathcal{N}_{M_k}(P_k).
\end{equation}
\end{lemma}

\begin{proof}[Proof of Lemma~\ref{semicont freq lemma1}]
Note that $\mathcal{K}_{M_k} \neq M_k$ for each sufficiently large $k,$ since $\mathcal{K}_M \neq M$ and $u_k \rightarrow u$ in $C^1(\Omega,\mathcal{A}_2(\mathbb{R}^m)).$ By translating and rotating slightly, 
we may assume without loss of generality that $P_k = P = 0$ and $T_{P_k} M_k = T_P M = \mathbb{R}^n \times \{0\}$. 
Let $\varepsilon$ be as in Lemma~\ref{freq_mono_lemma} with $\gamma = 2^{3 n + 12 \mathcal{N}_M(P) + 13}$.  Let $\varphi$ be a tangent function of $M$ at $0$.  By \eqref{semicont freq1 hyp}, the estimates \eqref{minimal avg schauder} and \eqref{minimal sym schauder} and the definition of tangent function, we can rescale so that $\Omega = B_1(0)$, \eqref{doubling condition2 hyp1} holds true with $u_k$ and $u$ in place of $u$, and \eqref{doubling condition2 hyp2} holds true for $u$.  Since $u_k \rightarrow u$ in $C^1(B_1(0),\mathcal{A}_2(\mathbb{R}^m))$, \eqref{doubling condition2 hyp2} holds true with $u_k$ in place of $u$. 
Hence by Lemma~\ref{doubling condition lemma2} and our choice of $\varepsilon$, \eqref{doubling condition2 concl} holds true for all $\rho \in (0,1/4]$ with $u_k$ in place of $u$. 
Thus we may apply Lemma~\ref{freq_mono_lemma} 
to conclude that there exists a constant $C_0 = C_0(n,m,\mathcal{N}_M(0)) \in (0,\infty)$ (independent of $k$) such that $e^{C_0 \varepsilon \rho^{1/2}} N_{M_k,0}(\rho)$ and $e^{C_0 \varepsilon \rho^{1/2}} N_{M,0}(\rho)$ are monotone nondecreasing in $\rho \in (0,1/8]$. 
\eqref{semicont freq1 concl} (with $P_k = P = 0$) follows as a standard consequence of the monotonicity of $e^{C_0 \varepsilon \rho^{1/2}} N_{M_k,0}(\rho)$ and $e^{C_0 \varepsilon \rho^{1/2}} N_{M,0}(\rho)$. 
\end{proof}

\begin{lemma} \label{semicont freq lemma2}
Let $u_k \in C^{1,1/2}(B_1(0),\mathcal{A}_2(\mathbb{R}^m))$ be two-valued functions whose graphs $M_k$ are stationary in $B_1(0) \times \mathbb{R}^m$ and assume that $\mathcal{K}_{M_k} \neq M_k$.  Let $\varphi \in C^{1,1/2}(B_1(0),\mathcal{A}_2(\mathbb{R}^m))$ be a non-zero, symmetric, harmonic two-valued function.  Suppose that for some $\mu \in (0,1]$, 
\begin{gather} 
	\label{semicont freq2 hyp1} \lim_{k \rightarrow \infty} \|u_k\|_{C^1(B_1(0))} = 0, \quad \sup_{k} [Du_k]_{\mu,B_1(0)} < \infty, \\
	\label{semicont freq2 hyp2} \lim_{k \rightarrow \infty} \int_{B_1(0)} \mathcal{G}\left( \frac{u_{k,s}}{\|u_{k,s}\|_{L^2(B_1(0))}}, \varphi \right)^2 = 0 ,  
\end{gather}
where $u_{k,s}$ is the symmetric part of $u_{k}$ (as in \eqref{avg and free defn}).  Let $P_k \in \mathcal{K}_{M_k}$ and $Y \in \mathcal{K}_{\varphi}$ such that $P_k \rightarrow (Y,0)$.  Then 
\begin{equation} \label{semicont freq2 concl}
	\mathcal{N}_{\varphi}(Y) \geq \limsup_{k \rightarrow \infty} \mathcal{N}_{M_k}(P_k). 
\end{equation}
\end{lemma}

\begin{proof}[Proof of Lemma~\ref{semicont freq lemma2}] 
Let $\varepsilon  = \varepsilon(n,m,\mathcal{N}_{\varphi}(0)) > 0$ to be later determined.  First we want to translate and rescale so that $Y = 0$, $P_k = 0$, $\sup_{k} [Du_k]_{1/2,B_1(0)} \leq \varepsilon$ and $N_{\varphi,0}(1) \leq \mathcal{N}_{\varphi}(0) + 1$.  This is straightforward except we need to make sure that \eqref{semicont freq2 hyp2} is preserved.  By \eqref{semicont freq2 hyp1} and the monotonicity of $N_{\varphi,Y}$, there exists $\rho \in (0,1 - |Y|)$ such that 
\begin{equation*}
	\sup_{k} \rho^{\mu} [Du_k]_{\mu,B_1(0)} \leq \varepsilon, \quad N_{\varphi,Y}(\rho) \leq \mathcal{N}_{\varphi}(0) + 1.
\end{equation*}
Notice that \eqref{semicont freq2 hyp2} and $P_k \rightarrow (Y,0)$ implies that 
\begin{equation} \label{semicont freq2 eqn1}
	\lim_{k \rightarrow \infty} \int_{B_{\rho}(0)} \mathcal{G}\left( \frac{u_{k,s}(\pi_0 P_k + X)}{\|u_{k,s}\|_{L^2(B_1(0))}}, \varphi(Y + X) \right)^2 = 0. 
\end{equation}
Thus by the triangle inequality and \eqref{semicont freq2 eqn1}, 
\begin{equation*}
	\lim_{k \rightarrow \infty} \left| \frac{\|u_{k,s}\|_{L^2(B_{\rho}(\pi_0 P_k))}}{\|u_{k,s}\|_{L^2(B_1(0))}} - \|\varphi\|_{L^2(B_{\rho}(Y))} \right| = 0 ,
\end{equation*}
which by applying the triangle inequality and \eqref{semicont freq2 eqn1} again gives us 
\begin{equation*}
	\lim_{k \rightarrow \infty} \int_{B_{\rho}(0)} \mathcal{G}\left( \frac{u_{k,s}(\pi_0 P_k + X)}{\|u_{k,s}\|_{L^2(B_{\rho}(\pi_0 P_k))}}, 
		\frac{\varphi(Y + X)}{\|\varphi\|_{L^2(B_{\rho}(Y))}} \right)^2 = 0. 
\end{equation*}
Thus by replacing $u_k$ and $\varphi$ with 
\begin{equation*}
	\frac{u_k(\pi_0 P_k + \rho X) - u_k(\pi_0 P_k)}{\rho}, \quad \frac{\varphi(Y + \rho X)}{\rho^{-n/2} \|\varphi\|_{L^2(B_{\rho}(Y))}}
\end{equation*}
(where $u_k(\pi_0 P_k + \rho X) - u_k(\pi_0 P_k)$ means we subtract the single value of $u_k(\pi_0 P_k)$ from both values of $u_k(\pi_0 P_k + \rho X)$), we may assume that $Y = 0$, $P_k = 0$, and $\sup_{k} [Du_k]_{\mu,B_1(0)} \leq \varepsilon$ and $N_{\varphi,0}(1) \leq \mathcal{N}_{\varphi}(0) + 1$.  By using \eqref{minimal avg schauder} and \eqref{minimal sym schauder} and rescaling as we did above, we can further assume that $\mu = 1/2$ so that $\sup_{k} [Du_k]_{1/2,B_1(0)} \leq \varepsilon$. 

Next we want to rotate $M_k$ slightly and rescale so that additionally $T_{P_k} M_k = \mathbb{R}^n \times \{0\}$, and in particular $u_k(0) = \{0,0\}$ and $Du_k(0) = \{0,0\}$.  Let $Q_k$ be a rotation of $\mathbb{R}^{n+m}$ such that $Q_k (T_{P_k} M_k) = \mathbb{R}^n \times \{0\}$ and $\|Q_k - I\| \rightarrow 0$.  Let $\widehat{u}_k \in C^{1,1/2}(B_{1/2}(0), \mathcal{A}_2(\mathbb{R}^m))$ such that ${\rm graph}\,\widehat{u}_k = Q_k (M_k)$.  Provided we assume that $\sup_{k} [Du_k]_{1/2,B_1(0)} \leq \varepsilon$ for $\varepsilon$ sufficiently small, by appyling the triangle inequality and Lemma~\ref{rotate lemma} (with $u_k, \widehat{u}_k$ in place of $u, \widehat{u}$) as we did above 
\begin{equation*}
	\lim_{k \rightarrow \infty} \int_{B_{1/2}(0)} \mathcal{G}\left( \frac{\widehat{u}_{k,s}}{\|\widehat{u}_{k,s}\|_{L^2(B_{1/2}(0))}}, \varphi \right)^2 = 0 , 
\end{equation*}
where $\widehat{u}_{k,s}$ is the symmetric part of $\widehat{u}_k$ (as in \eqref{avg and free defn}).

Thus we may replace $u_k$ with $2 \widehat{u}_k(X/2)$ in order to assume that $T_{P_k} M_k = \mathbb{R}^n \times \{0\}$ and thus $u_k(0) = \{0,0\}$ and $Du_k(0) = \{0,0\}$. 

By \eqref{semicont freq2 hyp2} and \eqref{minimal sym schauder}, $u_{k,s}/\|u_{k,s}\|_{L^2(B_1(0))} \rightarrow \varphi$ in the $C^1$-topology on compact subsets of $B_1(0)$.  Provided we assume that $\sup_{k} [Du_k]_{1/2,B_1(0)} \leq \varepsilon$ for $\varepsilon$ sufficiently small, we may apply Lemma~\ref{doubling condition lemma2} and Lemma~\ref{freq_mono_lemma} to show that there is a constant $C_0 = C_0(n,m,\mathcal{N}_{\varphi}(0)) \in (0,\infty)$ independent of $k$ such that $e^{C_0 \varepsilon \rho^{1/2}} N_{M_k,0}(\rho)$ is non-decreasing in $\rho \in (0,1/8]$.  Then using the monotonicity of $N_{M_k,P_k}$ and $N_{\varphi,Y}$ and the local $C^1$-convergence of $u_{k,s}/\|u_{k,s}\|_{L^2(B_1(0))} \rightarrow \varphi$ as we did in the proof of Lemma~\ref{semicont freq lemma2}, we obtain \eqref{semicont freq2 concl} (with $Y = 0$ and $P_k = 0$). 
\end{proof}

Now let $M$ be a two-valued stationary graph (as in Definition~\ref{twoval minimal graph}) and assume that $\mathcal{K}_M \neq M$.  Recalling the discussion from Subsection~\ref{sec:prelims harmonic subsec}, whenever $\varphi \in C^{1,1/2}(T_P M, T_P M^{\perp})$ is a tangent function of $M$ at $P \in \mathcal{K}_M$, the set 
\begin{equation*}
	S(\varphi) = \{ Z \in T_P M : \mathcal{N}_{\varphi}(Z) = \mathcal{N}_{\varphi}(0) \} 
\end{equation*}
is a linear subspace of $T_P M$ with the property that $\varphi(Z+X) = \varphi(X)$ for all $X \in T_P M$ and $Z \in S(\varphi)$.  For $j = 0,1,2,\ldots,n-2$, define 
\begin{equation*}
	\mathcal{K}^{(j)}_M = \{ P \in \mathcal{K}_M : \dim S(\varphi) \leq j \text{ for every tangent function $\varphi$ of $M$ at $P$} \} .
\end{equation*}

Observe that 
\begin{equation*}
	\mathcal{K}_M = \mathcal{K}^{(n-2)}_M \supseteq \mathcal{K}^{(n-3)}_M \supseteq \cdots \supseteq \mathcal{K}^{(1)}_M \supseteq \mathcal{K}^{(0)}_M. 
\end{equation*}

\begin{lemma} \label{stratification lemma} 
Let $M$ be a two-valued stationary graph (as in Definition~\ref{twoval minimal graph}) and assume that $\mathcal{K}_M \neq M$.  For each $j = 1,2,\ldots,n$, $\mathcal{K}^{(j)}_M$ has Hausdorff dimension at most $j$.  For $\alpha > 0$, $\{ P \in \mathcal{K}^{(0)}_M : \mathcal{N}_M(P) = \alpha \}$ is discrete.
\end{lemma}
\begin{proof}
The conclusion follows from a well-known argument, originally due to Almgren~\cite[Theorem 2.26]{Almgren} in the context of stationary varifolds, which uses Lemma~\ref{semicont freq lemma2} (semi-continuity of frequency) and the fact that $\mathcal{N}_{\varphi}(0) = \mathcal{N}_M(P)$ for any tangent function $\varphi$ of $M$ at $P$. 
\end{proof}

Thus in order to show that $\mathcal{K}_M$ is countably $(n-2)$-rectifiable, it suffices to show that $\mathcal{K}_M \setminus \mathcal{K}_M^{(n-3)}$ is countably $(n-2)$-rectifiable.  The points $P \in \mathcal{K}_M \setminus \mathcal{K}_M^{(n-3)}$ have the property that there is a tangent function $\varphi$ of $M$ at $P$ such that, for coordinates $X = (x_1,x_2,\ldots,x_n)$ with respect to some orthonormal basis for $T_P M$, $\varphi$ has the form 
\begin{equation*}
	\varphi(X) = \{ \pm \op{Re}((a+ib) \,(x_1+ix_2)^{\alpha}) \} 
\end{equation*}
on $T_P M$ for some constants $a,b \in T_P M^{\perp}$ and some $\alpha = k/2$ where $k \geq 3$ is an integer.  Recall from Subsection~\ref{sec:prelims harmonic subsec} that we call any such function $\varphi$ \emph{cylindrical}.

We will need the following lemma, which is an analogue to Lemma 2.4 of~\cite{Sim93} and will be used to prove Theorem~\ref{theorem2}. 

\begin{lemma} \label{lemma2_4} 
Let $3/2 \leq \alpha < K < \infty$ be given.  There are functions $\delta : (0,1) \rightarrow (0,1)$ and $R : (0,1) \rightarrow (4,\infty)$ depending only on $n$, $m$ and $K$ such that the following holds true.  Let $\varepsilon \in (0,1)$ and $u \in C^{1,1/2}(B_{R(\varepsilon)}(0),\mathbb{R}^m)$ be a two-valued function such that $M = {\rm graph}\,u$ is stationary in $B_{R(\varepsilon)}(0) \times \mathbb{R}^m$, $\mathcal{K}_M \neq M$, 
\begin{gather}
	\label{lemma2_4 tangent hyp1} u(0) = \{0,0\}, \quad Du(0) = \{0,0\}, \quad R(\varepsilon)^{1/2} [Du]_{1/2,B_{R(\varepsilon)}(0)} < \delta(\varepsilon), \\
	N_{M,0}(R(\varepsilon)/32) - \alpha < \delta(\varepsilon). \nonumber
\end{gather}
Assume that there exists a symmetric, homogeneous degree $\alpha$ harmonic two-valued function $\varphi^{(0)} \in C^{1,1/2}(\mathbb{R}^n, \mathcal{A}_2(\mathbb{R}^m))$ such that 
\begin{equation} \label{lemma2_4 tangent hyp2} 
	\int_{B_{R(\varepsilon)}(0)} \mathcal{G}\left( \frac{u_s}{R(\varepsilon)^{-n/2} \|u_s\|_{L^2(B_{R(\varepsilon)}(0)}}, \varphi^{(0)} \right)^2 
		< \delta(\varepsilon), 
\end{equation}
where $u_s$ is the symmetric part of $u$ (as in \eqref{avg and free defn}).  Let $P_1 \in \mathcal{K}_M$ such that $\mathcal{N}_M(P_1) \geq \alpha$.  

Then: 
\begin{enumerate}
	\item[(i)] $|N_{M,P_1}(\rho) - \alpha| \leq \varepsilon$ for all $0 < \rho \leq R(\varepsilon)/128$; 
	
	\item[(ii)] for every $\rho \in (0,1]$, there is a non-zero, symmetric, homogeneous, harmonic two-valued function $\varphi \in C^{1,1/2}(T_{P_1} M,$ $\mathcal{A}_2(T_{P_1} M^{\perp}))$ (depending on $\rho$) such that $|\mathcal{N}_{\varphi}(0) - \alpha| \leq \varepsilon^2$ and 
\begin{equation*}
	\int_{\mathbf{B}_1(0,T_{P_1} M)} \mathcal{G}(\widetilde{u}_{P_1,s,\rho},\varphi)^2 < \varepsilon^2, 
\end{equation*}
where $\widetilde{u}_{P_1,s,\rho}$ is as in \eqref{tan fn scaling};

	\item[(iii)] for every $\rho \in (0,1]$, either there is a non-zero, symmetric, homogeneous, harmonic two-valued function $\varphi \in C^{1,1/2}(T_{P_1} M, \mathcal{A}_2(T_{P_1} M^{\perp}))$ (depending on $\rho$) such that $|\mathcal{N}_{\varphi}(0) - \alpha| \leq \varepsilon^2$, $\dim S(\varphi) = n-2$, and 
\begin{equation*}
	\int_{\mathbf{B}_1(0,T_{P_1} M)} \mathcal{G}(\widetilde{u}_{P_1,s,\rho},\varphi)^2 < \varepsilon^2, 
\end{equation*}
or there is a $(n-3)$-dimensional subspace $L \subseteq T_{P_1} M$
\begin{equation*}
	\hspace{10mm}\pi_{P_1} \big\{ P \in (\mathcal{K}_M - P_1) \cap \overline{\mathbf{C}_{\rho}(0,T_{P_1} M)} : \mathcal{N}_M(P) \geq \alpha \big\} 
		\subset \big\{ X \in T_{P_1} M : \op{dist}(X,L) < \varepsilon \rho \big\} , 
\end{equation*}
where $\pi_{P_1}$ is the orthogonal projection map onto $T_{P_1} M$. 
\end{enumerate} 
\end{lemma}
\begin{proof} 
First observe that by \eqref{lemma2_4 tangent hyp1} and \eqref{lemma2_4 tangent hyp2}, we can apply Lemma~\ref{doubling condition lemma3} to conclude that, provided $\delta(\varepsilon)$ are sufficiently small, 
\begin{gather} 
	\label{lemma2_4_doubling} \int_{B_{\rho}(0)} |u_s|^2 \leq 2^{3n+12K+13} \int_{B_{\rho/2}(0)} |u_s|^2, \\
	\label{lemma2_4_doubling2} \int_{\mathbf{B}_{\rho}(0,T_{P} M)} |\widetilde{u}_{P_1,s}|^2 
		\leq 2^{3n+12K+13} \int_{\mathbf{B}_{\rho/2}(0,T_{P} M)} |\widetilde{u}_{P_1,s}|^2  
\end{gather}
for all $0 < \rho \leq R(\varepsilon)/16$, where we let $\widetilde{u}_{P_1} \in C^{1,1/2}(\mathbf{B}_{R(\varepsilon)/4}(0,T_{P_1} M), \mathcal{A}_2(T_{P_1} M^{\perp}))$ such that $M \cap \mathbf{C}_{R(\varepsilon)/4}(0,T_{P_1} M) = {\rm graph}\,\widetilde{u}_{P_1}$ (as in Definition~\ref{tildeu defn}) and $\widetilde{u}_{P_1,s}$ is the symmetric part of $\widetilde{u}_{P_1}$ (as in \eqref{avg and free defn}).  By \eqref{lemma2_4_doubling}, \eqref{lemma2_4_doubling2} and Lemma~\ref{freq_mono_lemma}, there exists a constant $C = C_1(n,m,K) \,\delta(\varepsilon) R(\varepsilon)^{-1/2}$ such that $e^{C \rho^{1/2}} N_{M,0}(\rho)$ and $e^{C \rho^{1/2}} N_{M,P_1}(\rho)$ are monotone nondecreasing in $\rho \in (0,R(\varepsilon)/32]$.  

Now to see (i), observe that by monotonicity of $N_{M,P_1}$ and $\mathcal{N}_M(P_1) \geq \alpha$, it suffices to show that $N_{M,P_1}(R(\varepsilon)/128) \leq \alpha + \varepsilon/2$.  We shall show this by contradiction.  Suppose that for some $K \in [3/2,\infty)$ and $\varepsilon \in (0,1)$ there exists $\delta_k \downarrow 0$, $R_k \rightarrow \infty$, $\alpha_k \in [3/2,K]$, $u_k \in C^{1,1/2}(B_{R_k}(0),\mathbb{R}^m)$ with $M_k = {\rm graph}\,u_k$ and $P_k \in \mathcal{K}_{M_k} \cap B_1(0) \times \mathbb{R}^m$ with $\mathcal{N}_{M_k}(P_k) \geq \alpha_k$ such that $M_k$ is stationary in $B_{R_k}(0) \times \mathbb{R}^m$, $\mathcal{K}_{M_k} \neq M_k$, 
\begin{gather}
	\label{lemma2_4_eqn1} u_k(0) = \{0,0\}, \quad Du_k(0) = \{0,0\}, \quad \|u_k\|_{C^{1,1/2}(B_{R_k}(0))} < \delta_k, \\
	\label{lemma2_4_eqn2}N_{M_k,0}(R_k/32) < \alpha_k + \delta_k, \\
	\label{lemma2_4_eqn3} N_{M_k,P_k}(R_k/128) \geq \alpha_k + \varepsilon^2, \\
	\label{lemma2_4_eqn4} \int_{B_{\rho}(0)} |u_{k,s}|^2 \leq 2^{3n+12 K+13} \int_{B_{\rho/2}(0)} |u_{k,s}|^2 \text{ for all } \rho \in (0,R_k/16], 
\end{gather} 
where $u_{k,s}$ is the symmetric part of $u_k$ (as in \eqref{avg and free defn}).  After passing to a subsequence, let $\alpha_k \rightarrow \alpha \in [3/2,K]$.  By \eqref{lemma2_4_eqn4} and the estimates \eqref{minimal sym schauder}, after passing to a subsequence is exists a non-zero symmetric two-valued function $\varphi \in C^{1,1/2}(B_{3/2}(0),$ $\mathcal{A}_2(\mathbb{R}^m))$ such that 
\begin{equation} \label{lemma2_4_eqn5}
	\frac{u_{k,s}(R_k X/32)}{(R_k/32)^{-n/2} \|u_{k,s}\|_{L^2(B_{R_k/32}(0))}} \rightarrow \varphi
\end{equation}
in $C^1(B_{3/2}(0),\mathcal{A}_2(\mathbb{R}^m))$.  By \eqref{mss3} with $u_k$ in place of $u$, \eqref{mss b est} and \eqref{lemma2_4_eqn1}, $\varphi$ is locally harmonic in $B_{3/2}(0) \setminus \mathcal{B}_{\varphi}$.  By \eqref{lemma2_4_eqn2} and the monotonicity of $N_{M_k,0}$, 
\begin{equation} \label{lemma2_4_eqn6}
	N_{\varphi,0}(1) = \lim_{k \rightarrow \infty} N_{M_k,0}(R_k/32) \leq \alpha. 
\end{equation}
By \eqref{lemma2_4_eqn1}, there exists a rotation $Q_k$ of $\mathbb{R}^{n+m}$ such that $Q_k (T_{P_k} M_k) = \mathbb{R}^n \times \{0\}$ and $\|Q_k - I\| \rightarrow 0$.  Let $\widehat{u}_k \in C^{1,1/2}(B_{R_k/4}(\pi_0 P_k), \mathcal{A}_2(\mathbb{R}^m))$ such that 
\begin{equation*}
	{\rm graph}\,\widehat{u}_k = P_k + (Q_k (M_k - P_k)) \cap B_{R_k/4}(\pi_0 P_k) \times \mathbb{R}^m
\end{equation*}
(as in Definition~\ref{hatu defn}) and let $\widehat{u}_k$ denote the symmetric part of $\widehat{u}_{k,s}$ (as in \eqref{avg and free defn}).  By using Lemma~\ref{rotate lemma} (with $u_k, \widehat{u}_k$ in place of $u, \widehat{u}$) and \eqref{lemma2_4_eqn5} and noting that $P_k/R_k \rightarrow 0$, 
\begin{equation} \label{lemma2_4_eqn7}
	\frac{\widehat{u}_{k,s}(\pi_0 P_k+R_k X/128)}{(R_k/128)^{-n/2} \|\widehat{u}_{k,s}\|_{L^2(B_{R_k/128}(0))}} \rightarrow \varphi
\end{equation}
in $C^0(B_{5/4}(0),\mathcal{A}_2(\mathbb{R}^m))$.  Thus by the estimates \eqref{minimal sym schauder}, we have the convergence in \eqref{lemma2_4_eqn7} in $C^1(B_{9/8}(0), \mathcal{A}_2(\mathbb{R}^m))$.  Hence it follows from \eqref{lemma2_4_eqn3} that 
\begin{equation*}
	N_{\varphi,0}(1) = \lim_{k \rightarrow \infty} N_{M_k,P_k}(R_k/128) \geq \alpha + \varepsilon^2, 
\end{equation*}
which contradicts \eqref{lemma2_4_eqn6}.  

To see (ii) and (iii), in light of the discussion above we may assume that $\widetilde{u}_{M,P_1}$ satisfies the doubling condition \eqref{lemma2_4_doubling2} and that $|N_{M,P_1}(1) - \alpha|$ is as small as needed.  Also, we may translate and rotate taking $P_1$ and $T_{P_1} M$ to $0$ and $\mathbb{R}^n \times \{0\}$.  Then (ii) and (iii) are easy consequences of the estimates \eqref{minimal sym schauder}, the monotonicity of frequency functions, and upper semi-continuity of frequency as in Lemma~\ref{semicont freq lemma2}.  In particular, the proof of (iii) is similar to that of~\cite[Lemma 2.4]{Sim93}. 
\end{proof}

\section{Statement of main results} \label{sec:results sec}

\begin{definition} \label{varphi0 defn}
Here and subsequently, we shall fix a number $\alpha$ such that $\alpha = k_0/2$ for some integer $k_0 \geq 3$ and a two-valued function $\varphi^{(0)} \in C^{1,1/2}(\mathbb{R}^n,\mathcal{A}_2(\mathbb{R}^m))$ defined by 
\begin{equation*} 
	\varphi^{(0)}(x_1,x_2,\ldots,x_n) = \{ \pm \op{Re}(c^{(0)} (x_1+ix_2)^{\alpha}) \}
\end{equation*}
where $c^{(0)} \in \mathbb{C}^m \setminus \{0\}$. 
\end{definition}

Notice that $\varphi^{(0)}$ is homogeneous degree $\alpha$ and locally harmonic in $\mathbb{R}^n \setminus \{0\} \times \mathbb{R}^{n-2}$.  Moreover, ${\rm graph}\, \varphi^{(0)} \cap (\mathbb{R}^n \setminus \{0\} \times \mathbb{R}^{n-2}) \times \mathbb{R}^m$ is a smoothly immersed submanifold of $(\mathbb{R}^n \setminus \{0\} \times \mathbb{R}^{n-2}) \times \mathbb{R}^m$. 

\begin{definition} \label{F defn}
Let $\alpha$ and $\varphi^{(0)}$ be as in Definition~\ref{varphi0 defn}.  For each $\varepsilon \in (0,1)$, we define $\mathcal{F}_{\varepsilon}(\varphi^{(0)})$ to be the set of all pairs $(u,\Lambda)$ consisting of a two-valued function $u \in C^{1,1/2}(B_1(0),\mathcal{A}_2(\mathbb{R}^m))$ and a positive number $\Lambda \in (0,\infty)$ such that: 
\begin{enumerate}
	\item[(i)] $M = {\rm graph}\, u$ is stationary in $B_1(0) \times \mathbb{R}^m$ (as in Definition~\ref{twoval minimal graph});
	\item[(ii)] $\|u\|_{C^{1,1/2}(B_1(0))} \leq \varepsilon$; 
	\item[(iii)] the symmetric part $u_s$ of $u$ (as in \eqref{avg and free defn}) satisfies 
	\begin{equation} \label{F tangent req}
		\int_{B_1(0)} \mathcal{G}\left( \frac{u_s}{\Lambda}, \varphi^{(0)} \right)^2 \leq \varepsilon^2. 
	\end{equation}
\end{enumerate}
\end{definition}

\begin{remark} \label{F defn rmk} {\rm 
By the triangle inequality, \eqref{F tangent req} implies that 
\begin{equation} \label{F defn rmk eqn1}
	\left| \frac{\|u_s\|_{L^2(B_1(0))}}{\Lambda} - \|\varphi^{(0)}\|_{L^2(B_1(0))} \right|^2 
		\leq \int_{B_1(0)} \mathcal{G}\left( \frac{u_s}{\Lambda}, \varphi^{(0)} \right)^2 \leq \varepsilon^2 
\end{equation}
and thus 
\begin{equation} \label{F defn rmk eqn2}
	\int_{B_1(0)} \mathcal{G}\left( \frac{u_s}{\|u_s\|_{L^2(B_1(0))}}, \frac{\varphi^{(0)}}{\|\varphi^{(0)}\|_{L^2(B_1(0))}} \right)^2 
		\leq C \varepsilon^2 
		\end{equation}
for some constant $C = C(\varphi^{(0)}) \in (0,\infty)$.  As a consequence of \eqref{F defn rmk eqn2} and Lemma~\ref{doubling condition lemma2}, if $(u,\Lambda) \in \mathcal{F}_{\varepsilon}(\varphi^{(0)})$ for $\varepsilon = \varepsilon(n,m,\varphi^{(0)})$ sufficiently small and $u(0) = \{0,0\}$ and $Du(0) = \{0,0\}$, then we have the doubling condition 
\begin{equation*}
	\int_{B_{\rho}(0)} |\widetilde{u}_{P,s}|^2 \leq 2^{3n + 12 \alpha + 13} \int_{B_{\rho/2}(0)} |\widetilde{u}_{P,s}|^2 
\end{equation*}
for each $P \in \mathcal{K}_M \cap B_{1/2}(0) \times \mathbb{R}^m$ and $\rho \in (0,1/16]$, where we let $\widetilde{u}_P \in C^{1,1/2}(\mathbf{B}_{1/4}(0, T_P M), \mathcal{A}_2(T_P M^{\perp}))$ such that $M \cap \mathbf{C}_{1/4}(0, T_P M) = {\rm graph}\,\widetilde{u}_P$ (as in Definition~\ref{tildeu defn}) and $\widetilde{u}_{P,s}$ is the symmetric part of $\widetilde{u}_P$ (as in \eqref{avg and free defn}).  In light of \eqref{F defn rmk eqn1} and \eqref{F defn rmk eqn2}, one might try to replace $\Lambda$ with $\|u_s\|_{L^2(B_1(0))}$.  However, as we will discuss in further detail in Remark~\ref{F defn rmk2}, \eqref{F defn rmk eqn1} and \eqref{F defn rmk eqn2} do not give us the necessary control to obtain the precise estimates needed to prove and apply Lemma~\ref{lemma1} below.
} \end{remark}

\begin{definition} \label{Phi defn}
Let $\alpha$ and $\varphi^{(0)}$ be as in Definition~\ref{varphi0 defn}.  For each $\varepsilon \in (0,1)$, $\Phi_{\varepsilon}(\varphi^{(0)})$ is the set of all symmetric two-valued functions $\varphi \in C^{1,1/2}(\mathbb{R}^n,\mathcal{A}_2(\mathbb{R}^m))$ such that 
\begin{equation*}
	\varphi(x_1,x_2,\ldots,x_n) = \{ \pm \op{Re}(c (x_1+ix_2)^{\alpha}) \}
\end{equation*}
for some $c \in \mathbb{C}^m$ and 
\begin{equation*}
	\int_{B_1(0)} \mathcal{G}(\varphi(X),\varphi^{(0)}(X))^2 dX \leq \varepsilon^2. 
\end{equation*}
\end{definition}

\begin{definition} \label{tilde Phi defn}
Let $\alpha$ and $\varphi^{(0)}$ be as in Definition~\ref{varphi0 defn}.  For each $\varepsilon \in (0,1)$, $\widetilde{\Phi}_{\varepsilon}(\varphi^{(0)})$ is the set of all symmetric two-valued functions $\varphi(e^A X)$, where $\varphi \in \Phi_{\varepsilon}(\varphi^{(0)})$ and $A = (A_{ij})$ is an $n \times n$ skew-symmetric matrix with $A_{ij} = 0$ if $1 \leq i,j \leq 2$, $A_{ij} = 0$ if $3 \leq i,j \leq n$, and $\|A\| \leq \varepsilon$. 
\end{definition}

\begin{remark} {\rm $\varphi^{(0)} \in \Phi_{\varepsilon}(\varphi^{(0)})$ and $\Phi_{\varepsilon}(\varphi^{(0)}) \subseteq \widetilde{\Phi}_{\varepsilon}(\varphi^{(0)})$.  Each $\varphi \in \widetilde{\Phi}_{\varepsilon}(\varphi^{(0)})$ is non-zero, homogeneous degree $\alpha$ and locally harmonic in $\mathbb{R}^n \setminus \mathcal{B}_{\varphi}$. 
} \end{remark}

\begin{remark} \label{graph_rmk} {\rm
Given $\varphi \in \Phi_{\varepsilon}(\varphi^{(0)})$, we can regard the graph of $\varphi$ as an immersed $C^{\infty}$ submanifold in $\mathbb{R}^n \setminus \{0\} \times \mathbb{R}^{n-2}$.  Thus given $\Omega \subseteq \mathbb{R}^n \setminus \{0\} \times \mathbb{R}^{n-2}$, we can consider single-valued functions $v : \op{graph} \varphi |_{\Omega} \rightarrow \mathbb{R}^m$.  In the case that $\alpha$ is a positive integer, the graph of $\varphi$ is the union of two embedded submanifolds, the graphs of $+\op{Re}(c (x_1+ix_2)^{\alpha})$ and $-\op{Re}(c (x_1+ix_2)^{\alpha})$.  Hence any function $v : \op{graph} \varphi |_{\Omega} \rightarrow \mathbb{R}^m$ consists to a function on the graph of $+\op{Re}(c (x_1+ix_2)^{\alpha})$ and another function on the graph of $-\op{Re}(c (x_1+ix_2)^{\alpha})$.  In the case that $\alpha = k_0/2$ for some odd integer $k_0 \geq 3$, the graph of $\varphi$ consists of a single branched submanifold parameterized by $(re^{i\theta},y,\op{Re}(c r^{\alpha} e^{i\theta}))$ for $r \geq 0$, $\theta \in [0,4\pi]$, and $y \in \mathbb{R}^{n-2}$.
} \end{remark}

The following is our main lemma.

\begin{lemma} \label{lemma1} 
Let $\alpha$ and $\varphi^{(0)}$ be as in Definition~\ref{varphi0 defn}.  For every $\vartheta \in (0,1/8)$ there exists $\delta_0, \varepsilon_0 \in (0,1)$ depending only on $n$, $m$, $\alpha$, $\varphi^{(0)}$, and $\vartheta$ such that if $(u,\Lambda) \in \mathcal{F}_{\varepsilon_0}(\varphi^{(0)})$ with $M = {\rm graph}\, u$ and $\varphi \in \widetilde{\Phi}_{\varepsilon_0}(\varphi^{(0)})$ such that $u(0) = \{0,0\}$, $Du(0) = \{0,0\}$ and $\mathcal{N}_M(0) \geq \alpha$, then either 
\begin{enumerate}
\item[(i)] $(B_{\delta_0}(0,y_0) \times \mathbb{R}^m) \cap \{ P \in \mathcal{K}_M \cap B_{1/2}(0) \times \mathbb{R}^m : \mathcal{N}_M(P) \geq \alpha \} = \emptyset$ for some $y_0 \in B^{n-2}_{1/2}(0)$ or 

\item[(ii)] there is a $\widetilde{\varphi} \in \widetilde{\Phi}_{\gamma \varepsilon_0}(\varphi^{(0)})$ such that 
\begin{equation*}
	\vartheta^{-n-2\alpha} \int_{B_{\vartheta}(0)} \mathcal{G}\left( \frac{u_s}{\Lambda}, \varphi \right)^2 
	\leq C\vartheta^{2\mu} \left( \int_{B_1(0)} \mathcal{G}\left( \frac{u_s}{\Lambda}, \varphi \right)^2 + \|Du\|_{C^0(B_1(0))} \right) , 
\end{equation*}
where $\gamma = \gamma(n,m,\alpha,\varphi^{(0)},\vartheta) \in [1,\infty)$, $\mu = \mu(n,m,\alpha,\varphi^{(0)}) \in (0,1)$ and $C = C(n,m,\alpha,\varphi^{(0)}) \in (0,\infty)$ are constants (and in particular $\mu$ and $C$ are independent of $\vartheta$). 
\end{enumerate}
\end{lemma}

By iteratively applying Lemma~\ref{lemma1} in a manner completely analogous to the corresponding argument in~\cite{Sim93}, we obtain the following result, from which Theorems A and B of the introduction will readily follow (see Section~\ref{sec:prove main result sec}).

\begin{theorem} \label{theorem2} 
Let $\Omega$ be an open subset of $\mathbb{R}^n$ and $M \subset \Omega \times \mathbb{R}^m$ be a two-valued stationary graph as in Definition~\ref{twoval minimal graph}.  Assume that $\mathcal{K}_M \neq M$.  Then $\mathcal{K}_M$ is countably $(n-2)$-rectifiable.  Moreover, if for each $\alpha \geq 3/2$ we let 
\begin{equation*}
	\mathcal{K}_{M,\alpha} = \{ P \in \mathcal{K}_M : \mathcal{N}_M(P) = \alpha \text{ and $M$ has a cylindrical tangent function at $P$} \},
\end{equation*}
then: 
\begin{enumerate}
	\item[(a)] For any compact set $K \subset \Omega \times \mathbb{R}^m$, $K \cap \mathcal{K}_{M,\alpha} \neq \emptyset$ for only finitely many $\alpha \geq 3/2$. 
	
	\item[(b)] For every $\alpha \geq 3/2$ and $\mathcal{H}^{n-2}$-a.e.~$P \in \mathcal{K}_{M,\alpha}$ there is a unique, non-zero, symmetric, cylindrical, homogeneous degree $\alpha$, harmonic two-valued function $\varphi^{(P)} \in C^{1,1/2}(T_P M, \mathcal{A}_2(T_P M^{\perp}))$, a number $\rho_P > 0$ and two-valued function $\widetilde{u}_P \in C^{1,1/2}(\mathbf{B}_{\rho_P}(0,T_P M), \mathcal{A}_2(T_P M^{\perp}))$ such that 
\begin{equation} \label{theorem2 conclb1}
	M \cap \mathbf{C}_{\rho_P}(0, T_P M) = {\rm graph}\,\widetilde{u}_P
\end{equation}
(as in Definition~\ref{tildeu defn}) and 
\begin{equation} \label{theorem2 conclb2}	
	\rho^{-n} \int_{\mathbf{B}_{\rho}(0, T_P M)} \mathcal{G}(\widetilde{u}_{P,s}, \varphi^{(p)})^2 \leq C_P \,\rho^{2\alpha + 2\mu_P}
\end{equation}
for all $\rho \in (0,\rho_P]$, where $\widetilde{u}_{P,s}$ is the symmetric part of $\widetilde{u}_P$ (as in \eqref{avg and free defn}) and $\mu_P \in (0,1)$ and $C_P \in (0,\infty)$ are constants depending only on $n$, $m$, $\alpha$, $M$ and $P$. 

	\item[(c)] For every $\alpha \geq 3/2$ there is an open set $V_{\alpha} \supset \mathcal{K}_{M,\alpha}$ such that $V_{\alpha} \cap \{ P \in \mathcal{K}_M : \mathcal{N}_M(P) \geq \alpha \}$ has locally finite $\mathcal{H}^{n-2}$-measure; i.e.~for each $P_1 \in V_{\alpha} \cap \{ P \in \mathcal{K}_M : \mathcal{N}_M(P) \geq \alpha \}$, there is $\rho > 0$ such that $\mathcal{H}^{n-2}(B^{n+m}_{\rho}(P_1) \cap \{ P \in \mathcal{K}_M : \mathcal{N}_M(P) \geq \alpha \}) < \infty$.  In particular, for each $\alpha \geq 3/2$, $\mathcal{K}_{M,\alpha}$ has locally finite $\mathcal{H}^{n-2}$-measure. 
\end{enumerate}
\end{theorem}

\section{Graphical representation relative to homogeneous harmonic functions} \label{sec:graphicalrep sec}

Let $\varphi^{(0)}$ be as in Definition~\ref{varphi0 defn}.  Suppose that $\varphi \in \Phi_{\varepsilon_0}(\varphi^{(0)})$ and that $(u,\Lambda) \in \mathcal{F}_{\varepsilon_0}(\varphi^{(0)})$ so that in particular the graph of $u$ is stationary and $w = u_s/\Lambda$ is $L^2$-close to $\varphi^{(0)}$.  We will show, in Lemma~\ref{lemma2_6} below, that outside a small tubular neighborhood of $\{0\} \times \mathbb{R}^{n-2}$, the graph of $w$ is the graph of a single-valued function $v$ over ${\rm graph}\,\varphi$ satisfying certain coarse estimates, giving a way to pair the values of $w(X)$ and $\varphi(X)$ for $X$ outside of that tubular neighborhood of $\{0\} \times \mathbb{R}^{n-2}$.  More precisely, 
we will construct a single-valued function $v \in C^2({\rm graph}\, \varphi |_U, \mathbb{R}^m)$ for some open set $U \subseteq B_1(0) \setminus \{0\} \times \mathbb{R}^{n-2}$ such that 
\begin{equation} \label{v graph rep}
	w(X) = u_s(X)/\Lambda = \{ \varphi_1(X) + v(X,\varphi_1(X)), -\varphi_1(X) + v(X,-\varphi_1(X)) \}
\end{equation}
for each $X \in U$, where we write $\varphi(X) = \{\pm \varphi_1(X)\}$.  Note that we regard ${\rm graph}\, \varphi |_U$ as an immersed submanifold of $U \times \mathbb{R}^m$ and thus the function $v$ is well-defined, see Remark~\ref{graph_rmk}.  We associate $v$ with a symmetric two-valued function 
\begin{equation} \label{hatv defn}
	\widehat{v}(X) = \{ v(X,\varphi_1(X)), v(X,-\varphi_1(X)) \} 
\end{equation}
for each $X \in U$.

To prove Lemma~\ref{lemma2_6}, we need the following elementary result. 

\begin{lemma} \label{separation_lemma} 
Let $\gamma \in (0,1)$.  Suppose $\psi \in C^{1,1/2}(B_1(0),\mathcal{A}_2(\mathbb{R}^m))$ is a symmetric two-valued harmonic function such that whenever $X \in \mathcal{Z}_{\psi}$ we have that $D\psi(X) = \{ \pm A \}$ for some rank one $m \times n$ matrix $A$.  There exists $\varepsilon = \varepsilon(n,m,\gamma,\psi) > 0$ such that the following holds true.  Let $u \in C^{1,1/2}(B_1(0),\mathcal{A}_2(\mathbb{R}^m))$ be a two-valued function whose graph is stationary in $B_1(0) \times \mathbb{R}^m$ and $\Lambda \in (0,\infty)$ be a positive number such that 
\begin{equation} \label{separation_hyp1} 
	\|Du\|_{C^{0,1/2}(B_1(0))} < \varepsilon, \quad \int_{B_1(0)} \mathcal{G}\left(\frac{u_s}{\Lambda},\psi\right)^2 < \varepsilon^2. 
\end{equation}
Let $v \in C^{1,1/2}(B_1(0),\mathcal{A}_2(\mathbb{R}^m))$ be a symmetric harmonic two-valued function such that 
\begin{equation} \label{separation_hyp2} 
	\int_{B_1(0)} \mathcal{G}(v,\psi)^2 < \varepsilon^2. 
\end{equation}
Then $u = \{u_1,u_2\}$ and $v = \{v_1,-v_1\}$ on $B_{\gamma}(0)$ for some single-valued functions $u_1,u_2,v_1 \in C^{\infty}(B_{\gamma}(0),\mathbb{R}^m)$ such that $u_1$ and $u_2$ solve the minimal surface system in $B_{\gamma}(0)$, $v_1$ is harmonic in $B_{\gamma}(0)$, and 
\begin{equation} \label{separation_concl} 
	\left\| \frac{u_1-u_2}{2 \Lambda} - v_1 \right\|_{C^{1,1/2}(B_{\gamma}(0))} 
	\leq C \left( \int_{B_1(0)} \mathcal{G}\left(\frac{u_s}{\Lambda},v\right)^2 + \|f\|_{L^{\infty}(B_{(1+\gamma)/2}(0))}^2 \right)^{1/2}  
\end{equation}
for some $C = C(n,m,\gamma) \in (0,\infty)$, where $f$ is as in \eqref{mss3} with $w = u_s/\Lambda$. 
\end{lemma}
\begin{proof} 
In the case $\mathcal{Z}_{\psi} = \emptyset$, by using \eqref{separation_hyp1}, \eqref{mss3}, standard elliptic estimates and the estimates \eqref{harmonic schauder} and \eqref{minimal sym schauder}, $u_s/\Lambda$ and $v$ are uniformly close to $\psi$ in $B_{(1+\gamma)/2}(0)$ and the conclusion of Lemma~7.1 readily follows.  
We will consider the special case of $\psi(X) = \{ \pm a x_1 e_1 \}$ on $B_1(0)$ where $a \in \mathbb{R}$, $a > 0$, and $e_1,e_2,\ldots,e_m$ denotes the standard basis for $\mathbb{R}^m$.  Then Lemma~\ref{separation_lemma} will hold true for general $\psi$ by approximation of $\psi$ by affine two-valued functions and a standard covering argument.  

Suppose that $\psi(X) = \{ \pm a x_1 e_1 \}$ on $B_1(0)$ for some $a > 0$ and set $\psi_1(X) = a x_1 e_1$.  Let $\delta = \delta(\gamma,\psi) \in (0,a/4)$ be a small.  Provided $\varepsilon > 0$ is sufficiently small (depending on $n,m,\gamma,\psi,\delta$), using \eqref{separation_hyp1}, \eqref{mss3}, standard elliptic estimates and \eqref{minimal sym schauder} we can show that
$u = \{u_1,u_2\}$ on $B_{(1+\gamma)/2}(0)$ for some single-valued solutions $u_1,u_2 \in C^{\infty}(B_{(7+\gamma)/8}(0),\mathbb{R}^m)$ to the minimal surface system such that 
\begin{equation} \label{separation_eqn2}
	\left\| \frac{u_1-u_2}{2 \Lambda} - \psi_1 \right\|_{C^3(B_{(1+\gamma)/2}(0))} < \delta . 
\end{equation}
Moreover, by \eqref{separation_hyp2}, standard elliptic estimates and the estimates \eqref{harmonic schauder}, 
$v = \{-v_1,+v_1\}$ on $B_{(3+5\gamma)/8}(0)$ for some single-valued harmonic function $v_1 \in C^{\infty}(B_{(3+5\gamma)/8}(0), \mathbb{R}^m)$ such that 
\begin{equation} \label{separation_eqn4} 
	\left\| v_1 - \psi_1 \right\|_{C^3(B_{(3+5\gamma)/8}(0))} < \delta . 
\end{equation}

Let $X = (x_1,x_2,\ldots,x_n) = (x_1,z)$ where $z = (x_2,x_3,\ldots,x_n)$.  Fix $z \in B^{n-1}_{(3+5\gamma)/8}(0)$.  We know that $|w_1(x_1,z) - v_1(x_1,z)| = |w_1(x_1,z) + v_1(x_1,z)|$ precisely when $v_1(x_1,z) \cdot w_1(x_1,z) = 0$.  Since by \eqref{separation_eqn2} and \eqref{separation_eqn4} we know that $w_1$ and $v_1$ are close to $a x_1 e_1$ in $C^2(B_{(3+\gamma)/4}(0))$ and $a_1^2 x_1^2$ is a strictly convex function of $x_1$, $v_1(x_1,z) \cdot w_1(x_1,z)$ is also a strictly convex function of $x_1$.  Hence $v_1(x_1,z) \cdot w_1(x_1,z) = 0$ for at most two $x_1$ with $(x_1,z) \in B_{(3+\gamma)/4}(0)$.  If $v_1(x_1,z) \cdot w_1(x_1,z) = 0$ for either zero or one $x_1$ with $(x_1,z) \in B_{(3+\gamma)/4}(0)$, then 
\begin{equation} \label{separation_eqn8} 
	\mathcal{G}(w(x_1,z),v(x_1,z)) = \sqrt{2} \,|w_1(x_1,z) - v_1(x_1,z)| \text{ for all } x_1 \text{ with } (x_1,z) \in B_{(3+\gamma)/4}(0). 
\end{equation}
Suppose $v_1(x_1,z) \cdot w_1(x_1,z) = 0$ for $x_1 = \zeta,\xi$ with $(\zeta,z), (\xi,z) \in B_{(3+\gamma)/4}(0)$ and 
$\xi < \zeta$.  Notice that using \eqref{separation_eqn2} and \eqref{separation_eqn4}, 
\begin{equation*}
	0 = | v_1(\xi,z) \cdot w_1(\xi,z) - v_1(\zeta,z) \cdot w_1(\zeta,z) | \geq a^2 |\xi^2 - \zeta^2| - 2 \delta a (|\xi|+|\zeta|) - 2 \delta^2 .  
\end{equation*}
It follows that $\zeta - \xi \leq 4\delta/a$, since either $|\xi| + |\zeta| \leq 2\delta/a$ or $|\xi| + |\zeta| > 2\delta/a$ and thus 
\begin{equation*} 
	\zeta - \xi \leq \frac{2 \delta a (|\xi|+|\zeta|) + 2 \delta^2}{a^2 (|\xi|+|\zeta|)} < \frac{3\delta}{a} . 
\end{equation*}
Observe that 
\begin{align} \label{separation_eqn10} 
	&\mathcal{G}(w(x_1,z),v(x_1,z)) = \sqrt{2} \,|w_1(x_1,z) - v_1(x_1,z)| \text{ if } x_1 \not\in [\xi,\zeta],  \\ 
	&\mathcal{G}(w(x_1,z),v(x_1,z)) = \sqrt{2} \,|w_1(x_1,z) + v_1(x_1,z)| \text{ if } x_1 \in [\xi,\zeta]. \nonumber
\end{align}

By the Taylor series expansion for $w_1$ and $v_1$ and \eqref{separation_eqn2} and \eqref{separation_eqn4}, 
\begin{align} \label{separation_eqn11} 
	| w_1(x_1,z) - w_1(\xi,z) - a (x_1-\xi) e_1 | \leq C \delta |x_1-\xi| , \\
	| v_1(x_1,z) - v_1(\xi,z) - a (x_1-\xi) e_1 | \leq C \delta |x_1-\xi| , \nonumber 
\end{align}
where we let $C \in (0,\infty)$ denotes universal constants. 
By substituting \eqref{separation_eqn11} with $x_1 = \zeta$ into $w_1(\zeta,z) \cdot v_1(\zeta,z) = 0$ and using $w_1(\xi,z) \cdot v_1(\xi,z) = 0$ 
we obtain 
\begin{align*}
	0 = | w_1(\zeta,z) \cdot v_1(\zeta,z) | \geq{}& | a (\zeta-\xi) e_1 \cdot (w_1(\xi,z) + v_1(\xi,z)) + a^2 (\zeta-\xi)^2 | 
		\\&\hspace{5mm} - C\delta (\zeta-\xi) (|w_1(\xi,z)| + |v_1(\xi,z)|) - C \delta a (\zeta-\xi)^2 
		\end{align*}
and thus 
\begin{equation*}
	| (w_1(\xi,z) + v_1(\xi,z)) \cdot e_1 + a (\zeta-\xi) | \leq (C\delta/a) (|w_1(\xi,z)| + |v_1(\xi,z)|) + C \delta (\zeta - \xi) 
\end{equation*}
Since $w_1(\xi,z) \cdot v_1(\xi,z) = 0$ implies that $|w_1(\xi,z)| + |v_1(\xi,z)| \leq \sqrt{2} \,|w_1(\xi,z) + v_1(\xi,z)|$, 
\begin{equation*}
	| (w_1(\xi,z) + v_1(\xi,z)) \cdot e_1 + a (\zeta-\xi) | \leq (C\delta/a) |w_1(\xi,z) + v_1(\xi,z)| + C \delta (\zeta - \xi) 
\end{equation*}
Hence by the triangle inequality,  
\begin{equation} \label{separation_eqn12} 
	| (w_1(\xi,z) + v_1(\xi,z)) \cdot e_1 + a (\zeta-\xi) | 
		\leq (C\delta/a) |w_1(\xi,z)^{\perp} + v_1(\xi,z)^{\perp}| + C \delta (\zeta - \xi) , 
\end{equation}
where we let $\eta^{\perp} = \eta - (\eta \cdot e_1) \,e_1$ for any $\eta \in \mathbb{R}^m$.

By \eqref{separation_eqn11} and \eqref{separation_eqn12} 
\begin{align} \label{separation_eqn13} 
	&\int_{\xi}^{\zeta} |w_1(x_1,z) + v_1(x_1,z)|^2 \,dx_1 \\
	&\hspace{15mm} \geq \frac{1}{2} \int_{\xi}^{\zeta} |w_1(\xi,z) + v_1(\xi,z) + 2 a (x_1-\xi) e_1|^2 \,dx_1 
		 - C \delta^2 (\xi-\zeta)^3 \nonumber \\
		 	&\hspace{15mm} = \frac{1}{2} \int_{\xi}^{\zeta} |(w_1(\xi,z) + v_1(\xi,z)) \cdot e_1 + 2 a (x_1-\xi)|^2 \,dx_1 \nonumber \\
		&\hspace{30mm} + \frac{1}{2} \int_{\xi}^{\zeta} |w_1(\xi,z)^{\perp} + v_1(\xi,z)^{\perp}|^2 \,dx_1 - C \delta^2 (\xi-\zeta)^3 \nonumber \\
		&\hspace{10mm} \geq \frac{1}{4} \int_{\xi}^{\zeta} a^2 (2x_1-\zeta-\xi)^2 \,dx_1 
		+ \frac{1}{4} \int_{\xi}^{\zeta} |w_1(\xi,z)^{\perp} + v_1(\xi,z)^{\perp}|^2 \,dx_1 - C \delta^2 (\xi-\zeta)^3 \nonumber \\ 
		&\hspace{10mm} \geq \frac{1}{24} a^2 \,(\zeta-\xi)^3 + \frac{1}{4} |w_1(\xi,z)^{\perp} + v_1(\xi,z)^{\perp}|^2 (\xi-\zeta). \nonumber 
\end{align}
On the other hand, by \eqref{separation_eqn11}, 

\begin{equation*}
	\int_{\xi}^{\zeta} |w_1(x_1,z) - v_1(x_1,z)|^2 \,dx_1 
	\leq 2 |w_1(\xi,z) - v_1(\xi,z)|^2 (\zeta-\xi) + C \delta^2 (\zeta-\xi)^3. \nonumber 
	\end{equation*}
By $|w_1(\xi,z) - v_1(\xi,z)| = |w_1(\xi,z) + v_1(\xi,z)|$ and \eqref{separation_eqn12}, 
\begin{align} \label{separation_eqn14} 
	&\int_{\xi}^{\zeta} |w_1(x_1,z) - v_1(x_1,z)|^2 \,dx_1 
	\\&\hspace{15mm} \leq 2 |w_1(\xi,z) + v_1(\xi,z)|^2 (\zeta-\xi) + C \delta^2 (\zeta-\xi)^3 \nonumber 
	\\&\hspace{15mm} = 2 |(w_1(\xi,z) + v_1(\xi,z)) \cdot e_1|^2 (\zeta-\xi) \nonumber
		\\&\hspace{30mm} + 2 |w_1(\xi,z)^{\perp} + v_1(\xi,z)^{\perp}|^2 (\zeta-\xi) + C \delta^2 (\zeta-\xi)^3 \nonumber 
	\\&\hspace{15mm} \leq 3 a^2 (\zeta-\xi)^3 + 3 |w_1(\xi,z)^{\perp} + v_1(\xi,z)^{\perp}|^2 (\zeta-\xi) . \nonumber
\end{align}

By combining \eqref{separation_eqn13} and \eqref{separation_eqn14}, 
\begin{equation*}
	\int_{\xi}^{\zeta} |w_1(x_1,z) - v_1(x_1,z)|^2 \,dx_1 \leq 72 \int_{\xi}^{\zeta} |w_1(x_1,z) + v_1(x_1,z)|^2 \,dx_1 
	\end{equation*}
for all $z \in B^{n-1}_{(3+5\gamma)/8}(0)$.  By \eqref{separation_eqn10} and $\xi - \zeta \leq 3\delta/a$, 
\begin{equation} \label{separation_eqn15} 
	\int_{-\rho(z)}^{\rho(z)} |w_1(x_1,z) - v_1(x_1,z)|^2 \,dx_1 
		\leq 36 \int_{-\rho(z)-3\delta/a}^{\rho(z)+3\delta/a} \mathcal{G}(w(x_1,z), v(x_1,z))^2 \,dx_1 
\end{equation}
for all $z \in B^{n-1}_{(3+5\gamma)/8}(0)$, where $\rho(z) = \sqrt{(3+5\gamma)^2/64 - |z|^2}$.  Integrating over $z \in B^{n-1}_{(3+5\gamma)/8}(0)$, requiring that $\delta \leq 5(1-\gamma)a/24$, 
\begin{equation} \label{separation_eqn16} 
	\int_{B_{(3+5\gamma)/8}(0)} |w_1 - v_1|^2 \leq C \int_{B_1(0)} \mathcal{G}(w, v)^2 
\end{equation}
for some constant $C = C(n,\gamma) \in (0,\infty)$.  Combining \eqref{separation_eqn16} with standard elliptic estimates for $\Delta (w_1^{\kappa} - v_1^{\kappa}) = -f_{1,\kappa}$ on $B_{(3+5\gamma)/8}(0)$, where $f_{1,\kappa} = -D_i (b_{\kappa \lambda}^{ij}(Du_a,Du_s) \,D_j w_1)$ ($b_{\kappa \lambda}^{ij}$ as in \eqref{mss b defn}), gives us \eqref{separation_concl}. 
\end{proof}

The following lemma is the analogue of~\cite[Lemma 2.6]{Sim93}.  Note that we use coordinates $X = (x,y)$ on $\mathbb{R}^n$, where $x = (x_1,x_2)$ and $y = (x_3,\ldots,x_n)$. 

\begin{lemma} \label{lemma2_6}  
Let $\alpha$ and $\varphi^{(0)}$ be as in Definition~\ref{varphi0 defn}.  Let $\beta, \gamma, \tau \in (0,1)$ with $\tau \leq (1-\gamma)/20$.  There is an $\varepsilon_0 = \varepsilon_0(n,m,\alpha,\varphi^{(0)},\gamma,\beta,\tau) \in (0,1)$ such that if $\varphi \in \Phi_{\varepsilon_0}(\varphi^{(0)})$ and $(u,\Lambda) \in \mathcal{F}_{\varepsilon_0}(\varphi^{(0)})$, there is an open set $U \subset B_1(0)$ such that 
\begin{gather}
	\label{lemma2_6 concl1} (x,y) \in U \Rightarrow (\widetilde{x},y) \in U \text{ whenever } |x| = |\widetilde{x}|, \\
	\label{lemma2_6 concl2} \{ (x,y) \in B_{\gamma}(0) : |x| > \tau \} \subset U,  
\end{gather}
and there is a single-valued function $v \in C^2(\op{graph} \varphi |_U,\mathbb{R}^m)$ and an associated two-valued harmonic function $\widehat{v} \in C^2(U,\mathbb{R}^m)$ given by \eqref{hatv defn} such that \eqref{v graph rep} holds true and 
\begin{equation*}
	\sup_{B_{\gamma}(0) \cap U} r^{-\alpha} |\widehat{v}| + \sup_{B_{\gamma}(0) \cap U} r^{1-\alpha} |D\widehat{v}| \leq \beta 
\end{equation*}
and 
\begin{align*}
	&\int_{B_{\gamma}(0) \cap U} (|\widehat{v}|^2 + r^2 |D\widehat{v}|^2) + \frac{1}{\Lambda^2} \int_{B_{\gamma}(0) \setminus U} (|u_s|^2 + r^2 |Du_s|^2) 
	\\&\hspace{35mm} \leq C \int_{B_1(0)} \mathcal{G}\left(\frac{u_s}{\Lambda}, \varphi\right)^2 + C \int_{B_{(3+\gamma)/4}(0)} r^4 |f|^2, 
\end{align*}
where $r = r(x,y) = |x|$ for $X = (x,y) \in B_1(0)$, $f$ is as in \eqref{mss3} with $w = u_s/\Lambda$ and $C = C(n,m,\alpha,\varphi^{(0)},\gamma,\beta) \in (0,\infty)$ is a constant.
\end{lemma}
\begin{proof} 
Set $w = u_s/\Lambda$.  For each $\zeta \in \mathbb{R}^n$ and $\rho \in (0,1]$, let 
\begin{equation*}
	A_{\rho,\kappa}(\zeta) = \{ (x,y) \in \mathbb{R}^2 \times \mathbb{R}^{n-2} : (|x| - \rho)^2 + |y - \zeta|^2 < \kappa^2 (1-\gamma)^2 \rho^2/4 \} .
\end{equation*}
Note that $A_{\rho,1}(\zeta) \subset B_1(0)$ whenever $\rho^2 + |\zeta|^2 < (1+\gamma)^2/4$.  By Lemma~\ref{separation_lemma}, there exists $\delta = \delta(n,m,\varphi^{(0)},\beta) > 0$ such that if $\rho^2 + |\zeta|^2 < (1+\gamma)^2/4$ and $\int_{A_{\rho,3/4}(\zeta)} \mathcal{G}(w,\varphi)^2 < \delta \rho^{n+2\alpha}$, then there exists a (unique) single-valued function $v_{\rho,\zeta} \in C^2(\op{graph} \varphi |_{A_{\rho,1/2}(\zeta)},\mathbb{R}^m)$ such that 
\begin{equation*}
	w(X) = \{ \varphi_1(X) + v_{\rho,\zeta}(X,\varphi_1(X)), -\varphi_1(X) + v_{\rho,\zeta}(X,-\varphi_1(X)) \}
\end{equation*}
for all $X \in A_{\rho,1/2}(\zeta)$, where $\varphi(X) = \{ \pm \varphi_1(X) \}$, and $\widehat{v}(X) = \{ v_{\rho,\zeta}(X,\varphi_1(X)), v_{\rho,\zeta}(X,-\varphi_1(X)) \}$ satisfies 
\begin{equation} \label{lemma2_6_eqn1}
	\rho^{-\alpha} \sup_{A_{\rho,1/2}(\zeta)} |\widehat{v}_{\rho,\zeta}| + \rho^{1-\alpha} \sup_{A_{\rho,1/2}(\zeta)} |D\widehat{v}_{\rho,\zeta}| 
		\leq 2^{-\alpha} \beta. 
\end{equation}
Let $U$ be the union of the $A_{\rho,1/2}(\zeta)$ such that $\rho^2 + |\zeta|^2 < (1+\gamma)^2/4$ and $\int_{A_{\rho,3/4}(\zeta)} \mathcal{G}(w,\varphi)^2 < \delta \rho^{n+2\alpha}$.  We then obtain a well-defined single-valued function $v \in C^2(\op{graph} \varphi |_U,\mathbb{R}^m)$ by requiring that $v = v_{\rho,\zeta}$ on ${\rm graph}\,\varphi |_{A_{\rho,1/2}(\zeta)}$ whenever $\rho^2 + |\zeta|^2 < (1+\gamma)^2/4$ and $\int_{A_{\rho,3/4}(\zeta)} \mathcal{G}(w,\varphi)^2 < \delta \rho^{n+2\alpha}$.  Since $\varphi \in \Phi_{\varepsilon_0}(\varphi^{(0)})$ and $(u,\Lambda) \in \mathcal{F}_{\varepsilon_0}(\varphi^{(0)})$, \eqref{lemma2_6 concl2} holds true 
provided $\varepsilon_0$ is sufficiently small.  By \eqref{lemma2_6_eqn1}, 
\begin{equation} \label{lemma2_6_eqn2}
	\sup_{B_{(1+\gamma)/2}(0) \cap U} r^{-\alpha} |\widehat{v}| + \sup_{B_{(1+\gamma)/2}(0) \cap U} r^{1-\alpha} |D\widehat{v}| \leq \beta.
\end{equation}

For each $(\xi,\zeta) \in B_{(1+\gamma)/2}(0) \cap \partial U$,  
\begin{equation} \label{lemma2_6_eqn3}
	\int_{A_{|\xi|,3/4}(\zeta)} \mathcal{G}(w,\varphi)^2 \geq \delta |\xi|^{n+2\alpha}
\end{equation}
otherwise $A_{|\xi|,1/2}(\zeta)$ is an open neighborhood of $(\xi,\zeta)$ with $A_{|\xi|,1/2}(\zeta) \subseteq U$.  Hence by \eqref{lemma2_6_eqn2} and \eqref{lemma2_6_eqn3}, for each $(\xi,\zeta) \in B_{(1+\gamma)/2}(0) \cap \partial U$, 
\begin{equation*}
	\int_{B_{10|\xi|}(0,\zeta) \cap B_{(1+\gamma)/2}(0) \cap U} (|\widehat{v}|^2 + r^2 |D\widehat{v}|^2) 
	\leq C \beta^2 \delta^{-1} \int_{A_{|\xi|,3/4}(\zeta)} \mathcal{G}(w,\varphi)^2
\end{equation*}
for some constant $C = C(n,\alpha) \in (0,\infty)$.  Since  
\begin{gather*}
	\{ X \in B_{(1+\gamma)/2}(0) \cap U : \op{dist}(X, B_{(1+\gamma)/2}(0) \cap \partial U) < r/2 \} \subset 
	\bigcup_{(\xi,\zeta) \in B_{(1+\gamma)/2} \cap \partial U} B_{2|\xi|}(0,\zeta) , \\
B_{2|\xi|}(0,\zeta) \cap B_{2|\xi'|}(0,\zeta') = \emptyset \hspace{3mm} \Rightarrow \hspace{3mm} 
	A_{|\xi|,3/4}(\zeta) \cap A_{|\xi'|,3/4}(\zeta') = \emptyset, 
\end{gather*}
by the Vitali covering lemma, 
\begin{equation} \label{lemma2_6_eqn4}
	\int_{\{ X \in B_{(1+\gamma)/2}(0) \cap U \,:\, \op{dist}(X, B_{(1+\gamma)/2} \cap \partial U) < r/2 \}} (|\widehat{v}|^2 + r^2 |D\widehat{v}|^2) 
	\leq C \int_{B_1(0)} \mathcal{G}(w,\varphi)^2 
\end{equation}
for $C = C(n,m,\alpha,\varphi^{(0)},\beta) \in (0,\infty)$.  Since $w$ satisfies \eqref{mss3} and $\varphi$ is harmonic, for each open ball $B \subset U$, $\widehat{v}$ satisfies 
\begin{equation} \label{lemma2_6 hatv pde}
	\Delta \widehat{v}_1^{\kappa} = f_{1,\kappa} 
\end{equation} 
in $B$, where in $B$ we let $\varphi(X) = \{\pm \varphi_1(X)\}$ for some single-valued harmonic function $\varphi_1 : B \rightarrow \mathbb{R}^m$, $\widehat{v}_1(X) = v(X,\varphi_1(X))$, and 
\begin{equation*} 
	f_{1,\kappa} = -D_i (b_{\kappa \lambda}^{ij}(Du_a,Du_s) \,D_j (\varphi_1^{\kappa} + \widehat{v}_1^{\kappa})) .
\end{equation*}
Thus $f_{\kappa} = \{ \pm f_{1,\kappa} \}$ is as in \eqref{mss3} with $w = u_s/\Lambda$.  
Define $d : U \rightarrow [0,\infty)$ by $d(X) = \op{dist}(X, \partial U \cap B_{(1+\gamma)/2}(0))$ and note that $d$ is Lipschitz with Lipschitz constant $1$.  Let $\psi : [0,\infty) \rightarrow [0,1]$ be a smooth function such that $\psi(t) = 0$ for $t \leq 1/4$, $\psi(t) = 1$ for $t \geq 1/2$, and $0 \leq \psi'(t) \leq 6$ for all $t$.  Let $\eta : B_1(0) \rightarrow [0,1]$ be a smooth function such that $\eta = 1$ on $B_{\gamma}(0)$, $\eta = 0$ on $B_1(0) \setminus B_{(1+\gamma)/2}(0)$ and $|D\eta| \leq 3/(1-\gamma)$.  By multiplying both sides of \eqref{lemma2_6 hatv pde} by $\widehat{v} r^2 \psi(d/r)^2 \eta^2$ and integrating by parts, 
\begin{align*}
	\int_U |D\widehat{v}|^2 r^2 \psi(d/r)^2 \eta^2 
	= {}& -2 \int_U \widehat{v}^{\kappa} D_i \widehat{v}^{\kappa} \left( x_i \psi(d/r)^2 \eta^2 
		+ r^2 \psi(d/r) \psi'(d/r) \eta^2 \left( \frac{D_i d}{r} - \frac{d x_i}{r^3} \right) \right) \\
	&\hspace{10mm} - 2 \int_U \widehat{v}^{\kappa} D_i \widehat{v}^{\kappa} r^2 \psi(d/r)^2 \eta D_i \eta 
		- \int_U f_{\kappa} \widehat{v}^{\kappa} r^2 \psi(d/r)^2 \eta^2
\end{align*}
Using Cauchy's inequality and the definition of $\psi$, 

\begin{align*}
	\int_U |D\widehat{v}|^2 r^2 \psi(d/r)^2 \eta^2 
	\leq {}& \int_U |\widehat{v}|^2 (13 \psi(d/r)^2 \eta^2 + 27 \psi'(d/r)^2 \eta^2 + 12 r^2 \psi(d/r)^2 |D\eta|^2) 
		\\&\hspace{10mm} + \int_U r^4 |f|^2 \psi(d/r)^2 \eta^2.
\end{align*}
Then using the definitions of $\psi$ and $\eta$, 
\begin{equation} \label{lemma2_6_eqn6}
	\int_{\{ X \in U \cap B_{\gamma}(0) \,:\, d(X) \geq r/2 \}} r^2 |D\widehat{v}|^2  
	\leq C \int_{U \cap B_{(1+\gamma)/2}(0)} |\widehat{v}|^2 + \int_{U \cap B_{(1+\gamma)/2}(0)} r^4 |f|^2  
\end{equation}
for some constant $C = C(\gamma) \in (0,\infty)$.  By the argument showing \eqref{separation_eqn15} in the proof of Lemma~\ref{separation_lemma}, provided $\varepsilon_0$ and $\delta$ are sufficiently small, whenever $r > \tau$ and $r^2 + |y-\zeta|^2 < (1+\gamma)^2/4$, 
\begin{equation*}
	\int_{\partial B^2_r(\xi) \times \{y\}} |\widehat{v}|^2 \leq C \int_{\partial B^2_r(\xi) \times \{y\}} \mathcal{G}(w,\varphi)^2 
\end{equation*}
for some constant $C \in (0,\infty)$.  Thus integrating over $\{ (r,y) : r^2 + |y-\zeta|^2 < (1+\gamma)^2/4 \}$,  
\begin{equation} \label{lemma2_6_eqn7}
	\int_{U \cap B_{(1+\gamma)/2}(0)} |\widehat{v}|^2 \leq C \int_{U \cap B_{(1+\gamma)/2}(0)} \mathcal{G}(w,\varphi)^2 
\end{equation}
for some constant $C = C(n,\gamma) \in (0,\infty)$.  Combining \eqref{lemma2_6_eqn6} and \eqref{lemma2_6_eqn7} gives us  
\begin{equation} \label{lemma2_6_eqn8}
	\int_{\{ X \in U \cap B_{\gamma}(0) \,:\, d(X) \geq r/2 \}} (|\widehat{v}|^2 + r^2 |D\widehat{v}|^2)  
		\leq C \int_{B_{(1+\gamma)/2}(0)} \mathcal{G}(w,v)^2 + C \int_{B_{(1+\gamma)/2}(0)} r^4 |f|^2  
\end{equation}
for some constant $C = C(n,\gamma) \in (0,\infty)$.  Therefore by \eqref{lemma2_6_eqn4} and \eqref{lemma2_6_eqn8}, 
\begin{equation*}
	\int_{U \cap B_{\gamma}(0)} (|\widehat{v}|^2 + r^2 |D\widehat{v}|^2) \leq C \int_{B_1(0)} \mathcal{G}(w,\varphi)^2 + \int_{B_{(1+\gamma)/2}(0)} r^4 |f|^2 
\end{equation*}
for some constant $C = C(n,m,\alpha,\varphi^{(0)},\gamma,\beta) \in (0,\infty)$. 

Observe that if $\rho^2 + |\zeta|^2 < (1+\gamma)^2/4$ such that $\int_{A_{\rho,3/4}(\zeta)} \mathcal{G}(w,\varphi)^2 \geq \delta \rho^{n+2\alpha}$, by elliptic estimates for \eqref{mss3}, i.e.~$\Delta w^{\kappa} = f_{\kappa}$ in $B_1(0)$, 
\begin{equation} \label{lemma2_6_eqn10} 
	\int_{A_{\rho,1/2}(\zeta)} (|w|^2 + r^2 |Dw|^2) \leq C \int_{A_{\rho,3/4}(\zeta)} |w|^2 + C \rho^4 \int_{A_{\rho,3/4}(\zeta)} |f|^2  
\end{equation}
for some constant $C = C(\gamma) \in (0,\infty)$.  By the triangle inequality and the homogeneity of $\varphi$,  
\begin{align*} 
	\int_{A_{\rho,3/4}(\zeta)} |w|^2 
	&\leq 2 \int_{A_{\rho,3/4}(\zeta)} \mathcal{G}(w,\varphi)^2 + 2 C \rho^{n+2\alpha} 
	\\&\leq 2 (1 + C \delta^{-1}) \int_{A_{\rho,3/4}(\zeta)} \mathcal{G}(w,\varphi)^2. 
\end{align*}
for $C = C(n,\alpha,\varphi^{(0)}) \in (0,\infty)$.  Hence \eqref{lemma2_6_eqn10} gives us 
\begin{equation} \label{lemma2_6_eqn11} 
	\int_{A_{\rho,1/2}(\zeta)} (|w|^2 + r^2 |Dw|^2)
	\leq C \int_{A_{\rho,3/4}(\zeta)} \mathcal{G}(w,\varphi)^2 + C \rho^4 \int_{A_{\rho,3/4}(\zeta)} |f|^2.  
\end{equation}
for $C = C(n,m,\alpha,\varphi^{(0)},\gamma,\beta) \in (0,\infty)$ for all $\zeta \in B^{n-2}_{\gamma}(0)$ and $\rho \in (0,1/2)$.  By \eqref{lemma2_6_eqn11} and an easy covering argument, it follows that 
\begin{equation*} 
	\int_{B_{\gamma}(0) \setminus U} (|w|^2 + r^2 |Dw|^2) 
		\leq C \int_{B_1(0)} \mathcal{G}(w,\varphi)^2 + C \int_{B_{(3+\gamma)/4}(0)} r^4 |f|^2 
\end{equation*}
for some constant $C = C(n,m,\alpha,\varphi^{(0)},\gamma,\beta) \in (0,\infty)$.  
\end{proof}

\section{A priori estimates: non-concentration of excess} \label{sec:nonconcentration sec}

Let $\varphi^{(0)}$ be as in Definition~\ref{varphi0 defn} and let $(u,\Lambda) \in \mathcal{F}_{\varepsilon_0}(\varphi^{(0)})$ with $u(0) = \{0,0\}$ and $Du(0) = \{0,0\}.$ Let $M = {\rm graph}\,u$ and $w = u_s/\Lambda$.  In particular, $0 \in \mathcal{K}_M$.  Set 
\begin{equation*}
	H_{w,0}(\rho) = \rho^{1-n} \int_{\partial B_{\rho}(0)} |w|^2, \quad 
	D_{w,0}(\rho) = \rho^{2-n} \int_{B_{\rho}(0)} |Dw|^2
\end{equation*}
for each $\rho \in (0,1]$.  (Thus $H_{w,0} = H_{M,0}/\Lambda$ and $D_{w,0} = D_{M,0}/\Lambda$ where $H_{M,0}$ and $D_{M,0}$ are as in Definition~\ref{freqfn defn}.)

\begin{lemma} \label{modified freq est lemma}
Let $\alpha \in \mathbb{R}$.  If $u \in C^{1,1/2}(B_1(0),\mathbb{R}^m)$ is a two-valued function such that $u(0) = \{0,0\}$, $Du(0) = \{0,0\}$, and the graph $M$ of $u$ is stationary in $B_1(0) \times \mathbb{R}^m$, then $w = u_s/\Lambda$ satisfies 
\begin{align} \label{modified freq est} 
	&\frac{d}{d\rho} (\rho^{-2\alpha} (D_{w,0}(\rho) - \alpha H_{w,0}(\rho))) 
		\\&\hspace{15mm} = 2\rho^{-n-2\alpha} \int_{\partial B_{\rho}(0)} |R D_R w - \alpha w|^2 
		- 2\rho^{1-n-2\alpha} \int_{B_{\rho}(0)} f_{\kappa} (R D_R w^{\kappa} - \alpha w^{\kappa}) \nonumber 
\end{align}
for each $\rho \in (0,1)$, where $R = R(X) = |X|$ and $f_{\kappa}$ is as in \eqref{mss3} (with $w = u_s/\Lambda$). 
\end{lemma} 
\begin{proof}
Arguing as in~\cite[Lemma 6.6]{SimWic16}, observe that since $w \in W^{2,2}_{\rm loc}(\Omega,\mathcal{A}_2(\mathbb{R}^m))$, $w^{\kappa} D_i w^{\kappa} \in W^{1,2}(\Omega)$ and $D_i w^{\kappa} D_j w^{\kappa} \in W^{1,2}(\Omega)$ with $D_l (w^{\kappa} D_i w^{\kappa}) = D_l w^{\kappa} D_i w^{\kappa} + w^{\kappa} D_{il} w^{\kappa}$ and $D_l (D_i w^{\kappa} D_j w^{\kappa}) = D_{il} w^{\kappa} D_j w^{\kappa} + D_i w^{\kappa} D_{jl} w^{\kappa}$ a.e.~in $\Omega$ for all $1 \leq i,j,l \leq n$.  Hence by \eqref{mss3} and integration by parts 
\begin{gather} 
	\label{monotonicity_identity1} \int_{\Omega} |Dw|^2 \zeta = -\int_{\Omega} w^{\kappa} D_i w^{\kappa} D_i \zeta 
		- \int_{\Omega} f_{\kappa} w^{\kappa} \zeta, \\
	\label{monotonicity_identity2} \int_{\Omega} \left( \tfrac{1}{2} |Dw|^2 \delta_{ij} - D_i w^{\kappa} D_j w^{\kappa} \right) D_i \zeta^j 
		= \int_{\Omega} f_{\kappa} D_j w^{\kappa} \zeta^j, 
\end{gather}
for all $\zeta, \zeta^1,\ldots,\zeta^n \in C^1_c(\Omega,\mathbb{R})$.  (These identities are analogous to the ``squash'' and ``squeeze'' deformations of~\cite[Sections 2.4 and 2.5]{Almgren}.)  By a standard argument (c.f.~\cite[Lemma 6.6]{SimWic16}), we can let $\zeta$ approximate the characteristic function $B_{\rho}(0)$ in \eqref{monotonicity_identity1} and $\zeta^j(X) = \zeta(X) \,x_j$ in \eqref{monotonicity_identity2} and thereby show that  
\begin{align*}
	D_{w,0}(\rho) &= \rho^{2-n} \int_{B_{\rho}(0)} |Dw|^2 
		= \rho^{2-n} \int_{\partial B_{\rho}(0)} w^{\kappa} D_R w^{\kappa} - \rho^{2-n} \int_{B_{\rho}(0)} f_{\kappa} w^{\kappa}, \\
	D'_{w,0}(\rho) &= 2\rho^{2-n} \int_{\partial B_{\rho}(0)} |D_R w|^2 - 2\rho^{1-n} \int_{B_{\rho}(0)} f_{\kappa} R D_R w^{\kappa}, \nonumber \\
	H'_{w,0}(\rho) &= 2\rho^{1-n} \int_{\partial B_{\rho}(0)} w^{\kappa} D_R w^{\kappa}. \nonumber 
\end{align*}
Hence 
\begin{align*}
	&\frac{d}{d\rho} (\rho^{-2\alpha} (D_{w,0}(\rho) - \alpha H_{w,0}(\rho))) 
	\\&\hspace{15mm} = \rho^{-2\alpha} (D'_{w,0}(\rho) - \alpha H'_{w,0}(\rho)) - 2 \alpha \rho^{2\alpha-1} (D_{w,0}(\rho) - \alpha H_{w,0}(\rho))
	\\&\hspace{15mm} = 2\rho^{-n-2\alpha} \int_{\partial B_{\rho}(0)} (R^2 |D_R w|^2 - 2 \alpha w^{\kappa} R D_R w^{\kappa} + \alpha^2 |w|^2) 
		\\&\hspace{30mm} - 2\rho^{1-n-2\alpha} \int_{B_{\rho}(0)} f_{\kappa} (R D_R w^{\kappa} - \alpha w^{\kappa}) 
	\\&\hspace{15mm} = 2\rho^{-n-2\alpha} \int_{\partial B_{\rho}(0)} |R D_R w - \alpha w|^2 
		- 2\rho^{1-n-2\alpha} \int_{B_{\rho}(0)} f_{\kappa} (R D_R w^{\kappa} - \alpha w^{\kappa}) . \qedhere
\end{align*}
\end{proof}

\begin{theorem} \label{mainestthm} 
Let $\alpha$ and $\varphi^{(0)}$ be as in Definition~\ref{varphi0 defn}.  For each $\gamma \in (0,1)$ there exists $\varepsilon_0 = \varepsilon_0(n,m,\alpha,\varphi^{(0)},\gamma) \in (0,1)$ such that if $\varphi \in \Phi_{\varepsilon_0}(\varphi^{(0)})$ and $(u,\Lambda) \in \mathcal{F}_{\varepsilon_0}(\varphi^{(0)})$ such that $u(0) = \{0,0\}$, $Du(0) = \{0,0\}$ and $\mathcal{N}_M(0) \geq \alpha$, where $M = {\rm graph}\,u$, then 
\begin{align*}
	&(a) \hspace{5mm} \frac{1}{\Lambda^2} \int_{B_{\gamma}(0)} |D_y u_s|^2 
		\leq C \int_{B_1(0)} \mathcal{G}\left(\frac{u_s}{\Lambda}, \varphi\right)^2 + C \|Du\|_{C^0(B_1(0))}^2  
	\\&(b) \hspace{5mm} \frac{1}{\Lambda^2} \int_{B_{\gamma}(0)} R^{2-n} \left| \frac{\partial (u_s/R^{\alpha})}{\partial R} \right|^2 
		\leq C \int_{B_1(0)} \mathcal{G}\left(\frac{u_s}{\Lambda}, \varphi\right)^2 + C \|Du\|_{C^0(B_1(0))}^2  
\end{align*}
for some constant $C = C(n,m,\alpha,\varphi^{(0)},\gamma) \in (0,\infty)$, where $R = R(X) = |X|$.
\end{theorem}

\begin{proof}
Set $w = u_s/\Lambda$.  Suppose that $\varepsilon$ is less than or equal to $\varepsilon_0$ in Lemma~\ref{lemma2_6} with $(1+\gamma)/2$ in place of $\gamma$, $\tau = (1-\gamma)/20$ and $\beta = 1/2$.  Let $U$, $v$ and $\widehat{v}$ be as in Lemma~\ref{lemma2_6}. 

By \eqref{modified freq est}, 
\begin{align} \label{mainest_eqn5}
	&\frac{d}{d\rho} (\rho^{n-2} (D_{w,0}(\rho) - \alpha H_{w,0}(\rho))) 
	= (n-2+2\alpha) \rho^{n-3+2\alpha} \cdot \rho^{-2\alpha} (D_{w,0}(\rho) - \alpha H_{w,0}(\rho)) \\
	&\hspace{10mm} + 2\rho^{n-2+2\alpha} \left( \rho^{-n-2\alpha} \int_{\partial B_{\rho}(0)} |R D_R w - \alpha w|^2 
	- \rho^{1-n-2\alpha} \int_{B_{\rho}(0)} f_{\kappa} (R D_R w^{\kappa} - \alpha w^{\kappa}) \right) \nonumber
\end{align}
for $\rho \in (0,1]$.  By $\mathcal{N}_M(0) \geq \alpha$ and \eqref{H_decay}, $\liminf_{\rho \downarrow 0} \rho^{-2\alpha} (D_{w,0}(\rho) - \alpha H_{w,0}(\rho)) \geq 0$.  Thus by integrating \eqref{modified freq est}, 
\begin{align} \label{mainest_eqn6}
	\rho^{-2\alpha} (D_{w,0}(\rho) - \alpha H_{w,0}(\rho)) 
	&\geq 2 \int_{B_{\rho}(0)} R^{-n-2\alpha} |R D_R w - \alpha w|^2 
	\\&\hspace{15mm} - 2 \int_0^{\rho} \tau^{1-n-2\alpha} \int_{B_{\tau}(0)} f_{\kappa} (R D_R w^{\kappa} - \alpha w^{\kappa}) d\tau \nonumber
\end{align}
for $\rho \in (0,1]$.  By \eqref{mainest_eqn5} and \eqref{mainest_eqn6}, 
\begin{align} \label{mainest_eqn7}
	&\frac{d}{d\rho} (\rho^{n-2} (D_{w,0}(\rho) - \alpha H_{w,0}(\rho))) 
	\geq 2 \frac{d}{d\rho} \left( \rho^{n-2-2\alpha} \int_{B_{\rho}(0)} R^{-n-2\alpha} |R D_R w - \alpha w|^2 \right) 
	\\&\hspace{15mm} - 2 \frac{d}{d\rho} \left( \rho^{n-2-2\alpha} \int_0^{\rho} \tau^{1-n-2\alpha} 
		\int_{B_{\tau}(0)} f_{\kappa} (R D_R w^{\kappa} - \alpha w^{\kappa}) d\tau \right) \nonumber 
\end{align}
for $\rho \in (0,1]$.  Let $\psi : [0,\infty) \rightarrow \mathbb{R}$ be a non-increasing smooth function with $\psi(t) = 1$ for $t \in [0,\gamma]$ and $\psi(t) = 0$ for $t \geq (1+\gamma)/2$.  Multiplying both sides of \eqref{mainest_eqn7} by $\psi(\rho)^2$ and integrating using integration by parts yields 
\begin{align} \label{mainest_eqn8}
	&\int_{0}^{\infty} \frac{d}{d\rho} (\rho^{n-2} D_{w,0}(\rho) - \alpha \rho^{n-2} H_{w,0}(\rho)) \,\psi(\rho)^2 \,d\rho
	\\&\hspace{15mm} \geq -4 \int_{0}^{\infty} \rho^{n-2-2\alpha} \int_{B_{\rho}(0)} R^{-n-2\alpha} |R D_R w - \alpha w|^2 \,\psi(\rho) \,\psi'(\rho) \,d\rho 
		\nonumber 
	\\&\hspace{30mm} + 4 \int_{0}^{\infty} \rho^{n-2-2\alpha} \int_0^{\rho} \tau^{1-n-2\alpha} 
		\int_{B_{\tau}(0)} f_{\kappa} (R D_R w^{\kappa} - \alpha w^{\kappa}) \psi(\rho) \psi'(\rho) \,d\tau \,d\rho \nonumber 
\end{align}
By Fubini's theorem and Cauchy-Schwarz, 
\begin{align*} 
	&\int_0^{\rho} \tau^{1-n-2\alpha} \int_{B_{\tau}(0)} |f| \,|R D_R w - \alpha w| \,dX \,d\tau
	\\&\hspace{15mm} = \int_{B_{\rho}(0)} \int_R^{\rho} \tau^{1-n-2\alpha} \,|f| \,|R D_R w - \alpha w| \,d\tau \,dX \nonumber 
	\\&\hspace{15mm} \leq \int_{B_{\rho}(0)} \frac{R^{2-n-2\alpha}}{n-2+2\alpha} |f| |R D_R w - \alpha w| \nonumber
	\\&\hspace{15mm} \leq \frac{1}{2} \int_{B_{\rho}(0)} R^{-n-2\alpha} |R D_R w - \alpha w|^2 
		+ \frac{1}{2(n-2+2\alpha)^2} \int_{B_{\rho}(0)} R^{4-n-2\alpha} |f|^2 \nonumber
\end{align*}
for each $\gamma \leq \rho \leq (1+\gamma)/2$.  Hence we can bound the last term on the right-hand side of \eqref{mainest_eqn8} by 
\begin{align} \label{mainest_eqn10}
	&\int_{0}^{\infty} \frac{d}{d\rho} (\rho^{n-2} D_{w,0}(\rho) - \alpha \rho^{n-2} H_{w,0}(\rho)) \,\psi(\rho)^2 \,d\rho 
	\\&\hspace{15mm} \geq -2 \int_{0}^{\infty} \rho^{n-2-2\alpha} \int_{B_{\gamma}(0)} R^{-n-2\alpha} |R D_R w - \alpha w|^2 \,\psi(\rho) \,\psi'(\rho) \,d\rho 
		\nonumber \\&\hspace{30mm} + \frac{2}{(n-2+2\alpha)^2} \int_{0}^{\infty} \rho^{n-2-2\alpha} \int_{B_{(1+\gamma)/2}(0)} R^{4-n-2\alpha} \,|f|^2 
			\,\psi(\rho) \,\psi'(\rho) \,d\rho \nonumber 
\end{align}
It follows from the definition of $\psi$ that $-2 \int_{0}^{\infty} \psi(\rho) \,\psi'(\rho) \,d\rho = 1$ and thus 
\begin{equation*}
	\gamma^{n-2-2\alpha} \leq -2 \int_{0}^{\infty} \rho^{n-2-2\alpha} \,\psi(\rho) \,\psi'(\rho) \,d\rho \leq \left(\frac{1+\gamma}{2}\right)^{n-2-2\alpha} . \end{equation*}
Hence \eqref{mainest_eqn10} gives us 
\begin{align} \label{mainest_eqn11}
	&\int_{0}^{\infty} \frac{d}{d\rho} (\rho^{n-2} D_{w,0}(\rho) - \alpha \rho^{n-2} H_{w,0}(\rho)) \,\psi(\rho)^2 \,d\rho 
	\\&\hspace{15mm} \geq 2 \gamma^{n-2-2\alpha} \int_{B_{\gamma}(0)} R^{-n-2\alpha} |R D_R w - \alpha w|^2 
		\nonumber \\&\hspace{30mm} - \frac{2}{(n-2+2\alpha)^2} \left(\frac{1+\gamma}{2}\right)^{n-2-2\alpha} 
			\int_{B_{(1+\gamma)/2}(0)} R^{4-n-2\alpha} \,|f|^2 \nonumber 
\end{align}
By the coarea formula 
\begin{equation} \label{mainest_eqn12}
	\int_{0}^{\infty} \frac{d}{d\rho} (\rho^{n-2} \,D_{w,0}(\rho)) \psi(\rho)^2 \,d\rho 
		= \int_{0}^{\infty} \int_{\partial B_{\rho}(0)} |Dw|^2 \psi(\rho)^2 \,d\rho 
		= \int |Dw|^2 \psi(R)^2 
\end{equation}
and by integration by parts and the coarea formula again 
\begin{align} \label{mainest_eqn13}
	&\int_{0}^{\infty} \frac{d}{d\rho} (\rho^{n-2} H_{w,0}(\rho)) \psi(\rho)^2 d\rho 
		= -2 \int_{0}^{\infty} \rho^{n-2} H_{w,0}(\rho) \psi(\rho) \psi'(\rho) d\rho 
		\\&\hspace{15mm} = -2 \int_{0}^{\infty} \int_{\partial B_{\rho}(0)} \rho^{-1} |w|^2 \psi(\rho) \psi'(\rho) \,d\rho 
		= -2 \int R^{-1} |w|^2 \psi(R) \psi'(R). \nonumber 
\end{align}
Hence by \eqref{mainest_eqn11}, \eqref{mainest_eqn12}, and \eqref{mainest_eqn13}, 
\begin{align} \label{mainest_eqn14}
	\int_{B_{\gamma}(0)} R^{2-n} \left| \frac{\partial (w/R^{\alpha})}{\partial R} \right|^2 
	\leq {}& C \int (|Dw|^2 \psi(R)^2 + 2\alpha R^{-1} |w|^2 \psi(R) \psi'(R))  
	\\&\hspace{10mm} + C \int_{B_{(\gamma+1)/2}(0)} R^{4-n-2\alpha} \,|f|^2 \nonumber
\end{align}
for $C = C(n,\alpha,\gamma) \in (0,\infty)$. 

Now let $(\zeta^1,\ldots,\zeta^n) = \psi(R)^2 \,(x_1,x_2,0,\ldots,0)$ in \eqref{monotonicity_identity2} to obtain  
\begin{align} \label{mainest_eqn15}
	&\int (|Dw|^2 - |D_x w|^2) \psi(R)^2 
	= -2 \int \left( \tfrac{1}{2} r^2 |Dw|^2 - r^2 |D_r w|^2 \right) R^{-1} \psi(R) \psi'(R) \\
	&\hspace{15mm} + 2 \int r D_r w^{\kappa} (y \cdot D_y w^{\kappa}) R^{-1} \psi(R) \psi'(R) 
		+ \int r f_{\kappa} D_r w^{\kappa} \psi(R)^2 \nonumber
\end{align}
where $r = r(X) = |(x_1,x_2)|$.  Let $\zeta = \alpha \,\psi(R)^2$ in \eqref{monotonicity_identity1} to obtain 
\begin{equation} \label{mainest_eqn16}
	\alpha \int |Dw|^2 \psi(R)^2 = -2 \alpha \int (r w^{\kappa} D_r w^{\kappa} + w^{\kappa} (y \cdot D_y w^{\kappa})) R^{-1} \psi(R) \psi'(R) 
	- \alpha \int f_{\kappa} w^{\kappa} \psi(R)^2.
\end{equation}
By adding \eqref{mainest_eqn15} and \eqref{mainest_eqn16} and adding $\int 2\alpha^2 R^{-1} |w|^2 \psi(R) \psi'(R)$ to both sides, we obtain 
\begin{align*}
	&\int ((\alpha |Dw|^2 + |D_y w|^2) \psi(R)^2 + 2\alpha^2 R^{-1} |w|^2 \psi(R) \psi'(R))
	\\&\hspace{10mm} = -2 \int \left( \tfrac{1}{2} r^2 |Dw|^2 - r D_r w^{\kappa} (r D_r w^{\kappa} - \alpha w^{\kappa}) - \alpha^2 |w|^2 \right) 
		R^{-1} \psi(R) \psi'(R) \nonumber 
	\\&\hspace{20mm} + 2 \int (y \cdot D_y w^{\kappa})(r D_r w^{\kappa} - \alpha w^{\kappa}) R^{-1} \psi(R) \psi'(R) 
	+ \int f_{\kappa} (r D_r w^{\kappa} - \alpha w^{\kappa}) \psi(R)^2. 
\end{align*}
By Cauchy's inequality, 
\begin{align} \label{mainest_eqn17}
	&\int \left( \left( \alpha |Dw|^2 + \tfrac{1}{2} |D_y w|^2 \right) \psi(R)^2 + 2\alpha^2 R^{-1} |w|^2 \psi(R) \psi'(R) \right)
	 \\&\hspace{10mm} \leq -2 \int \left( \tfrac{1}{2} r^2 |Dw|^2 - r D_r w^{\kappa} (r D_r w^{\kappa} - \alpha w^{\kappa}) 
		- \alpha^2 |w|^2 \right) R^{-1} \psi(R) \psi'(R) 
	\nonumber \\&\hspace{20mm}  + 2 \int |r D_r w - \alpha w|^2 (\psi(R)^2 + \psi'(R)^2) + \frac{1}{8} \int |f|^2 \psi(R)^2. \nonumber
\end{align}

Now recall that $\varphi(X) = \{\pm \varphi_1(X)\}$ for $\varphi_1(X) \in \mathbb{R}^m$ and that $w$ and $v$ satisfy \eqref{v graph rep} as in Lemma~\ref{lemma2_6}.  Note that, since $\varphi$ and $w$ are symmetric, $v(X,-\varphi_1(X)) = -v(X,\varphi_1(X))$ for each $X \in U$.  Let $v_1(X) = v(X,\varphi_1(X))$ for each $X \in U$.  Using the estimates of Lemma~\ref{lemma2_6} (with $(1+\gamma)/2$ in place of $\gamma$), 
\begin{align} \label{mainest_eqn18}
	&-2 \int_U \left( \tfrac{1}{2} r^2 |Dw|^2 - r D_r w^{\kappa} (r D_r w^{\kappa} - \alpha w^{\kappa}) 
		- \alpha^2 |w|^2 \right) R^{-1} \psi(R) \psi'(R) 
	 \\&\hspace{6mm} + 2 \int_U |r D_r w - \alpha w|^2 (\psi(R)^2 + \psi'(R)^2) 
	 \leq -4 \int_U \left( \tfrac{1}{2} r^2 |D\varphi_1|^2 - \alpha^2 |\varphi_1|^2 \right) R^{-1} \psi(R) \psi'(R) \nonumber
	 \\&\hspace{12mm} - 4 \int_U \left( r^2 D\varphi_1^{\kappa} \cdot Dv_1^{\kappa} - r D_r \varphi_1^{\kappa} (r D_r v_1^{\kappa} - \alpha v_1^{\kappa}) 
	 	- 2\alpha^2 \varphi_1^{\kappa} v_1^{\kappa} \right) R^{-1} \psi(R) \psi'(R) \nonumber 
	 \\&\hspace{18mm} + C \int_{B_1(0)} \mathcal{G}(w,\varphi) + C \int_{B_{(3+\gamma)/4}(0)} r^4 |f|^2 \nonumber
\end{align}
for some constant $C = C(n,m,\alpha,\varphi^{(0)},\gamma) \in (0,\infty)$.  Let $(re^{i\theta},y)$ denote cylindrical coordinates on $\mathbb{R}^n$, where $r \geq 0$, $\theta \in \mathbb{R}$, and $y \in \mathbb{R}^{n-2}$.  If $\alpha$ is a positive integer, then $\varphi(X) = \{ \pm \op{Re}(c r^{\alpha} e^{i\alpha\theta}) \}$ where $-\op{Re}(c r^{\alpha} e^{i\alpha\theta})$ and $+\op{Re}(c r^{\alpha} e^{i\alpha\theta})$ are both single-valued harmonic functions on $U$.  If instead $\alpha = k_0/2$ for some integer $k_0 \geq 3$, $\varphi(re^{i\theta},y) = \{ \pm \op{Re}(c r^{\alpha} e^{i\alpha\theta}) \}$ where $\op{Re}(c r^{\alpha} e^{i\alpha\theta})$ is $4\pi$-periodic as a function of $\theta$.  Thus by \eqref{lemma2_6 concl1} we can apply integration by parts with respect to $\theta$ and use the homogeneity of $\varphi$ and $D_{\theta\theta} \varphi + \alpha^2 \varphi = 0$ to obtain 
\begin{align*}
	\int_U \left( \tfrac{1}{2} r^2 |D\varphi_1|^2 - \alpha^2 |\varphi_1|^2 \right) R^{-1} \psi(R) \psi'(R) 
	&= \frac{1}{2} \int_U \left( |D_{\theta} \varphi_1|^2 - \alpha^2 |\varphi_1|^2 \right) R^{-1} \psi(R) \psi'(R) 
	\\&= \frac{1}{2} \int_U \left( -\varphi^{\kappa} D_{\theta\theta} \varphi_1^{\kappa} - \alpha^2 |\varphi_1|^2 \right) R^{-1} \psi(R) \psi'(R)  = 0 
\end{align*}
and 
\begin{align*}
	&\int_U \left( r^2 D\varphi_1^{\kappa} \cdot Dv_1^{\kappa} - r D_r \varphi_1^{\kappa} (r D_r v_1^{\kappa} - \alpha v_1^{\kappa}) 
	 	- 2 \alpha^2 \varphi_1^{\kappa} v_1^{\kappa} \right) R^{-1} \psi(R) \psi'(R) 
	\\&\hspace{15mm} = \int_U \left( D_{\theta} \varphi_1^{\kappa} D_{\theta} v_1^{\kappa} 
		- \alpha^2 \varphi_1^{\kappa} v_1^{\kappa} \right) R^{-1} \psi(R) \psi'(R) 
	\\&\hspace{15mm} = \int_U \left( -D_{\theta\theta} \varphi_1^{\kappa} v_1^{\kappa} - \alpha^2 \varphi_1^{\kappa} v_1^{\kappa} \right) 
		R^{-1} \psi(R) \psi'(R)  = 0. 
\end{align*}
Hence \eqref{mainest_eqn18} gives us 
\begin{align} \label{mainest_eqn19}
	&-2 \int_U \left( \tfrac{1}{2} r^2 |Dw|^2 - r D_r w^{\kappa} (r D_r w^{\kappa} - \alpha w^{\kappa}) 
		- \alpha^2 |w|^2 \right) R^{-1} \psi(R) \psi'(R) 
	 \\&\hspace{15mm} + 2 \int_U |r D_r w - \alpha w|^2 (\psi(R)^2 + \psi'(R)^2) 
	 \leq C \int_{B_1(0)} \mathcal{G}(w,\varphi)^2 + C \int_{B_{(3+\gamma)/4}(0)} r^4 |f|^2,  \nonumber
\end{align}
where $C = C(n,m,\alpha,\varphi^{(0)},\gamma) \in (0,\infty)$ is a constant.  Again using Lemma~\ref{lemma2_6}, 
\begin{align} \label{mainest_eqn20}
	&-2 \int_{B_{(1+\gamma)/2}(0) \setminus U} \left( \tfrac{1}{2} r^2 |Dw|^2 - r D_r w^{\kappa} (r D_r w^{\kappa} - \alpha w^{\kappa}) 
		- \alpha^2 |w|^2 \right) R^{-1} \psi(R) \psi'(R) 
	 \\&\hspace{15mm} + 2 \int_{B_{(1+\gamma)/2}(0) \setminus U} |r D_r w - \alpha w|^2 (\psi(R)^2 + \psi'(R)^2) \nonumber 
	 \\&\hspace{30mm} \leq C \int_{B_1(0)} \mathcal{G}(w,\varphi)^2 + C \int_{B_{(3+\gamma)/4}(0)} r^4 |f|^2 \nonumber
\end{align}
for some constant $C = C(n,m,\alpha,\varphi^{(0)},\gamma) \in (0,\infty)$.  Using \eqref{mainest_eqn19} and \eqref{mainest_eqn20} to bound the right-hand side of \eqref{mainest_eqn17},  
\begin{align} \label{mainest_eqn21}
	&\int \left( \left( \alpha |Dw|^2 + \tfrac{1}{2} |D_y w|^2 \right) \psi(R)^2 + 2\alpha^2 R^{-1} |w|^2 \psi(R) \psi'(R) \right)
	  \\&\hspace{15mm} \leq C \int_{B_1(0)} \mathcal{G}(w,\varphi)^2 + C \int_{B_{(3+\gamma)/4}(0)} |f|^2 \nonumber 
\end{align}
for some constant $C = C(n,m,\alpha,\varphi^{(0)},\gamma) \in (0,\infty)$.  Combining \eqref{mainest_eqn14} and \eqref{mainest_eqn21}, and noting that since $n \geq 2$ and $\alpha \geq 3/2$ we have $4-n-2\alpha < 0$, 
\begin{align} \label{mainest_eqn22}
	&\int_{B_{\gamma}(0)} R^{2-n} \left| \frac{\partial (w/R^{\alpha})}{\partial R} \right|^2 + \int_{B_{\gamma}(0)} |D_y w|^2 \\ 
	&\hspace{15mm} \leq C \int_{B_1(0)} \mathcal{G}(w,\varphi)^2 
		+ C \int_{B_{(3+\gamma)/4}(0)} R^{4-n-2\alpha} |f|^2 \nonumber
\end{align}
for $C = C(n,m,\alpha,\varphi^{(0)},\gamma) \in (0,\infty)$.  

By \eqref{mss3 f est}, \eqref{minimal avg schauder} and \eqref{minimal sym decay}, $|f| \leq C(n,m) \,\|Du\|_{C^0(B_1(0))}^2 \,(|Dw| + R |D^2 w|)$ and thus 
\begin{equation} \label{mainest_eqn23}
	\int_{B_{(3+\gamma)/4}(0)} R^{4-n-2\alpha} |f|^2 
	\leq C \|Du\|_{C^0(B_1(0))}^4 \int_{B_{(3+\gamma)/4}(0)} R^{4-n-2\alpha} (|Dw|^2 + R^2 |D^2 w|^2)
\end{equation}
for some constant $C = C(n,m) \in (0,\infty)$.  By \eqref{minimal sym schauder} and \eqref{L2_decay}, 
\begin{align*}
	&\int_{B_{(3+\gamma)/4}(0)} R^{4-n-2\alpha} (|Dw|^2 + R^2 |D^2 w|^2)
	\\&\hspace{15mm} \leq C \sum_{k=0}^{\infty} 2^{-(4-n-2\alpha) k} \int_{B_{2^{-k} (3+\gamma)/4}(0)} (|Dw|^2 + 2^{-2k} |D^2 w|^2) .
	\\&\hspace{15mm} \leq C \sum_{k=0}^{\infty} 2^{-(2-n-2\alpha) k} \int_{B_{2^{-k} (7+\gamma)/8}(0)} |w|^2
	\leq C \sum_{k=0}^{\infty} 2^{-2k} \int_{B_1(0)} |w|^2 
	\leq C ,
\end{align*}
where $C = C(n,m,\alpha,\varphi^{(0)},\gamma) \in (0,\infty)$ and the last step uses $\|w\|_{L^2(B_1(0))} \leq \|\varphi^{(0)}\|_{L^2(B_1(0))} + 1$ (by \eqref{F defn rmk eqn1}).  Hence \eqref{mainest_eqn23} gives us 
\begin{equation} \label{mainest_eqn24}
	\int_{B_{(3+\gamma)/4}(0)} R^{4-n-2\alpha} |f|^2 \leq C \|Du\|_{C^0(B_1(0))}^4 
\end{equation}
for some constant $C = C(n,m,\alpha,\gamma) \in (0,\infty)$.  Using \eqref{mainest_eqn24} to bound the last term on the right-hand side of \eqref{mainest_eqn22} yields the conclusion of Theorem~\ref{mainestthm}. 
\end{proof}

\begin{corollary} \label{noncon_at_origin_cor} 
Let $\alpha$ and $\varphi^{(0)}$ be as in Definition~\ref{varphi0 defn}.  For each $\gamma,\sigma \in (0,1)$ there exists $\varepsilon_0 = \varepsilon_0(n,m,\alpha,\varphi^{(0)},\gamma) \in (0,1)$ such that if $\varphi \in \Phi_{\varepsilon_0}(\varphi^{(0)})$ and $(u,\Lambda) \in \mathcal{F}_{\varepsilon_0}(\varphi^{(0)})$ such that $0 \in \mathcal{K}_M$ and $\mathcal{N}_M(0) \geq \alpha$, where $M = {\rm graph}\,u$, then 
\begin{equation} \label{ncao_concl}
	\int_{B_{\gamma}(0)} R^{-n+\sigma-2\alpha} \mathcal{G}\left(\frac{u_s}{\Lambda}, \varphi\right)^2
	\leq C \int_{B_1(0)} \mathcal{G}\left(\frac{u_s}{\Lambda}, \varphi\right)^2 + C \|Du\|_{C^0(B_1(0))}^2  
\end{equation}
for some constant $C = C(n,m,\alpha,\varphi^{(0)},\gamma,\sigma) \in (0,\infty)$. 
\end{corollary}

\begin{proof} 
\textit{Part 1.}  First we will consider the special case where $u(0) = \{0,0\}$ and $Du(0) = \{0,0\}$.  Set $w = u_s/\Lambda$.  Recall that for any single-valued vector fields $\zeta = (\zeta^1,\ldots,\zeta^n) \in C_c^{0,1}(\mathbb{R}^n)$ that 
\begin{equation} \label{ncao_eqn2}
	\int D_i \zeta^i = 0.  
\end{equation}
Let $\psi : [0,\infty) \rightarrow \mathbb{R}$ be a non-increasing smooth function such that $\psi(t) = 1$ if $t \in [0,\gamma]$ and $\psi(t) = 0$ if $t \geq (1+\gamma)/2$.  For each $\delta \in (0,1)$ let $\eta_{\delta} : [0,\infty) \rightarrow \mathbb{R}$ be a smooth function such that $\eta_{\delta}(t) = 0$ if $t \in [0,\delta/2]$, $\eta_{\delta}(t) = 1$ if $t \geq \delta$, and $0 \leq \eta'_{\delta} \leq 3/\delta$.  Letting $\zeta^i = \psi(R)^2 \eta_{\delta}(R) R^{-n+\sigma-2\alpha} \mathcal{G}(w,\varphi)^2 X_i$ in \eqref{ncao_eqn2} gives us 
\begin{align} \label{ncao_eqn3}
	&\sigma \int \psi(R)^2 \eta_{\delta}(R) R^{-n+\sigma-2\alpha} \mathcal{G}(w,\varphi)^2  
	= - \int \psi(R)^2 \eta_{\delta}(R) R^{1-n+\sigma} D_R (R^{-2\alpha} \mathcal{G}(w,\varphi)^2) 
		\\&\hspace{15mm} - \int 2 \psi(R) \psi'(R) \eta_{\delta}(R) R^{1-n+\sigma-2\alpha} \mathcal{G}(w,\varphi)^2 
		- \int \psi(R)^2 \eta'_{\delta}(R) R^{1-n+\sigma-2\alpha} \mathcal{G}(w,\varphi)^2 . \nonumber
\end{align}
Observe that 
\begin{equation*}
	|D_R (R^{-2\alpha} \mathcal{G}(w,\varphi)^2)| 
	= |D_R (\mathcal{G}(w/R^{\alpha}, \varphi/R^{\alpha})^2)| 
	= 4 \,\mathcal{G}(w/R^{\alpha}, \varphi/R^{\alpha}) \,|D_R (w/R^{\alpha})| 
\end{equation*}
a.e.~in $B_1(0)$.  Thus using the Cauchy inequality in \eqref{ncao_eqn3} and using Theorem~ \ref{mainestthm} (after dropping the last term on the right hand side) 
\begin{align} \label{ncao_eqn4}
	&\int \psi(R)^2 \eta_{\delta}(R) R^{-n+\sigma-2\alpha} \mathcal{G}(w,\varphi)^2 
	\\&\hspace{15mm} \leq \frac{36}{\sigma^2} \int \psi(R)^2 R^{2-n+\sigma} \left| \frac{\partial (w/R^{\alpha})}{\partial R} \right|^2  
		+ \frac{9}{\sigma^2} \int_{B_1(0)} \psi'(R)^2 R^{2-n+\sigma-2\alpha} \mathcal{G}(w,\varphi)^2 \nonumber 
	\\&\hspace{15mm} \leq C \int_{B_1(0)} \mathcal{G}(w,\varphi)^2 + C \|Du\|_{C^0(B_1(0))}^4 \nonumber 
\end{align}
for some constant $C = C(n,m,\alpha,\varphi^{(0)},\gamma,\sigma) \in (0,\infty)$.  Letting $\delta \downarrow 0$ in \eqref{ncao_eqn4} using the monotone
convergence theorem gives us \eqref{ncao_concl}.

\noindent \textit{Part 2.}  Now we consider the general case.  Let $M = {\rm graph}\, u$.  By translating assume that $u(0) = \{0,0\}$.  Let $Q$ be a rotation such that $Q(T_0 M) = \mathbb{R}^n \times \{0\}$ and $\|Q - I\| \leq C(n,m) \,|Du(0)|$ and, assuming $\varepsilon_0$ is sufficiently small, let $\widehat{u} \in C^{1,1/2}(B_{(1+\gamma)/2}(0), \mathcal{A}_2(\mathbb{R}^m))$ be the two-valued function such that $Q(M) \cap B_{(1+\gamma)/2}(0) \times \mathbb{R}^m = {\rm graph}\,\widehat{u}$ (as in Definition~\ref{hatu defn}).

Observe that $\|\widehat{u}\|_{C^{1,1/2}(B_{(3+\gamma)/4}(0))} \leq C(n,m,\gamma) \,\varepsilon_0$.  Moreover, by \eqref{F defn rmk eqn1} and Lemma~\ref{rotate lemma} (with $u, \widehat{u}_P$ in place of $u,\widehat{u}$) provided $\varepsilon_0$ is sufficiently small we have that 
\begin{equation} \label{ncao_eqn06}
	\int_{B_{(3+\gamma)/4}(0)} \mathcal{G}\left( \frac{\widehat{u}_s}{\Lambda}, \varphi^{(0)} \right)^2 \leq C \varepsilon_0^2 
\end{equation}
for some constant $C = C(n,m,\varphi^{(0)},\gamma) \in (0,\infty)$.  Hence we may use Part 1 and $\|D\widehat{u}\|_{C^0(B_{(1+\gamma)/2}(0))} \leq C(n,m) \|Du\|_{C^0(B_1(0))}$ to obtain 
\begin{equation} \label{ncao_eqn6} 
	\int_{B_{\gamma}(0)} R^{-n+\sigma-2\alpha} \mathcal{G}\left( \frac{\widehat{u}_s}{\Lambda}, \varphi \right)^2 
	\leq C \int_{B_{(1+\gamma)/2}(0)} \mathcal{G}\left( \frac{\widehat{u}_s}{\Lambda}, \varphi \right)^2 + C \|Du\|_{C^0(B_1(0))}^4  
\end{equation}
where $C = C(n,m,\alpha,\varphi^{(0)},\gamma,\sigma) \in (0,\infty)$.  By Lemma~\ref{rotate lemma} (with $\widehat{u}_P, u$ in place of $u,\widehat{u}$), \eqref{ncao_eqn06} and \eqref{L2_decay}, provided $\varepsilon_0$ is sufficiently small we have for each $\rho \in (0,(1+\gamma)/2]$ that 
\begin{align} \label{ncao_eqn7}
	\int_{B_{\rho}(0)} \mathcal{G}\left( \frac{u_s}{\Lambda}, \frac{\widehat{u}_s}{\Lambda} \right)^2
	&\leq \frac{C |Du(0)|^2}{\Lambda^2} \int_{B_{\beta \rho}(0)} |\widehat{u}_{P,s}|^2
	\\&\leq \frac{C \rho^{n+2\alpha} |Du(0)|^2}{\Lambda^2} \int_{B_{(3+\gamma)/4}(0)} |\widehat{u}_{P,s}|^2
	\leq C \rho^{n+2\alpha} |Du(0)|^2 , \nonumber 
\end{align}
where $\beta = (3+\gamma)/(2+2\gamma) > 1$ and $C = C(n,m,\alpha,\varphi^{(0)},\gamma) \in (0,\infty)$.  Hence by \eqref{ncao_eqn6} and \eqref{ncao_eqn7} with $\rho = (1+\gamma)/2$, 
\begin{equation} \label{ncao_eqn8}
	\int_{B_{\gamma}(0)} R^{-n+\sigma-2\alpha} \mathcal{G}\left( \frac{\widehat{u}_s}{\Lambda}, \varphi \right)^2 
	\leq C \int_{B_1(0)} \mathcal{G}\left( \frac{u_s}{\Lambda}, \varphi \right)^2 + C \|Du\|_{C^0(B_1(0))}^2  
\end{equation}
for some constant $C = C(n,m,\alpha,\varphi^{(0)},\gamma,\sigma) \in (0,\infty)$.  By \eqref{ncao_eqn7} with $\rho = 2^{-k} \gamma$ for $k = 0,1,2,\ldots$ 
\begin{align} \label{ncao_eqn9}
	\int_{B_{\gamma}(0)} R^{-n+\sigma-2\alpha} \mathcal{G}\left( \frac{u_s}{\Lambda}, \frac{\widehat{u}_s}{\Lambda} \right)^2 
	&\leq C \sum_{k=0}^{\infty} 2^{(n-\sigma+2\alpha) k} \int_{B_{2^{-k} \gamma}(0)} 
		\mathcal{G}\left( \frac{u_s}{\Lambda}, \frac{\widehat{u}_s}{\Lambda} \right)^2
	\\&\leq C \sum_{k=0}^{\infty} 2^{-\sigma k} |Du(0)|^2 = C |Du(0)|^2  \nonumber 
\end{align}
where $C = C(n,m,\alpha,\varphi^{(0)},\gamma,\sigma) \in (0,\infty)$.  Combining \eqref{ncao_eqn8} and \eqref{ncao_eqn9} using the triangle inequality yields \eqref{ncao_concl}.
\end{proof}

Now let $\varepsilon > 0$, $\varphi^{(0)}$ be as in Definition~\ref{varphi0 defn}, $\varphi \in \Phi_{\varepsilon}(\varphi^{(0)})$ and $(u,\Lambda) \in \mathcal{F}_{\varepsilon}(\varphi^{(0)})$.  Recall that we use coordinates $X = (x,y)$ on $\mathbb{R}^n$, where $x = (x_1,x_2)$ and $y = (x_3,\ldots,x_n)$.  Consider a point $Z = (\xi,\zeta) \in \mathcal{K}_u \cap B_{1/2}(0)$.  Observe that for each $X = (x,y) \in B_1(0) \setminus \{0\} \times \mathbb{R}^{n-2}$, $\varphi = \{\pm \varphi_1\}$ on $B^2_{|x|/2}(x) \times \mathbb{R}^{n-2}$ for some single-valued harmonic function $\varphi_1$.  By Taylor's theorem applied to $\varphi_1$ and the homogeneity of $\varphi$, if $|\xi| \leq |x|/2$ then
\begin{equation*} 
	\varphi(X-Z) = \{ \pm (\varphi_1(X) - D_x \varphi_1(X) \cdot \xi + \mathcal{R}(x,\xi)) \}
\end{equation*}
where $|\mathcal{R}(x,\xi)| \leq C |x|^{\alpha-2} |\xi|^2$ with $C = C(\alpha,\sup_{S^1} |D^2 \varphi|) \in (0,\infty)$.  Hence if $|\xi| \leq |x|/2$ then 
\begin{equation} \label{translate eqn1}
	\mathcal{G}(\varphi(X-Z), \varphi(X) - D_x \varphi(X) \cdot \xi) \leq C |x|^{\alpha-2} |\xi|^2, 
\end{equation}
where $C = C(\varphi^{(0)}) \in (0,\infty)$ and we define   
\begin{equation*}
	\varphi(X) - D_x \varphi(X) \cdot \xi = \{ \pm (\varphi_1(X) - D_x \varphi_1(X) \cdot \xi) \} . 
\end{equation*}
Notice that if $|\varphi(X)| \geq |D_x \varphi(X) \cdot \xi|$ then 
\begin{equation*}
	\mathcal{G}(\varphi(X), \varphi(X) - D_x \varphi(X) \cdot \xi) = |D_x \varphi(X) \cdot \xi|
\end{equation*}
and thus by the triangle inequality and \eqref{translate eqn1} if $|\xi| \leq |x|/2$ and $|\varphi(X)| \geq |D_x \varphi(X) \cdot \xi|$ then 
\begin{equation} \label{translate eqn3}
	\mathcal{G}(u(X),\varphi(X-Z)) \geq |D_x \varphi(X) \cdot \xi| - \mathcal{G}(u(X), \varphi(X)) - C |x|^{\alpha-2} |\xi|^2
\end{equation}
for some constant $C = C(\varphi^{(0)}) \in (0,\infty)$.  By the triangle inequality and the fundamental theorem of calculus, for all $X = (x,y) \in B_1(0)$, 
\begin{equation} \label{translate eqn4}
	|\mathcal{G}(u(X),\varphi(X-Z)) - \mathcal{G}(u(X),\varphi(X))| \leq \mathcal{G}(\varphi(X-Z),\varphi(X)) \leq C |\xi|
\end{equation}
for some constant $C = C(\varphi^{(0)}) \in (0,\infty)$.  It follows using \eqref{minimal sym schauder} that there exists $\tau(\varepsilon) \in (0,1)$ with $\tau(\varepsilon) \rightarrow 0$ as $\varepsilon \downarrow 0$ such that $|\xi| < \tau(\varepsilon)$ for every $Z = (\xi,\zeta) \in \mathcal{K}_u \cap B_{1/2}(0)$.  Hence by \eqref{translate eqn4} with $\varphi^{(0)}$ in place of $\varphi$, 
\begin{align*}
	&\int_{B_{1/2}(0)} \mathcal{G}\left( \frac{u(X-Z)}{\Lambda}, \varphi^{(0)}(X) \right)^2 dX
		\leq \int_{B_1(0)} \mathcal{G}\left( \frac{u(X)}{\Lambda}, \varphi^{(0)}(X-Z) \right)^2 dX
		\\&\hspace{15mm} \leq 2 \int_{B_1(0)} \mathcal{G}\left( \frac{u(X)}{\Lambda}, \varphi^{(0)}(X) \right)^2 dX + 2 C^2 |\xi|^2 
		\leq 2 \varepsilon^2 + 2 C^2 \tau^2(\varepsilon) 
\end{align*}
and thus $(2 u(Z + X/2), 2^{1-\alpha} \Lambda) \in \mathcal{F}_{C(\varepsilon + \tau(\varepsilon))}(\varphi^{(0)})$ for some constant $C = C(\varphi^{(0)}) \in (0,\infty)$.

\begin{corollary} \label{lemma3_9}
Let $\alpha$ and $\varphi^{(0)}$ be as in Definition~\ref{varphi0 defn}.  There exists $\varepsilon_0 = \varepsilon_0(n,m,\alpha,\varphi^{(0)}) \in (0,1)$ such that if $\varphi \in \Phi_{\varepsilon_0}(\varphi^{(0)})$, $(u,\Lambda) \in \mathcal{F}_{\varepsilon_0}(\varphi^{(0)})$ and $Z \in \mathcal{K}_u \cap B_{1/2}(0)$ such that $\mathcal{N}_M(Z,u_1(Z)) \geq \alpha$, where $M = {\rm graph}\,u$ and $u(Z) = \{u_1(Z), u_1(Z)\}$ for some $u_1(Z) \in \mathbb{R}^m$, then 
\begin{align*}
	&(a) \hspace{5mm} \op{dist}(Z, \{0\} \times \mathbb{R}^{n-2})^2 
		\leq C \int_{B_1(0)} \mathcal{G}\left( \frac{u_s}{\Lambda}, \varphi \right)^2 + C \|Du\|_{C^0(B_1(0))}^2 
	\\&(b) \hspace{5mm} \int_{B_1(0)} \mathcal{G}\left( \frac{u_s(X)}{\Lambda}, \varphi(X-Z) \right)^2 \,dX 
		\leq C \int_{B_1(0)} \mathcal{G}\left( \frac{u_s}{\Lambda}, \varphi \right)^2 + C \|Du\|_{C^0(B_1(0))}^2 
\end{align*}
for some constant $C = C(n,m,\varphi^{(0)},\alpha) \in (0,\infty)$. 
\end{corollary}

\begin{proof} 
We claim that there is a constant $\delta_1 = \delta_1(\varphi^{(0)}) > 0$ and, given $\rho \in (0,1/4)$, there is a constant $\varepsilon_0 = \varepsilon_0(\rho,\varphi^{(0)}) > 0$ such that for $\varphi \in \Phi_{\varepsilon_0}(\varphi^{(0)})$, $(u,\Lambda) \in \mathcal{F}_{\varepsilon_0}(\varphi^{(0)})$ and $Z = (\xi,\zeta) \in \mathcal{K}_u \cap B_{1/2}(0)$, 
\begin{align} \label{lemma3_9_eqn1}
	\mathcal{L}^n \{ X \in B_{\rho}(Z) : \delta_1 |\xi| |x|^{\alpha-1} \leq |D_x \varphi(X) \cdot \xi| \leq |\varphi(X)| \} \geq \delta_1 \rho^n. 
\end{align}
Suppose \eqref{lemma3_9_eqn1} is false; then for every $\delta_1 > 0$ there is a $\rho > 0$ such that with $\varepsilon_j = 1/j$, there exists $(u_j,\Lambda_j) \in \mathcal{F}_{\varepsilon_j}(\varphi^{(0)})$, $\varphi_j \in \Phi_{\varepsilon_j}(\varphi^{(0)})$ and $Z_j = (\xi_j,\zeta_j) \in \mathcal{K}_{u_j}$ such that 
\begin{equation*}
	\mathcal{L}^n \{ X \in B_{\rho}(Z_j) : \delta_1 |\xi_j| |x|^{\alpha-1} \leq |D_x \varphi_j(X) \cdot \xi_j| \leq |\varphi_j(X)| \} < \delta_1 \rho^n. 
\end{equation*}
After passing to a subsequence, $\varphi_j \rightarrow \varphi^{(0)}$ in $C^1(B_1(0))$, $\xi_j \rightarrow 0$, $\zeta_j \rightarrow \zeta$ for some $\zeta \in \overline{B^{n-2}_1(0)}$ and $\xi_j/|\xi| \rightarrow a$ for some $a \in S^1$ such that 
\begin{equation*} 
	\mathcal{L}^n \{ X \in B_{\rho}(0,\zeta) : \delta_1 |x|^{\alpha-1} < |D_x \varphi^{(0)}(X) \cdot a| 
		\text{ and } |\varphi^{(0)}(X)| > 0 \} \leq \delta_1 \rho^n.
\end{equation*}
Notice that $\varphi^{(0)}(X) \neq \{0,0\}$ for $\mathcal{L}^n$-a.e.~$X \in B_1(0)$ so we in fact have that 
\begin{equation} \label{lemma3_9_eqn2}
	\mathcal{L}^n \{ X \in B_{\rho}(0,\zeta) : \delta_1 |x|^{\alpha-1} < |D_x \varphi^{(0)}(X) \cdot a| \} \leq \delta_1 \rho^n.
\end{equation}
Thus we have shown that for every $\delta_1 > 0$ there is $\rho > 0$, $Z \in \{0\} \times \mathbb{R}^{n-2}$ and $a \in S^1$ such that \eqref{lemma3_9_eqn2} holds true.  By translating and rescaling, we may suppose that $Z = 0$ and $\rho = 1$.  Thus for $\delta_1 = 1/j$, there is an $a_j \in S^1$ such that 
\begin{equation*}
	\mathcal{L}^n \{ X \in B_1(0) : (1/j) |x|^{\alpha-1} < |D_x \varphi^{(0)}(X) \cdot a_j| \} < 1/j.
\end{equation*}
After passing to a subsequence, $a_j \rightarrow a$ with $a \in S^1$ such that 
\begin{equation*}
	D_x \varphi^{(0)}(X) \cdot a = 0 \text{ $\mathcal{L}^n$-a.e.~in } B_1(0),
\end{equation*}
which is obviously false since $\varphi^{(0)}(X) = \op{Re}(c^{(0)} (x_1+ix_2)^{\alpha})$.

Let $Z = (\xi,\zeta) \in \mathcal{K}_u \cap B_{1/2}(0)$ such that $\mathcal{N}_M(Z) \geq \alpha$.  Let $\rho \in (0,1/4)$ be chosen.  With $\delta_1$ as in the claim, choose $\kappa = \kappa(n) \in (0,1)$ such that $\mathcal{L}^n(B^2_{\kappa \delta_1^{1/2} \rho}(0) \times B^{n-2}_{\rho}(0)) < \delta_1 \rho^n/2$.  Using \eqref{lemma3_9_eqn1} and \eqref{translate eqn3} we can find some set $S \subset B_{\rho}(Z)$ with $\mathcal{L}^n(S) \geq \delta_1 \rho^n/2$ such that 
\begin{align} \label{lemma3_9_eqn5}
	&\frac{1}{2} \delta_1^{\alpha+1} \kappa^{2\alpha-2} \rho^{n+2\alpha-2} |\xi|^2 \leq \int_{S} |x|^{2\alpha-2} |\xi|^2 
		\leq \delta_1^2 \int_{S} |D_x \varphi(X) \cdot \xi|^2 
	\\&\hspace{10mm} \leq 3 \int_{B_{\rho}(Z)} \mathcal{G}\left( \frac{u_s(X)}{\Lambda}, \varphi(X-Z) \right)^2 \,dX 
		+ 3 \int_{B_{\rho}(Z)} \mathcal{G}\left( \frac{u_s(X)}{\Lambda}, \varphi(X) \right)^2 \,dX \nonumber 
		\\&\hspace{20mm}  + 3 C \int_{B_{\rho}(Z) \cap \{|x| \geq |\xi|/2\}} |x|^{2\alpha-4} |\xi|^4 
		+ \int_{B_{\rho}(Z) \cap \{|x| \leq |\xi|/2\}} |D_x \varphi(X)|^2 |\xi|^2 \nonumber 
\end{align}
for all $\rho \in (0,1/4)$ and some constant $C = C(\varphi^{(0)}) \in (0,\infty)$.  We need to bound the terms on the right-hand side of \eqref{lemma3_9_eqn5}.  For the first term on the right-hand side of \eqref{lemma3_9_eqn5}, by Corollary~\ref{noncon_at_origin_cor} with $(2 u(Z + X/2), 2^{1-\alpha} \Lambda)$ in place of $(u,\Lambda)$ and $\sigma = 1/2$ and by \eqref{translate eqn4}, 
\begin{align} \label{lemma3_9_eqn6}
	&\rho^{-n-2\alpha+\sigma} \int_{B_{\rho}(Z)} \mathcal{G}\left( \frac{u_s(X)}{\Lambda}, \varphi(X-Z) \right)^2 \,dX \\
	&\hspace{15mm} \leq C \int_{B_1(0)} \mathcal{G}\left( \frac{u_s(X)}{\Lambda}, \varphi(X-Z) \right)^2 \,dX + C \|Du\|_{C^0(B_1(0))}^2 
		\nonumber \\
	&\hspace{15mm} \leq C \int_{B_1(0)} \mathcal{G}\left( \frac{u_s(X)}{\Lambda}, \varphi(X) \right)^2 \,dX + C|\xi|^2 + C \|Du\|_{C^0(B_1(0))}^2 \nonumber
\end{align}
for some constant $C = C(n,m,\varphi^{(0)},\alpha) \in (0,\infty)$.  For the last two terms on the right-hand side of \eqref{lemma3_9_eqn5}, by direct computation
\begin{align} 
	\label{lemma3_9_eqn7} &\int_{B_{\rho}(Z) \cap \{|x| \geq |\xi|/2\}} |x|^{2\alpha-4} |\xi|^4 
		\leq C (|\xi|^{n+2\alpha-4} + \rho^{n+2\alpha-4}) |\xi|^4, \\
	\label{lemma3_9_eqn8} &\int_{B_{\rho}(Z) \cap \{|x| \leq |\xi|/2\}} |D_x \varphi(X)|^2 |\xi|^2 
		\leq C \rho^{n-2} \int_{B^2_{|\xi|/2}(0)} |x|^{2\alpha-2} |\xi|^2 \leq C \rho^{n-2} |\xi|^{2\alpha}
\end{align}
where $C = C(n,\alpha,\varphi^{(0)}) \in (0,\infty)$.  Therefore, by \eqref{lemma3_9_eqn5}, \eqref{lemma3_9_eqn6}, \eqref{lemma3_9_eqn7} and \eqref{lemma3_9_eqn8}, 
\begin{align} \label{lemma3_9_eqn9}
	\rho^{n+2\alpha-2} |\xi|^2 
	\leq {}& C \int_{B_1(0)} \mathcal{G}\left( \frac{u_s}{\Lambda}, \varphi \right)^2 + C \|Du\|_{C^0(B_1(0))}^2 \rho^{n+2\alpha-\sigma} \\
		&\hspace{10mm} + C (\rho^{2-\sigma} + \rho^{-2} |\xi|^2 + \rho^{-n-2\alpha+2} |\xi|^{n+2\alpha-2} + \rho^{-2\alpha} |\xi|^{2\alpha}) \rho^{n+2\alpha-2} |\xi|^2 \nonumber
\end{align}
for some constant $C = C(n,m,\varphi^{(0)},\alpha) \in (0,\infty)$.  Notice that given $\tau > 0$, we may choose $\varepsilon_0 = \varepsilon_0(\varphi^{(0)},\tau) > 0$ small enough to ensure that $|\xi| < \tau$.  Choose $\rho = \rho(n,m,\alpha,\varphi^{(0)}) \in $ and $\tau = \tau(n,m,\alpha,\varphi^{(0)})$ small enough that $C (\rho^{2-\sigma} + \rho^{-2} \tau^2 + \rho^{-n-2\alpha+2} \tau^{n+2\alpha-2} + \rho^{-2\alpha} \tau^{2\alpha}) < 1/2$.  Then whenever $\varepsilon_0 = \varepsilon_0(n,m,\alpha,\varphi^{(0)})$ is sufficiently small, under the hypotheses of the corollary we conclude from \eqref{lemma3_9_eqn9} that  
\begin{equation} \label{lemma3_9_eqn10}
	|\xi|^2 \leq C \int_{B_1(0)} \mathcal{G}\left( \frac{u_s}{\Lambda}, \varphi \right)^2 + C \|Du\|_{C^0(B_1(0))}^2 
\end{equation}
for some constant $C = C(n,m,\varphi^{(0)},\alpha) \in (0,\infty)$.  This is conclusion (a).  Conclusion (b) follows directly from \eqref{translate eqn4} and \eqref{lemma3_9_eqn10}.
\end{proof}

\begin{corollary} \label{thm3_1} 
Let $\alpha$ and $\varphi^{(0)}$ be as in Definition~\ref{varphi0 defn}.  For every $\gamma, \tau, \sigma \in (0,1)$ with $\tau \leq (1-\gamma)/20$ there exists $\varepsilon_0 = \varepsilon_0(n,m,\alpha,\varphi^{(0)},\gamma,\tau) \in (0,1)$ such that if $\varphi \in \Phi_{\varepsilon_0}(\varphi^{(0)})$, $(u,\Lambda) \in \mathcal{F}_{\varepsilon_0}(\varphi^{(0)})$, $v$ is as in Lemma~\ref{lemma2_6} with $\beta = 1/2$ and $Z = (\xi,\zeta) \in \mathcal{K}_u \cap B_{1/2}(0)$ such that $\mathcal{N}_M(Z,u_1(Z)) \geq \alpha$, where $M = {\rm graph}\,u$ and $u(Z) = \{u_1(Z), u_1(Z)\}$ for some $u_1(Z) \in \mathbb{R}^m$, then 

\begin{align*} 
	&\int_{B_{\gamma}(0)} \frac{\mathcal{G}(u_s/\Lambda,\varphi)^2}{|X-Z|^{n-\sigma}} \,dX
	+ \int_{B_{\gamma}(0) \cap \{ |x| > \tau \}} \frac{|v(X,\varphi(X)) + D_x \varphi(X) \cdot \xi|^2}{|X-Z|^{n+2\alpha-\sigma}} \,dX
	\\&\hspace{15mm} \leq C \int_{B_1(0)} \mathcal{G}\left( \frac{u_s}{\Lambda}, \varphi \right)^2 + C \|Du\|_{C^0(B_1(0))}^2, 
\end{align*}
where $C = C(n,m,\alpha,\varphi^{(0)},\gamma,\sigma) \in (0,\infty)$ is a constant (independent of $\tau$) and for each $X = (x,y) \in B_{\gamma}(0)$ we note that $\varphi = \{\pm \varphi\}$ in $B^2_{|x|}(x) \times \mathbb{R}^{n-2}$ for some single-valued harmonic function $\varphi_1$ and we let $v(X,\varphi(X)) + D_x \varphi(X) \cdot \xi = \{ \pm (v(X,\varphi_1(X)) + D_x \varphi_1(X) \cdot \xi) \}$. 
\end{corollary}

\begin{proof}
By \eqref{translate eqn4} and Corollary~\ref{lemma3_9}, 
\begin{align} \label{thm3_1_eqn1}
	&\int_{B_{\gamma}(0)} \frac{\mathcal{G}(u_s(X)/\Lambda,\varphi(X))^2}{|X-Z|^{n-\sigma}} \,dX
	\leq 2 \int_{B_{1/4}(Z)} \frac{\mathcal{G}(u_s(X)/\Lambda,\varphi(X-Z))^2}{|X-Z|^{n-\sigma}} \,dX \\&\hspace{30mm} 
		+ C \int_{B_{1/4}(Z)} \frac{|\xi|^2}{|X-Z|^{n-\sigma}}
		+ \int_{B_{\gamma}(0) \setminus B_{1/4}(Z)} \frac{\mathcal{G}(u_s(X)/\Lambda,\varphi(X))^2}{|X-Z|^{n-\sigma}} \,dX \nonumber 
	\\&\hspace{15mm} \leq 2 \int_{B_{1/4}(Z)} \frac{\mathcal{G}(u_s(X)/\Lambda,\varphi(X-Z))^2}{|X-Z|^{n-\sigma}} 
		+ C \int_{B_1(0)} \mathcal{G}\left( \frac{u_s}{\Lambda}, \varphi \right)^2 + C \|Du\|_{C^0(B_1(0))}^2 \nonumber
\end{align}
where $C = C(n,m,\alpha,\varphi^{(0)},\sigma) \in (0,\infty)$.  To bound the first term on the right hand side of \eqref{thm3_1_eqn1}, observe that by Corollary~\ref{noncon_at_origin_cor}, \eqref{translate eqn4} and Corollary~\ref{lemma3_9}, 
\begin{align} \label{thm3_1_eqn2}
	&\int_{B_{1/4}(Z)} \frac{\mathcal{G}(u_s(X)/\Lambda,\varphi(X-Z))^2}{|X-Z|^{n+2\alpha-\sigma}} \,dX \\
	&\hspace{15mm} \leq C \int_{B_{1/2}(Z)} \mathcal{G}\left( \frac{u_s(X)}{\Lambda}, \varphi(X-Z) \right)^2 \,dX  + C \|Du\|_{C^0(B_1(0))}^2 
		\nonumber \\
	&\hspace{15mm} \leq C \int_{B_1(0)} \mathcal{G}\left( \frac{u_s(X)}{\Lambda}, \varphi(X) \right)^2 \,dX + C \|Du\|_{C^0(B_1(0))}^2. \nonumber
\end{align}  
for some constant $C = C(n,\varphi^{(0)},\alpha,\sigma) \in (0,\infty)$.  Hence by \eqref{thm3_1_eqn1} and \eqref{thm3_1_eqn2},  
\begin{equation*}
	\int_{B_{\gamma}(0)} \frac{\mathcal{G}(u_s(X)/\Lambda,\varphi(X))^2}{|X-Z|^{n-\sigma}} 
	\leq C \int_{B_1(0)} \mathcal{G}\left( \frac{u_s}{\Lambda}, \varphi \right)^2 + C \|Du\|_{C^0(B_1(0))}^2
\end{equation*}
for $C = C(n,\varphi^{(0)},\alpha,\gamma,\sigma) \in (0,\infty)$.  

Recall that for each $X = (x,y) \in B_1(0) \setminus \{0\} \times \mathbb{R}^{n-2}$, $\varphi = \{\pm \varphi_1\}$ on $B^2_{|x|/2}(x) \times \mathbb{R}^{n-2}$ for some single-valued harmonic function $\varphi_1$.  By Taylor's theorem applied to $\varphi_1$ and the homogeneity of $\varphi$, if $|\xi| \leq |x|/2$ then
\begin{equation*}
	\varphi_1(X-Z) = \varphi_1(X) - D_x \varphi_1(X) \cdot \xi + \mathcal{R}(x,\xi) 
\end{equation*}
where $|\mathcal{R}(x,\xi)| \leq C |x|^{\alpha-2} |\xi|^2$ with $C = C(\varphi^{(0)}) \in (0,\infty)$.  Hence provided $\varepsilon_0$ is small enough that $|\xi| \leq \tau/2$, if $X = (x,y) \in B_{\gamma}(0)$ with $|x| \geq \tau$ then 
\begin{equation*} 
	|v(X,\varphi_1(X)) + D_x \varphi_1(X) \cdot \xi| \leq |\varphi_1(X) + v(X,\varphi_1(X)) - \varphi_1(X-Z)| + C |x|^{\alpha-2} |\xi|^2. 
\end{equation*}
Hence 
\begin{align} \label{thm3_1_eqn4}
	&\int_{B_{\gamma}(0) \cap \{ |x| > \tau \}} \frac{|v(X,\varphi(X)) - D_x \varphi(X) \cdot \xi|^2}{|X-Z|^{n+2\alpha-\sigma}}  
	\\&\hspace{10mm} \leq 2 \int_{B_{\gamma}(0) \cap B_{1/4}(Z) \cap \{ |x| > \tau \}} \frac{|\varphi(X) + v(X,\varphi(X)) - \varphi(X-Z)|^2}{
			|X-Z|^{n+2\alpha-\sigma}}  \nonumber
		\\&\hspace{20mm} + C \tau^{2\alpha-4} |\xi|^4 \int_{B_{1/4}(Z) \cap \{|x| > \tau\}} \frac{1}{|X-Z|^{n+2\alpha-\sigma}}  \nonumber 
		\\&\hspace{30mm}+ 2 \cdot 4^{n+2\alpha-\sigma} \int_{(B_{\gamma}(0) \setminus B_{1/4}(Z)) \cap \{ |x| > \tau \}} 
			(|v(X,\varphi(X))|^2 + C |x|^{2\alpha-2} |\xi|^2) , \nonumber
\end{align}
where $\varphi(X) + v(X,\varphi(X)) - \varphi(X-Z) = \{ \pm (\varphi_1(X) + v(X,\varphi_1(X)) - \varphi_1(X-Z)) \}$ for each $X = (x,y) \in B_{\gamma}(0)$ with $|x| \geq \tau$.  Let us bound the terms on the right-hand side of \eqref{thm3_1_eqn4}.  To bound the the first term, recall \eqref{v graph rep} and argue as we did to show \eqref{separation_eqn15} in the proof of Lemma~\ref{separation_lemma} to obtain that, assuming $\varepsilon_0$ is sufficiently small, whenever $r > \tau$ and $r^2 + |y-\zeta|^2 < 1/16$ 
\begin{equation} \label{thm3_1_eqn5}
	\int_{\partial B^2_r(\xi) \times \{y\}} |\varphi(X) + v(X,\varphi(X)) - \varphi(X-Z)|^2 
	\leq C \int_{\partial B^2_r(\xi) \times \{y\}} \mathcal{G}\left( \frac{u_s(X)}{\Lambda}, \varphi(X-Z) \right)^2 
\end{equation}
for some constant $C \in (0,\infty)$.  By multiplying \eqref{thm3_1_eqn5} by $r^{n-1} |X-Z|^{-n-2\alpha+\sigma}$ and integrating with respect to $r$ and $y$, 
\begin{equation*}
	\int_{B_{1/4}(Z) \cap \{ |x| > \tau \}} \frac{|\varphi(X) + v(X,\varphi(X)) - \varphi(X-Z)|^2}{|X-Z|^{n+2\alpha-\sigma}} 
	\leq C \int_{B_{1/4}(Z) \cap \{ |x| > \tau \}} \frac{\mathcal{G}(u(X)/\Lambda,\varphi(X-Z))^2}{|X-Z|^{n+2\alpha-\sigma}}. \nonumber 
\end{equation*}
Thus by \eqref{thm3_1_eqn2}, 
\begin{align} \label{thm3_1_eqn6}
	&\int_{B_{1/4}(Z) \cap \{ |x| > \tau \}} \frac{|\varphi(X) + v(X,\varphi(X)) - \varphi(X-Z)|^2}{|X-Z|^{n+2\alpha-\sigma}} 
	\\&\hspace{15mm} \leq C \int_{B_1(0)} \mathcal{G}\left( \frac{u_s}{\Lambda}, \varphi \right)^2 + C \|Du\|_{C^0(B_1(0))}^2.  \nonumber 
\end{align}
To bound the second term on the right-hand side of \eqref{thm3_1_eqn4}, observe that provided $\varepsilon_0$ is small enough that $|\xi| \leq \min\{\tau/4,\tau^2\}$, 
\begin{align} \label{thm3_1_eqn7}
	&\tau^{2\alpha-4} |\xi|^4 \int_{B_{1/4}(Z) \cap \{|x| > \tau\}} \frac{1}{|X-Z|^{n+2\alpha-\sigma}}
	\leq \tau^{2\alpha-4} |\xi|^4 \int_{B_{1/4}(Z) \setminus B_{\tau/2}(Z)} \frac{1}{|X-Z|^{n+2\alpha-\sigma}}
	\\&\hspace{15mm} \leq C \tau^{\sigma-4} |\xi|^4 \leq C |\xi|^2 \leq C \int_{B_1(0)} \mathcal{G}\left( \frac{u_s}{\Lambda}, \varphi \right)^2 
		+ C \|Du\|_{C^0(B_1(0))}^2 \nonumber 
\end{align}
for some constant $C = C(n,m,\alpha,\varphi^{(0)},\sigma) \in (0,\infty)$, where in the last step we used Corollary~\ref{lemma3_9}. To bound the third term on the right-hand side of \eqref{thm3_1_eqn4}, by Lemma~\ref{lemma2_6} we have that 
\begin{equation} \label{thm3_1_eqn8}
	\int_{B_{\gamma}(0) \cap \{ |x| > \tau \}} |v(X,\varphi(X))|^2 \leq C \int_{B_1(0)} \mathcal{G}\left( \frac{u_s}{\Lambda}, \varphi \right)^2 
		+ C \int_{B_{(3+\gamma)/4}(0)} r^4 |f|^2
\end{equation}
for some constant $C = C(n,m,\alpha,\varphi^{(0)},\gamma) \in (0,\infty)$.  By \eqref{mss3 f est}, \eqref{minimal avg schauder} and \eqref{minimal sym decay}, 
\begin{equation} \label{thm3_1_eqn10}
	|f| \leq C \|Du\|_{C^0(B_1(0))}^2 \Lambda^{-1} (|Du_s| + |D^2 u_s|)  
\end{equation}
on $B_{(3+\gamma)/4}(0)$ for some constant $C = C(n,m) \in (0,\infty)$.  Hence by \eqref{thm3_1_eqn10} and \eqref{minimal sym schauder},  
\begin{align} \label{thm3_1_eqn11}
	\int_{B_{(3+\gamma)/4}(0)} r^4 |f|^2 
		&\leq \frac{C \|Du\|_{C^0(B_1(0))}^4}{\Lambda^2} \int_{B_{(3+\gamma)/4}(0)} (|Du_s|^2 + |D^2 u_s|^2)
		\\&\leq \frac{C \|Du\|_{C^0(B_1(0))}^4}{\Lambda^2} \int_{B_1(0)} |u_s|^2 
		= C \|Du\|_{C^0(B_1(0))}^4 \nonumber 
\end{align}
for some constant $C = C(n,m) \in (0,\infty)$, where the last step uses $\Lambda^{-1} \|u_s\|_{L^2(B_1(0))} = 1$.  Using \eqref{thm3_1_eqn11} to bound the right-hand side of \eqref{thm3_1_eqn8}, 
\begin{equation} \label{thm3_1_eqn12}
	\int_{B_{\gamma}(0) \cap \{ |x| > \tau \}} |v(X,\varphi(X))|^2 \leq C \int_{B_1(0)} \mathcal{G}\left( \frac{u_s}{\Lambda}, \varphi \right)^2 
		+ C \,\|Du\|_{C^0(B_1(0))}^4
\end{equation}
for some constant $C = C(n,m,\alpha,\varphi^{(0)},\gamma) \in (0,\infty)$.  By Corollary~\ref{lemma3_9} and direction computation,  
\begin{equation} \label{thm3_1_eqn13}
	\int_{(B_{\gamma}(0) \setminus B_{1/4}(Z)) \cap \{ |x| > \tau \}} |x|^{2\alpha-2} |\xi|^2 
		\leq C \int_{B_1(0)} \mathcal{G}\left( \frac{u_s}{\Lambda}, \varphi \right)^2 + C \,\|Du\|_{C^0(B_1(0))}^2
\end{equation}
for some constant $C = C(n,m,\alpha,\varphi^{(0)}) \in (0,\infty)$.  Therefore, by bounding the right-hand side of \eqref{thm3_1_eqn4} using \eqref{thm3_1_eqn6}, \eqref{thm3_1_eqn7}, \eqref{thm3_1_eqn12} and \eqref{thm3_1_eqn13}, 
\begin{equation*}
	\int_{B_{\gamma}(0) \cap \{ |x| > \tau \}} \frac{|v(X,\varphi(X)) - D_x \varphi(X) \cdot \xi|^2}{|X-Z|^{n+2\alpha-\sigma}}  
	\leq C \int_{B_1(0)} \mathcal{G}\left( \frac{u_s}{\Lambda}, \varphi \right)^2 + C \|Du\|_{C^0(B_1(0))}^2. 
\end{equation*}
for some constant $C = C(n,m,\alpha,\varphi^{(0)},\gamma,\sigma) \in (0,\infty)$.  
\end{proof}

\begin{corollary} \label{cor3_2} 
Let $\alpha$ and $\varphi^{(0)}$ be as in Definition~\ref{varphi0 defn}.  For every $\sigma, \delta \in (0,1/4)$ there exists $\varepsilon_0 = \varepsilon_0(n,m,\alpha,\varphi^{(0)},\sigma) > 0$ such that if $\varphi \in \Phi_{\varepsilon_0}(\varphi^{(0)})$ and $(u,\Lambda) \in \mathcal{F}_{\varepsilon_0}(\varphi^{(0)})$ and if 
\begin{equation} \label{cor3_2 nogaps}
	B_{\delta}(0,y) \times \mathbb{R}^m \cap \{ P \in \mathcal{K}_M \cap B_{1/2}(0) \times \mathbb{R}^m : \mathcal{N}_M(P) \geq \alpha \} 
	\neq \emptyset \text{ for all } y \in B^{n-2}_{1/2}(0), 
\end{equation}
where $M = {\rm graph}\,u$, then 
\begin{equation} \label{cor3_2 concl}
	\int_{B_{1/2}(0)} \frac{\mathcal{G}(u_s/\Lambda,\varphi)^2}{r_{\delta}^{2-\sigma}} 
	\leq C \int_{B_1(0)} \mathcal{G}(u_s/\Lambda, \varphi)^2 + C \|Du\|_{C^0(B_1(0))}^2,
\end{equation}
where $C = C(n,m,\alpha,\varphi^{(0)},\sigma) > 0$, $r = |x|$ and $r_{\delta} = \max\{r,\delta\}$
\end{corollary}
\begin{proof} 
Follows from Corollary~\ref{thm3_1} using a argument similar to that of~\cite[Corollary~3.2]{Sim93}.
\end{proof}

\begin{remark} \label{F defn rmk2} {\rm 
As we previously noted in Remark~\ref{F defn rmk}, one might try to replace $\Lambda$ with $\Lambda_{0,1} = \|u_s\|_{L^2(B_1(0))}$.  If we did this, then in Step 2 of the above proof of Corollary~\ref{noncon_at_origin_cor}, when we rotate $M$ and apply Step 1 we obtain 
\begin{equation*}
	\int_{B_{\gamma}(0)} R^{-n+\sigma-2\alpha} \mathcal{G}(\widehat{u}_s/\widehat{\Lambda}_{0,\overline{\gamma}}, \varphi)^2 
	\leq C \int_{B_{\overline{\gamma}}(0)} \mathcal{G}(\widehat{u}_s/\widehat{\Lambda}_{0,\overline{\gamma}}, \varphi)^2 
		+ C \|Du\|_{C^0(B_1(0))}^2  
\end{equation*}
where $\overline{\gamma} = (1+\gamma)/2$, $\widehat{\Lambda}_{0,\overline{\gamma}} = \overline{\gamma}^{-n/2-1} \|\widehat{u}_s\|_{L^2(B_{\overline{\gamma}}(0))}$ and $C = C(n,m,\alpha,\varphi^{(0)},\gamma,\sigma) \in (0,\infty)$ is a constant.  
To then replace $\widehat{\Lambda}_{0,\overline{\gamma}}$ with $\Lambda_{0,1}$, we would need to know that $|\widehat{\Lambda}_{0,\overline{\gamma}}/\Lambda_{0,1} - 1|^2$ is bounded in terms of the excess $\int_{B_1(0)} \mathcal{G}(u_s/\Lambda_{0,1},\varphi)^2 +\|Du\|_{C^0(B_1(0))}^2$.  Such an estimate is much stronger than \eqref{F defn rmk eqn1}, which asserts that $|\Lambda_{0,1}/\Lambda - \|\varphi^{(0)}\|_{L^2(B_1(0))} |^2 \leq \varepsilon^2$.  The same issue arises in the proofs of Corollaries~\ref{lemma3_9} and \ref{thm3_1} when we translate $Z$ to the origin and apply Corollary~\ref{noncon_at_origin_cor} with $2 u(Z+X/2)$ in place of $u$ and also in the proof of Theorem~\ref{theorem1} below.  For this reason, we use a general parameter $\Lambda$ in the definition of $\mathcal{F}_{\varepsilon}(\varphi^{(0)})$.} 
\end{remark}

\section{The linear theory: asymptotic decay of blow-ups} \label{sec:lineartheory sec} 

In this section we recall some facts about the asymptotic decay of blow-ups from Section 7 of~\cite{KrumWic1}.  
Let $\varphi^{(0)}(X) = \{ \pm \op{Re}(c^{(0)}(x_1+ix_2)^{\alpha}) \}$ be as in Definition~\ref{varphi0 defn}.  Recall from Remark~\ref{graph_rmk} that $\varphi^{(0)}$ is an immersed submanifold.  More precisely, let $X = (re^{i\theta},y)$ denote cylindrical coordinates on $\mathbb{R}^n$, where $r \geq 0$, $\theta \in \mathbb{R}$ and $y \in \mathbb{R}^{n-2}$.  If $\alpha$ is an integer, ${\rm graph}\,\varphi^{(0)}$ is the union of the graphs of the single-valued harmonic functions $-\op{Re}(c^{(0)} r^{\alpha} e^{i\alpha\theta})$ and $+\op{Re}(c^{(0)} r^{\alpha} e^{i\alpha\theta})$.  If instead $\alpha = k_0/2$ for some odd integer $k_0 \geq 3$, ${\rm graph}\,\varphi^{(0)}$ is a single branched submanifold parameterized by $(re^{i\theta},y,\op{Re}(c^{(0)} r^{\alpha} e^{i\alpha\theta}))$ for $r \geq 0$, $\theta \in [0,4\pi]$ and $y \in \mathbb{R}^{n-2}$.  Hence for each Lebesgue measurable set $\Omega \subseteq \mathbb{R}^n$ we can consider Lebesgue measurable functions $w : {\rm graph}\,\varphi^{(0)} |_{\Omega} \rightarrow \mathcal{A}_2(\mathbb{R}^m)$.  Given Lebesgue measurable functions $w,v : {\rm graph}\,\varphi^{(0)} |_{\Omega} \rightarrow \mathcal{A}_2(\mathbb{R}^m)$ we let 
\begin{align*}
	&w(X,\varphi^{(0)}(X)) = \{w(X,\varphi^{(0)}_1(X)), w(X,-\varphi^{(0)}_1(X))\}, \\ 
	&w(X,\varphi^{(0)}(X)) - v(X,\varphi^{(0)}(X)) 
		\\&\hspace{15mm} = \{w(X,\varphi^{(0)}_1(X)) - v(X,\varphi^{(0)}_1(X)), w(X,-\varphi^{(0)}_1(X)) - v(X,-\varphi^{(0)}_1(X)) \}
\end{align*}
on each open ball $B \subset \Omega \setminus \{0\} \times \mathbb{R}^{n-2}$, where $\varphi^{(0)} = \{\pm \varphi^{(0)}_1\}$ on $B$ for some single-valued harmonic function $\varphi_1$.  We let $L^2({\rm graph}\, \varphi^{(0)} |_{\Omega}, \mathbb{R}^m)$ denote the Hilbert space of all Lebesgue measurable functions $w : {\rm graph}\, \varphi^{(0)} |_{\Omega} \rightarrow \mathbb{R}^m$ such that 
\begin{equation*}
	\int_{B_1(0)} |w(X,\varphi^{(0)}(X))|^2 \,dX < \infty .
\end{equation*}
(In particular, by definition $w \in L^2({\rm graph}\, \varphi^{(0)} |_{\Omega}, \mathbb{R}^m)$ is an $L^2$-function with respect to the Lebesgue measure on the domain $\Omega \subseteq \mathbb{R}^n$.)

\begin{lemma} \label{lemma4_2}
Let $\alpha$ and $\varphi^{(0)}$ be as in Definition~\ref{varphi0 defn}.  Let $\sigma \in (0,1)$.  If $w \in L^2(\op{graph} \varphi^{(0)} |_{B_1(0)}, \mathbb{R}^m) \cap C^2(\op{graph} \varphi^{(0)} |_{B_1(0) \setminus \{0\} \times \mathbb{R}^{n-2}},\mathbb{R}^m)$ such that $w(X,\varphi^{(0)}(X))$ is a two-valued function that is symmetric, homogeneous degree $\alpha$, and harmonic on $B_1(0) \setminus \{0\} \times \mathbb{R}^{n-2}$ and satisfies 
\begin{equation} \label{lemma4_2_decay}
	\int_{B_1(0)} \frac{|w(re^{i\theta},y,\varphi^{(0)}(re^{i\theta},y)) - \kappa(re^{i\theta},y,\varphi^{(0)}(re^{i\theta},y))|^2}{r^{2+2\alpha-\sigma}} < \infty 
\end{equation}
for some 
\begin{equation*}
	\kappa(re^{i\theta},y,\varphi^{(0)}(re^{i\theta},y)) = \kappa_1(r,y) \cdot D_1 \varphi^{(0)}(re^{i\theta},y) + \kappa_2(r,y) \cdot D_2 \varphi^{(0)}(re^{i\theta},y)
\end{equation*}
where $\kappa_1, \kappa_2 \in L^{\infty}(B^{n-1}_1(0),\mathbb{R}^m)$, then 
\begin{align} \label{lemma4_2_eqn1}
	w(re^{i\theta},y,\varphi^{(0)}(re^{i\theta},y)) ={}& a_0 r^{\alpha} \cos(\alpha \theta) + b_0 r^{\alpha} \sin(\alpha \theta) 
	\\& + \sum_{j=1}^{n-2} \left( a_j D_1 \varphi^{(0)}(re^{i\theta},y) \cdot y_j + b_j D_2 \varphi^{(0)}(re^{i\theta},y) \cdot y_j \right) \nonumber
\end{align}
for some $a_0,b_0 \in \mathbb{R}^m$ and $a_j,b_j \in \mathbb{R}$ for $j = 1,\ldots,n-2$. 
\end{lemma}
\begin{proof} 
See~\cite[Lemma 7.1]{KrumWic1}.
\end{proof}

\begin{definition} \label{L defn}
$\mathcal{L} \subset L^2(\op{graph} \varphi^{(0)} |_{B_1(0)}, \mathbb{R}^m)$ is the span of all functions $a_0 r^{\alpha} \cos(\alpha \theta)$ and $b_0 r^{\alpha} \sin(\alpha \theta)$ for $a_0,b_0 \in \mathbb{R}^m$ and $D_k \varphi^{(0)}(re^{i\theta},y) \,y_j$ for $k \in \{1,2\}$ and $j \in \{1,2,\ldots,n-2\}$. 
\end{definition}

\begin{definition} \label{L projection defn}
Given $w \in L^2(\op{graph} \varphi^{(0)} |_{B_1(0)}, \mathbb{R}^m)$ and $\rho \in (0,1]$, we say $\psi_{\rho} \in \mathcal{L}$ is the \emph{orthogonal projection} of $w$ onto $\mathcal{L}$ in $B_{\rho}(0)$ if 
\begin{equation*}
	\int_{B_{\rho}(0)} |w(X,\varphi^{(0)}(X)) - \psi_{\rho}(X,\varphi^{(0)}(X))|^2 
	= \inf_{\psi \in \mathcal{L}} \int_{B_{\rho}(0)} |w(X,\varphi^{(0)}(X)) - \psi(X,\varphi^{(0)}(X))|^2 . 
\end{equation*}
\end{definition}

Whenever $w \in L^2(\op{graph} \varphi^{(0)} |_{B_1(0)}, \mathbb{R}^m)$ and $\psi_{\rho} \in \mathcal{L}$ is the orthogonal projection of $w$ onto $\mathcal{L}$ in $B_{\rho}(0)$, we let 
\begin{equation*}
	w_{\rho} = w - \psi_{\rho}. 
\end{equation*}
By standard Hilbert space theory, $\psi_{\rho}$ exists and is unique and moreover $w_{\rho}$ is orthogonal to $\mathcal{L}$ in $L^2(\op{graph} \varphi^{(0)} |_{B_{\rho}(0)}, \mathbb{R}^m)$. 

\begin{lemma} \label{lemma4_14} 
Let $\alpha$ and $\varphi^{(0)}$ be as in Definition~\ref{varphi0 defn}.  Let $\vartheta \in (0,1/16)$, $\sigma \in (0,1)$, and $\beta_1, \beta_2 \in (0,\infty)$.  If $w \in L^2(\op{graph} \varphi^{(0)} |_{B_1(0)}, \mathbb{R}^m) \cap C^2(\op{graph} \varphi^{(0)} |_{B_1(0) \setminus \{0\} \times \mathbb{R}^{n-2}},\mathbb{R}^m)$ such that $w(X,\varphi^{(0)}(X))$ is a harmonic symmetric 2-valued function on $B_1(0) \setminus \{0\} \times \mathbb{R}^{n-2}$ and for each $\rho \in [\vartheta,1/4]$, 
\begin{equation*} 
	\int_{B_{\rho/2}(0)} \frac{|w_{\rho}(X,\varphi^{(0)}(X)) - D_x \varphi^{(0)}(X) \cdot \lambda_{\rho}(z)|^2}{|X - (0,z)|^{n+2\alpha-\sigma}} 
		\leq \beta_1 \rho^{-n-2\alpha+\sigma}  \int_{B_{\rho}(0)} |w_{\rho}(X,\varphi^{(0)}(X))|^2
\end{equation*}
for every $z \in B^{n-2}_{\rho/2}(0)$ and some bounded function $\lambda_{\rho} : B^{n-2}_{\rho/2}(0) \rightarrow \mathbb{R}^2$ such that  
\begin{equation*} 
	\sup_{z \in B^{n-2}_{\rho/2}(0)} |\lambda_{\rho}(z)|^2 \leq \beta_2 \rho^{2-n-2\alpha} \int_{B_{\rho}(0)} |w_{\rho}(X,\varphi^{(0)}(X))|^2 
\end{equation*}
and 
\begin{equation} \label{lemma4_14_eqn1}
	\int_{B_{\rho/4}(0)} R^{2-n} \left| \frac{\partial}{\partial R} \left( \frac{w(X,\varphi^{(0)}(X))}{R^{\alpha}} \right) \right|^2 
	\leq \beta_2 \rho^{-n-2\alpha} \int_{B_{\rho}(0)} |w_{\rho}(X,\varphi^{(0)}(X))|^2, 
\end{equation}
then 
\begin{equation*} 
	\vartheta^{-n-2\alpha} \int_{B_{\vartheta}(0)} |w_{\vartheta}(X,\varphi^{(0)}(X))|^2 \leq C \vartheta^{2\mu} \int_{B_1(0)} |w_1(X,\varphi^{(0)}(X))|^2
\end{equation*}
for some constants $\mu \in (0,1)$ and depending only $C \in (0,\infty)$ on $n$, $\alpha$, $\varphi^{(0)}$, $\sigma$, $\beta_1$, $\beta_2$.  Note that in particular $C$ and $\mu$ are independent of $\vartheta$.  
\end{lemma}
\begin{remark}
Note that in \eqref{lemma4_14_eqn1} that 
\begin{equation*}
	\frac{\partial}{\partial R} \left( \frac{w_{\rho}(X,\varphi^{(0)}(X))}{R^{\alpha}} \right) 
	= \frac{\partial}{\partial R} \left( \frac{w(X,\varphi^{(0)}(X))}{R^{\alpha}} \right) . 
\end{equation*}
\end{remark}
\begin{proof}[Proof of Lemma~\ref{lemma4_14}]
See~\cite[Lemmas 7.2 and 7.3]{KrumWic1}. 
\end{proof}

\section{Proof of the excess decay lemma} \label{sec:excess decay sec}

\begin{proof}[Proof of Lemma~\ref{lemma1}] 
By rotating slightly, we may assume that $\varphi \in \Phi_{\varepsilon_0}(\varphi^{(0)})$.  Fix $\alpha$ and $\varphi^{(0)}$ as in Definition~\ref{varphi0 defn} and $\vartheta \in (0,1/4)$.  For $k \in \{1,2,3,\ldots\}$ let $0 < \varepsilon_k \leq \delta_k \downarrow 0$, $\varphi_k \in \Phi_{\varepsilon_k}(\varphi^{(0)})$, $(u_k,\Lambda_k) \in \mathcal{F}_{\varepsilon_k}(\varphi^{(0)})$ and $M_k = {\rm graph}\,u_k$ and assume that 
\begin{equation*}
	u_k(0) = \{0,0\}, \quad Du_k(0) = \{0,0\}, \quad \mathcal{N}_{M_k}(0) \geq \alpha
\end{equation*}
and that alternative (i) of Lemma~\ref{lemma1} is false, that is for every $y_0 \in B^{n-2}_{1/2}(0)$ 
\begin{equation*}
	(B_{\delta_k}(0,y_0) \times \mathbb{R}^m) \cap \{ P \in \mathcal{K}_{M_k} \cap B_{1/2}(0) \times \mathbb{R}^m : \mathcal{N}_{M_k}(P) \geq \alpha \} \neq \emptyset . 
\end{equation*}
We want to show that there are constants $\gamma = \gamma(n,m,\alpha,\varphi^{(0)},\vartheta) \in [1,\infty)$, $\mu = \mu(n,m,\alpha,\varphi^{(0)}) \in (0,1)$ and $C = C(n,m,\alpha,\varphi^{(0)}) \in (0,\infty)$ and for infinitely many $k$ there exists $\widetilde{\varphi}_k \in \widetilde{\Phi}_{\gamma \varepsilon_k}(\varphi^{(0)})$ such that   
\begin{equation} \label{lemma1_eqn1} 
	\vartheta^{-n-2\alpha} \int_{B_{\vartheta}(0)} \mathcal{G}(u_k,\widetilde{\varphi}_k)^2 
	\leq C\vartheta^{2\mu} \left( \int_{B_1(0)} \mathcal{G}(u_k,\varphi_k)^2 + \|Du_k\|_{C^0(0)} \right) . 
\end{equation}
By the arbitrariness of the sequences this will complete the proof of Lemma~\ref{lemma1}.  

Let $\tau_k \downarrow 0$ slowly enough that Lemma~\ref{lemma2_6} holds with $\gamma = 3/4$, $\tau = \tau_k\vartheta$, $\beta = 1/2$, $u = u_k$, $\Lambda = \Lambda_k$ and $\varphi = \varphi_k$ and for each $\rho \in [\vartheta,1/2]$, Theorem~\ref{mainestthm}, Corollary~\ref{lemma3_9}, Corollary~\ref{thm3_1} and Corollary~\ref{cor3_2} hold true with $\gamma = 1/2$, $\tau = \tau_k$, $u(X) = \rho^{-1} u_k(\rho X)$, $\Lambda = \rho^{\alpha-1} \Lambda_k$, $\varphi = \varphi_k$ and $\sigma = 1/2$.  In particular, by Lemma~\ref{lemma2_6} there exists an open set $U_k$ with $\{ (x,y) \in B_{3/4}(0) \cap \{ |x| > \tau_k\vartheta \} \subset U_k$ and a function $v_k \in C^2(\op{graph} \varphi_k |_{U_k},\mathbb{R}^m)$ such that 
\begin{equation} \label{lemma1_eqn2} 
	\frac{u_{k,s}(X)}{\Lambda_k} = \{ \varphi_{k,1}(X) + v_k(X,\varphi_{k,1}(X)), -\varphi_{k,1}(X) + v_k(X,-\varphi_{k,1}(X)) \} 
\end{equation} 
for each $X \in U_k$, where $u_{k,s}(X) = \{ \pm (u_{k,1}(X) - u_{k,2}(X))/2 \}$ is the symmetric part of $u_k(X) = \{u_{k,1}(X),u_{k,2}(X)\}$ (as in \eqref{avg and free defn}) and $\varphi_k(X) = \{ \pm \varphi_{k,1}(X) \}$.  For each $\rho \in [\vartheta,1/2]$, by Theorem~\ref{mainestthm} and Corollary~\ref{cor3_2} we have that 
\begin{equation} \label{lemma1_eqn3} 
	\int_{B_{\rho/2}(0)} R^{2-n} \left| \frac{\partial}{\partial R} \left(\frac{u_{k,s}}{ R^{\alpha} \Lambda_k}\right) \right|^2 
	\leq C \rho^{-n-2\alpha} \int_{B_{\rho}(0)} \mathcal{G}\left( \frac{u_{k,s}}{\Lambda_k}, \varphi_k \right)^2 + C \|Du_k\|_{C^0(B_{\rho}(0))}^2
\end{equation} 
and 
\begin{equation} \label{lemma1_eqn4} 
	\rho^{-n-2\alpha+1/2} \int_{B_{\rho/2}(0)} \frac{\mathcal{G}(u_{k,s}/\Lambda_k,\varphi_k)^2}{r_{\delta}^{1/2}}  
	\leq C \rho^{-n-2\alpha} \int_{B_{\rho}(0)} \mathcal{G}\left( \frac{u_{k,s}}{\Lambda_k}, \varphi_k \right)^2 
	+ C \|Du_k\|_{C^0(B_{\rho}(0))}^2 
\end{equation} 
for each $\delta \geq 2\delta_k/\vartheta$, where where $r = |x|$ and $r_{\delta} = \max\{r,\delta\}$ and $C = C(n,m,\alpha,\varphi^{(0)}) \in (0,\infty)$ is a constant (independent of $\vartheta$).  Moreover, for each $k$, $\rho \in [\vartheta,1/2]$ and $z \in B^{n-2}_{\rho/2}(0)$ there exists $Z_k \in \mathcal{K}_{u_k} \cap B_{\rho/2}(0) \cap B_{\delta_k}(0,z)$ such that $\mathcal{N}_{M_k}(Z_k,u_{k,1}(Z_k)) \geq \alpha$, where $u_k(Z_k) = \{u_{k,1}(Z_k),u_{k,1}(Z_k)\}$ for $u_{k,1}(Z_k) \in \mathbb{R}^m$.  By Corollary~\ref{lemma3_9} and Corollary~\ref{thm3_1}, 
\begin{align} 
	\label{lemma1_eqn5} &\frac{|\xi_k|^2}{\rho^2} \leq C \rho^{-n-2\alpha} \int_{B_{\rho}(0)} \mathcal{G}\left( \frac{u_{k,s}}{\Lambda_k}, \varphi_k \right)^2
		+ C \|Du_k\|_{C^0(B_{\rho}(0))}^2, \\
	\label{lemma1_eqn6} &\rho^{-1/2} \int_{B_{\rho/2}(0) \cap \{ r > \tau_k\rho \}} \frac{|v_k(X,\varphi'_k(X)) + D_x \varphi^{(0)}(X) \cdot \xi_k|^2}{
		|X-Z_k|^{n+2\alpha-1/2}} \\&\hspace{15mm} \leq C \rho^{-n-2\alpha} \int_{B_{\rho}(0)} \mathcal{G}\left( \frac{u_{k,s}}{\Lambda_k}, 
		\varphi_k \right)^2 + C \|Du_k\|_{C^0(B_{\rho}(0))}^2, \nonumber
\end{align} 
where $\xi_k$ is the projection of $Z_k$ onto $\{0\} \times \mathbb{R}^{n-2}$ and $C = C(n,m,\alpha,\varphi^{(0)}) \in (0,\infty)$ is a constant (independent of $\vartheta$). 

Since $\varphi_k \in \Phi_{\varepsilon_k}(\varphi^{(0)})$, $\varphi^{(0)}(re^{i\theta},y) = \op{Re}(c^{(0)} r^{\alpha} e^{i\alpha \theta})$ and $\varphi_k(re^{i\theta},y) = \op{Re}(c_k r^{\alpha} e^{i\alpha \theta})$ for some constants $c^{(0)}, c_k \in \mathbb{C}^m$ with $|c_k - c^{(0)}| \leq C(n,m,\alpha) \,\varepsilon_k$.  Define $w_k : {\rm graph}\,\varphi^{(0)} |_{U_k} \rightarrow \mathbb{R}^m$ by 
\begin{equation*}
	w_k(re^{i\theta},y,\op{Re}(c^{(0)} r^{\alpha} e^{i\alpha\theta})) = \frac{v_k(re^{i\theta},y,\op{Re}(c_k r^{\alpha} e^{i\alpha\theta}))}{E_k} 
\end{equation*}
for each $X = (re^{i\theta},y) \in U_k$, where the excess $E_k$ is defined by 
\begin{equation*}
	E_k = \left( \int_{B_1(0)} \mathcal{G}\left( \frac{u_{k,s}}{\Lambda_k}, \varphi_k \right)^2 + \|Du_k\|_{C^0(B_1(0))} \right)^{1/2} 
\end{equation*}
(without $\|Du_k\|_{C^0(B_1(0))}$ squared so that $E_k^{-1} \|Du_k\|_{C^0(B_1(0))} \rightarrow 0$).  Associate $w_k$ with the symmetric two-valued function $\widehat{w}_k : U_k \rightarrow \mathcal{A}_2(\mathbb{R}^m)$ defined by 
\begin{equation*}
	\widehat{w}_k(X) = w_k(X,\varphi^{(0)}(X)) = \{ \pm w_k(re^{i\theta},y,\op{Re}(c^{(0)} r^{\alpha} e^{i\alpha\theta})) \}
\end{equation*}
for each $X = (re^{i\theta},y) \in U_k$.  Since \eqref{mss3} holds true with $u_k$ and $u_{k,s}/\Lambda_k$ in place of $u$ and $w$ and $\varphi_k$ is harmonic, in each open ball $B \subset U_k$, $w_k$ satisfies  
\begin{equation} \label{lemma1_eqn7}
	\Delta \widehat{w}_{k,1}^{\kappa} = f_{k,1,\kappa} \text{ in } U_k, 
\end{equation}
where for each $X = (re^{i\theta},y) \in B$ we let $\varphi_{k,1}(X) = \op{Re}(c_k r^{\alpha} e^{i\alpha\theta})$ and $\widehat{w}_{k,1}(X) = w_k(re^{i\theta},y,\op{Re}(c^{(0)} r^{\alpha} e^{i\alpha\theta}))$ (with $|c_k - c^{(0)}| \leq C(n,m,\alpha) \,\varepsilon_k$), $u_{k,a}(X) = (u_{k,1}(X) + u_{k,2}(X))/2$ (as in \eqref{avg and free defn}) and 
\begin{equation*}
	f_{k,1,\kappa} = -E_k^{-1} \,D_i (b^{ij}_{\kappa\lambda}(Du_{k,a},Du_{k,s}) \,D_j (\varphi_{k,1}^{\kappa} + E_k w_{k,1}^{\kappa}) .
\end{equation*}
Since $(u_k,\Lambda_k) \in \mathcal{F}_{\varepsilon_k}(\varphi^{(0)})$, 
\begin{equation*}
	\lim_{k \rightarrow \infty} \|u_k\|_{C^{1,1/2}(B_1(0))} = 0, \quad \Lambda_k^{-1} \|u_{k,s}\|_{L^2(B_1(0))} \leq \|\varphi^{(0)}\|_{L^2(B_1(0))} + 1. 
\end{equation*}
Thus by \eqref{mss3 f est}, \eqref{minimal avg schauder} and \eqref{minimal sym decay}, $f_k = \{ \pm f_{k,1} \}$ satisfies 
\begin{align*}
	|f_k| &\leq C E_k^{-1} \Lambda_k^{-1} \left( (|Du_{k,a}|^2 + |Du_{k,s}|^2) \,|D^2 u_{k,s}| 
		+ (|Du_{k,a}| |D^2 u_{k,a}| + |Du_{k,s}| |D^2 u_{k,s}|) \,|Du_{k,s}| \right) 
	\\&\leq C E_k^{-1} \Lambda_k^{-1} \|Du_k\|_{C^0(B_1(0))}^2 \|u_{k,s}\|_{L^2(B_1(0))} \leq C E_k
\end{align*}
in $U$ for some constant $C = C(n,m,\varphi^{(0)}) \in (0,\infty)$.  In particular, $\|f_k\|_{C^0(B)} \rightarrow 0$ as $k \rightarrow \infty$ for each closed ball $B \subset B_1(0) \setminus \{0\} \times \mathbb{R}^{n-2}$.  Moreover, by standard elliptic estimates applied to \eqref{lemma1_eqn7}, for each closed ball $B \subset B_1(0) \setminus \{0\} \times \mathbb{R}^{n-2}$, 
\begin{equation*} 
	\|\widehat{w}_k\|_{C^{1,1/2}(B)} \leq C
\end{equation*}
for all sufficiently large $k$ and some constant $C = C(n,m,\varphi^{(0)},B) \in (0,\infty)$.  After passing to a subsequence, there exists a function $w : {\rm graph}\,\varphi^{(0)} |_{B_{1/2}(0) \setminus \{0\} \times \mathbb{R}^{n-2}} \rightarrow \mathbb{R}^m$ such that $w_k \rightarrow w$ in $C^1(\op{graph} \varphi^{(0)} |_K,\mathbb{R}^m)$ for each compact set $K \subseteq B_1(0) \setminus \{0\} \times \mathbb{R}^{n-2}$.  By \eqref{lemma1_eqn7}, $w(X,\varphi^{(0)}_1(X))$ is harmonic in each open ball $B \subset B_{1/2}(0) \setminus \{0\} \times \mathbb{R}^{n-2}$, where we let $\varphi^{(0)} = \{ \pm \varphi^{(0)}_1(X) \}$ in $B$ for some single-valued harmonic function $\varphi^{(0)}_1 : B \rightarrow \mathbb{R}^m$.  By \eqref{lemma1_eqn4}, 
\begin{equation*} 
	\int_{B_{\rho}(0) \cap B^2_{\delta}(0) \times \mathbb{R}^{n-2}} \mathcal{G}\left( \frac{u_{k,s}}{\Lambda_k}, \varphi_k \right)^2 
		\leq \frac{C \delta^{1/2} E_k^2}{\rho^{1/2}}
\end{equation*}
for all $\rho \in [\vartheta,1/4]$ and $2\delta_k/\vartheta \leq \delta < \rho$, where $C = C(n,m,\alpha,\varphi^{(0)}) \in (0,\infty)$ is a constant.  Thus 
\begin{equation} \label{lemma1_eqn8} 
	\lim_{j \rightarrow \infty} \frac{1}{E_k^2} \int_{B_{\rho}(0)} \mathcal{G}\left( \frac{u_{k,s}}{\Lambda_k}, \varphi_k \right)^2 
		= \int_{B_{\rho}(0)} |w(X,\varphi^{(0)}(X))|^2. 
\end{equation}
for each $\rho \in [\vartheta,1/4]$, where 
\begin{equation*}
	w(X,\varphi^{(0)}(X)) = \{ \pm w(re^{i\theta},y,\op{Re}(c^{(0)} (x_1+ix_2)^{\alpha})) \}
\end{equation*}
for each $X = (re^{i\theta},y) \in B_{1/4}(0) \setminus \{0\} \times \{0\} \times \mathbb{R}^{n-2}$.  

By dividing both sides of \eqref{lemma1_eqn3} by $E_k^2$ and letting $k \rightarrow \infty$ using \eqref{lemma1_eqn8}, for every $\rho \in [\vartheta,1/4]$, 
\begin{equation} \label{lemma1_eqn9} 
	\int_{B_{\rho/2}(0)} R^{2-n} \left| \frac{\partial}{\partial R} \left(\frac{w(X,\varphi^{(0)}(X))}{R^{\alpha}}\right) \right|^2 
	\leq C \rho^{-n-2\alpha} \int_{B_{\rho}(0)} |w(X,\varphi^{(0)}(X))|^2 
\end{equation}
for $C = C(n,m,\alpha,\varphi^{(0)}) \in (0,\infty)$ (independent of $\vartheta$).  Moreover, for each for each $k$, $\rho \in [\vartheta,1]$ and $z \in B^{n-2}_{\rho/2}(0)$ there exists $Z_k \in \mathcal{K}_{u_k} \cap B_{\rho/2}(0) \cap B_{\delta_k}(0,z)$ such that \eqref{lemma1_eqn5} and \eqref{lemma1_eqn6} hold true.  Letting $\xi_k$ denote the projection of $Z_k$ onto $\{0\} \times \mathbb{R}^{n-2}$, by \eqref{lemma1_eqn5} after passing to a subsequence there exists $\lambda(z) \in \mathbb{R}^2$ such that $\xi_k/E_k \rightarrow \lambda(z)$ and, using \eqref{lemma1_eqn8}, 
\begin{equation} \label{lemma1_eqn10} 
	|\lambda(z)|^2 \leq C \rho^{2-n-2\alpha} \int_{B_{\rho}(0)} |w(X,\varphi^{(0)}(X))|^2 
\end{equation}
for some constant $C = C(n,m,\alpha,\varphi^{(0)}) \in (0,\infty)$.  By dividing both sides of \eqref{lemma1_eqn6} by $E_k^2$ and letting $k \rightarrow \infty$ using \eqref{lemma1_eqn8}, 
\begin{equation} \label{lemma1_eqn11} 
	\int_{B_{\rho/2}(0)} \frac{|w(X,\varphi'_k(X)) + D_x \varphi^{(0)}(X) \cdot \lambda(z)|^2}{|X-(0,z)|^{n+2\alpha-1/2}} 
	\leq C \rho^{-n-2\alpha+1/2} \int_{B_{\rho}(0)} |w(X,\varphi^{(0)}(X))|^2 
\end{equation} 
for all $\rho \in [\vartheta,1/4]$ and $z \in B^{n-2}_{\rho/2}(0)$, where $C = C(n,m,\alpha,\varphi^{(0)}) \in (0,\infty)$ is a constant (independent of $\vartheta$).  Notice that this construction gives us for each $\rho \in [\vartheta,1/4]$ and $z \in B^{n-2}_{\rho/2}(0)$ there is at least on $\lambda(z)$ such that \eqref{lemma1_eqn11} holds true, with $\lambda(z)$ possibly not unique or depending on $\rho$.  However, if there existed $\vartheta \leq \rho_1 \leq \rho_2 \leq 1/4$, $z \in B^{n-2}_{\rho_1}(0)$ and $\lambda_1,\lambda_2 \in \mathbb{R}^2$ such that 
\begin{equation*} 
	\int_{B_{\rho_j/2}(0)} \frac{|w(X,\varphi'_k(X)) + D_x \varphi^{(0)}(X) \cdot \lambda_j|^2}{|X-(0,z)|^{n+2\alpha-1/2}} 
	\leq C \rho_j^{-n-2\alpha+1/2} \int_{B_{\rho_j}(0)} |w(X,\varphi^{(0)}(X))|^2 
\end{equation*} 
for $j = 1,2$, then by the triangle inequality 
\begin{equation*} 
	\int_{B_{\rho_1/2}(0)} \frac{|D_x \varphi^{(0)}(X) \cdot (\lambda_1 - \lambda_2)|^2}{|X-(0,z)|^{n+2\alpha-1/2}} < \infty,
\end{equation*} 
which since $\varphi^{(0)}$ is homogeneous degree $\alpha$ and $D_1 \varphi^{(0)}(re^{i\theta},y)$ and $D_2 \varphi^{(0)}(re^{i\theta},y)$ are $L^2$-orthogonal (as functions of $\theta$) necessarily implies that $\lambda_1 = \lambda_2$.  Therefore, we have a well-defined bounded function $\lambda : B^{n-2}_{1/8}(0) \rightarrow \mathbb{R}^2$ such that \eqref{lemma1_eqn10} and \eqref{lemma1_eqn11} hold true for all $\rho \in [\vartheta,1/4]$ and $z \in B^{n-2}_{\rho/2}(0)$.  In particular, by \eqref{lemma1_eqn10} with $\rho = 1$, $|\lambda(z)| \leq C$ for all $z \in B^{n-2}_{1/8}(0)$ and some constant $C = C(n,m,\alpha,\varphi^{(0)}) \in (0,\infty)$. 

Let $\psi \in \mathcal{L}$, where $\mathcal{L}$ is as in Definition~\ref{L defn}, and express $\psi$ as 
\begin{align*} 
	\psi(re^{i\theta},y,\varphi^{(0)}(re^{i\theta},y)) ={}& a_0 r^{\alpha} \cos(\alpha \theta) + b_0 r^{\alpha} \sin(\alpha \theta) 
	\\& + \sum_{j=1}^{n-2} \left( a_j D_1 \varphi^{(0)}(re^{i\theta},y) \,y_j + b_j D_2 \varphi^{(0)}(re^{i\theta},y) \,y_j \right) \nonumber
\end{align*}
for some $a_0,b_0 \in \mathbb{R}^m$ and $a_j,b_j \in \mathbb{R}$ for $j = 1,\ldots,n-2$.  Since $a_0 r^{\alpha} \cos(\alpha \theta)$, $b_0 r^{\alpha} \sin(\alpha \theta)$, $D_1 \varphi^{(0)}(re^{i\theta},y) \,y_j$, and $D_2 \varphi^{(0)}(re^{i\theta},y) \,y_j$ ($1 \leq j \leq n-2$) are mutually orthogonal in $L^2(B_1(0), \pi_0^*d\mathcal{L}^n,\mathbb{R}^m)$,  
\begin{equation} \label{lemma1_eqn12}
	|a_0|^2 + |b_0|^2 + \sum_{j=1}^{n-2} (|a_j|^2 + |b_j|^2) \leq C \int_{B_1(0)} |\psi(X,\varphi^{(0)}(X))|^2 \,dX 
\end{equation}
for some constant $C = C(n,m,\alpha) \in (0,\infty)$.  Let $B = (B_{ij})$ be the $n \times n$ skew-symmetric matrix with $B_{ij} = 0$ if $1 \leq i,j \leq 2$, $B_{1j} = -B_{j1} = a_j$, $B_{2j} = -B_{j2} = b_j$ and $B_{ij} = 0$ if $3 \leq i,j \leq n$.  Recall that since $\varphi_k \in \Phi_{\varepsilon_k}(\varphi^{(0)})$, $\varphi^{(0)}(re^{i\theta},y) = \op{Re}(c^{(0)} r^{\alpha} e^{i\alpha \theta})$ and $\varphi_k(re^{i\theta},y) = \op{Re}(c_k r^{\alpha} e^{i\alpha \theta})$ for some constants $c^{(0)}, c_k \in \mathbb{C}^m$ with $|c_k - c^{(0)}| \leq C \varepsilon_k$ for some constant $C = C(n,\alpha) \in (0,\infty)$.  Thus for each $k \in \{1,2,3,\ldots\}$ we have a function $\widetilde{\varphi}_k \in \widetilde{\Phi}_{C \varepsilon_k}(\varphi^{(0)})$, where $C = 1 + C_1(n,m,\alpha) \,\|\psi\|_{L^2(B_1(0))}$ (by \eqref{lemma1_eqn12}), such that $\widetilde{\varphi}_k$ is defined by 
\begin{equation} \label{lemma1_eqn13}
	\widetilde{\varphi}_k(e^{-E_k B} X) = \{ \pm \op{Re}( (c_k + E_k (a_0 + i b_0)) \,(x_1+ix_2)^{\alpha} ) \}
\end{equation} 
for each $X = (re^{i\theta},y) \in B_1(0) \setminus \{0\} \times \mathbb{R}^{n-2}$.  We can repeat the blow-up procedure above with $u_k \circ e^{-E_k B}$, $\widetilde{\varphi}_k \circ e^{-E_k B}$ and $E_k$ in place of $u_k$, $\varphi_k$ and $E_k$ to produce a blow-up $\widetilde{w}$.  For each $X = (x,y) \in B_{3/4}(0)$ with $|x| \geq \tau$ let $\varphi_k = \{\pm \varphi_{k,1}\}$ on $B_{|x|/2}(X)$ for some single-valued harmonic function $\varphi_{k,1}$.  By \eqref{lemma1_eqn2} and Taylor's theorem applied to $\varphi_{k,1}$, for every $\tau \in (0,3/4)$ and sufficiently large $k$ 
\begin{align} \label{lemma1_eqn14}
	\frac{u_{k,s}(e^{-E_k B} X)}{\Lambda_k} &= \{ \pm (\varphi_{k,1}(e^{-E_k B} X) + v_k(e^{-E_k B} X,\varphi_{k,1}(e^{-E_k B} X))) \}  
	\\&= \Bigg\{ \pm \Bigg( \varphi_{k,1}(X) - E_k \sum_{j=1}^{n-2} (a_j D_{x_1} \varphi_{k,1}(X) + b_j D_{x_2} \varphi_{k,1}(X))\,y_j \nonumber 
		\\& \hspace{45mm} + v_k(e^{-E_k B} X,\varphi_{k,1}(e^{-E_k B} X)) + \mathcal{R}_k(X) \Bigg) \Bigg\} , \nonumber
\end{align} 
where $|\mathcal{R}_k(X)| \leq C(\varphi^{(0)}) \,E_k^2 \,\tau^{\alpha-2}$.  Thus by \eqref{lemma1_eqn13} and \eqref{lemma1_eqn14}, $\widetilde{w} = w - \psi$.  By replacing $u_k$ and $\varphi_k$ by $u_k \circ e^{-E_k B}$ and $\widetilde{\varphi}_k \circ e^{-E_k B}$ in the argument leading to \eqref{lemma1_eqn8}, \eqref{lemma1_eqn9}, \eqref{lemma1_eqn10} and \eqref{lemma1_eqn11} we obtain the following estimates for $\widetilde{w} = w - \psi$.  In place of \eqref{lemma1_eqn8} we obtain 
\begin{equation} \label{lemma1_eqn16} 
	\lim_{j \rightarrow \infty} \frac{1}{E_k^2} \int_{B_{\rho}(0)} \mathcal{G}\left( \frac{u_{k,s}}{\Lambda_k}, \widetilde{\varphi}_k \right)^2 
		= \int_{B_{\rho}(0)} |w(X,\varphi^{(0)}(X)) - \psi(X,\varphi^{(0)}(X))|^2. 
\end{equation}
for all $\rho \in [\vartheta,1/4]$.  In place of \eqref{lemma1_eqn9} we obtain  
\begin{align} \label{lemma1_eqn17} 
	&\int_{B_{\rho/2}(0)} R^{2-n} \left| \left(\frac{\partial}{\partial R} \frac{w(X,\varphi^{(0)}(X))}{R^{\alpha}}\right) \right|^2 
	\\&\hspace{15mm} \leq C \rho^{-n-2\alpha} \int_{B_1(0)} |w(X,\varphi^{(0)}(X)) - \psi(X,\varphi^{(0)}(X))|^2 \nonumber 
\end{align}
for all $\rho \in [\vartheta,1/4]$ and some constant $C = C(n,m,\alpha,\varphi^{(0)}) \in (0,\infty)$ (independent of $\vartheta$).  In place of \eqref{lemma1_eqn10} and \eqref{lemma1_eqn11}, we have that there exists a bounded function $\lambda_{\psi} : B^{n-2}_{1/8}(0) \rightarrow \mathbb{R}^2$ (depending on $\psi$) such that 
\begin{align} 
	\label{lemma1_eqn18} &|\lambda_{\psi}(z)|^2 \leq C \rho^{2-n-2\alpha} \int_{B_{\rho}(0)} |w(X,\varphi^{(0)}(X)) - \psi(X,\varphi^{(0)}(X))|^2, \\
	\label{lemma1_eqn19} &\int_{B_{\rho/2}(0)} \frac{|w(X,\varphi'_k(X)) - \psi(X,\varphi^{(0)}(X)) + D_x \varphi^{(0)}(X) \cdot \lambda_{\psi}(z)|^2}{
		|X-(0,z)|^{n+2\alpha-1/2}} 
	\\&\hspace{15mm} \leq C \rho^{-n-2\alpha+1/2} \int_{B_{\rho}(0)} |w(X,\varphi^{(0)}(X)) - \psi(X,\varphi^{(0)}(X))|^2 \nonumber
\end{align} 
for all $\rho \in [\vartheta,1/4]$ and $z \in B^{n-2}_{\rho/2}(0)$ and some constant $C = C(n,m,\alpha,\varphi^{(0)}) \in (0,\infty)$ (independent of $\vartheta$).   

Now for given $\rho \in [\vartheta,1/4]$, choose $\psi_{\rho} \in \mathcal{L}$ to be the homogeneous degree $\alpha$ function such that 
\begin{equation*}
	\int_{B_{\rho}(0)} |w(X,\varphi^{(0)}(X)) - \psi_{\rho}(X,\varphi^{(0)}(X))|^2 
	= \inf_{\psi \in \mathcal{L}} \int_{B_{\rho}(0)} |w(X,\varphi^{(0)}(X)) - \psi(X,\varphi^{(0)}(X))|^2. 
\end{equation*}
Since $\int_{B_1(0)} |w(X,\varphi^{(0)}(X))|^2 \leq 1$, 
\begin{align} \label{lemma1_eqn20} 
	\int_{B_{1}(0)} |\psi_{\rho}(X,\varphi^{(0)}(X))|^2 &= \rho^{-n-2\alpha} \int_{B_{\rho}(0)} |\psi_{\rho}(X,\varphi^{(0)}(X))|^2 \\
		&\leq 2 \rho^{-n-2\alpha} \int_{B_{\rho}(0)} |w(X,\varphi^{(0)}(X))|^2 
			\nonumber \\&\hspace{15mm} + 2 \rho^{-n-2\alpha} \int_{B_{\rho}(0)} |w(X,\varphi^{(0)}(X)) - \psi_{\rho}(X,\varphi^{(0)}(X))|^2 \nonumber \\
		&\leq 4 \rho^{-n-2\alpha} \int_{B_{\rho}(0)} |w(X,\varphi^{(0)}(X))|^2 \leq 4\vartheta^{-n-2\alpha} \nonumber 
\end{align}
for all $\rho \in [\vartheta,1/4]$.  
By \eqref{lemma1_eqn17}, \eqref{lemma1_eqn18} and \eqref{lemma1_eqn19} we have all the necessary estimates to apply Lemma~\ref{lemma4_14} (with $\sigma = 1/2$) to conclude that  
\begin{equation} \label{lemma1_eqn21} 
	\vartheta^{-n-2\alpha} \int_{B_{\vartheta}(0)} |w(X,\varphi^{(0)}(X)) - \psi_{\vartheta}(X,\varphi^{(0)}(X))|^2 
	\leq C \vartheta^{2\mu} \int_{B_1(0)} |w(X,\varphi^{(0)}(X))|^2 \leq C \vartheta^{2\mu}
\end{equation}
for some constants $\mu = \mu(n,m,\alpha,\varphi^{(0)}) \in (0,1)$ and $C = C(n,m,\alpha,\varphi^{(0)}) \in (0,\infty)$, where we used the fact that $\int_{B_1(0)} |w(X,\varphi^{(0)}(X))|^2 = 1$.  Let $\widetilde{\varphi}_k \in \widetilde{\Phi}_{\gamma \varepsilon_k}(\varphi^{(0)})$ be as in \eqref{lemma1_eqn13} corresponding to $\psi = \psi_{\vartheta}$, where by \eqref{lemma1_eqn20} we take $\gamma = C(n,m,\alpha) \,\vartheta^{-n-2\alpha}$.  By \eqref{lemma1_eqn21} and \eqref{lemma1_eqn16} with $\rho = \vartheta$, 
\begin{equation*}
	\vartheta^{-n-2\alpha} \int_{B_{\vartheta}(0)} \mathcal{G}(u_j,\widetilde{\varphi}_j)^2 \leq 2C \vartheta^{2\mu} E_j^2, 
\end{equation*}
for $j$ sufficiently large, which is \eqref{lemma1_eqn1} as required. 
\end{proof}

\section{Proofs of the main results} \label{sec:prove main result sec}

\begin{theorem} \label{theorem1}
Let $\alpha$ and $\varphi^{(0)}$ be as in Definition~\ref{varphi0 defn}.  There exists $\varepsilon, \delta, \overline{\mu} \in (0,1)$ depending only on $n$, $m$, $\alpha$, and $\varphi^{(0))}$ such that if $u \in C^{1,1/2}(B_4(0),\mathcal{A}_2(\mathbb{R}^m))$ is a two-valued function such that the $M = {\rm graph}\,u$ is stationary in $B_4(0) \times \mathbb{R}^{n-2}$, $\|u\|_{C^{1,1/2}(B_4(0))} \leq \varepsilon^2$ and 
\begin{equation} \label{thm1_hyp}
	\int_{B_4(0)} \mathcal{G}\left(\frac{u_s}{4^{-n/2} \|u_s\|_{L^2(B_4(0))}}, \varphi^{(0)} \right)^2 \leq \varepsilon^2 , 
\end{equation}
then 
\begin{equation*}
	\{ P \in \mathcal{K}_M \cap B_1(0) \times \mathbb{R}^m : \mathcal{N}_M(P) \geq \alpha \} \subseteq S \cup T
\end{equation*}
where $S$ is contained in a properly embedded $(n-2)$-dimensional $C^{1,\overline{\mu}}$-submanifold $\Gamma$ of $B_1(0) \times \mathbb{R}^{n-2}$ with $\mathcal{H}^{n-2}(\Gamma) \leq 2\omega_{n-2}$ and $T \subseteq \bigcup_{j=1}^{\infty} B_{\rho_j}(P_j)$ for a countable family of balls $\{B_{\rho_j}(P_j)\}$ with $\sum_{j=1}^{\infty} \rho_j^{n-2} \leq 1-\delta$.  Moreover, for $\mathcal{H}^{n-2}$-a.e.~$P \in S$, there exists a unique non-zero, symmetric, homogeneous degree $\alpha$, cylindrical, harmonic two-valued function $\varphi^{(P)} \in C^{1,1/2}(T_P M,\mathcal{A}_2(T_P M^{\perp}))$ such that  
\begin{equation} \label{thm1 concl2}
	\rho^{-n-2\alpha} \int_{B_{\rho}(0)} \mathcal{G}(\widetilde{u}_{P,s}, \varphi^{(P)})^2 \leq C \rho^{2\overline{\mu}} \|u_s\|_{L^2(B_4(0))}^2
\end{equation}
for all $\rho \in (0,1]$ and some constant $C = C(n,m,\alpha) \in (0,\infty)$, where we let $\widetilde{u}_P \in C^{1,1/2}(\mathbf{B}_1(0, T_P M),$ $\mathcal{A}_2(T_P M^{\perp}))$ be a two-valued function such that $(M-P) \cap \mathbf{C}_1(0, T_P M) = {\rm graph}\,\widetilde{u}_P$ (as in Definition~\ref{tildeu defn}) and $\widetilde{u}_{P,s}$ be the symmetric part of $\widetilde{u}_P$ (as in \eqref{avg and free defn}). 
\end{theorem}

\begin{proof}
Choose $\vartheta \in (0,1/64)$ such that $C \vartheta^{2\mu} < 1/4$, where $C$ and $\mu$ as in Lemma~\ref{lemma1}.  Let $\varepsilon_0$ and $\delta_0$ be as in Lemma~\ref{lemma1}.  Suppose that $u \in C^{1,1/2}(B_4(0),\mathcal{A}_2(\mathbb{R}^m))$ satisfies the hypotheses of Theorem~\ref{theorem1} for $C\varepsilon \leq \varepsilon_0$ where $C = C(n,m,\alpha,\varphi^{(0)}) \in [1,\infty)$ to be chosen.  Define 
\begin{equation*}
	\mathcal{K}_M^* = \{ P \in M \cap B_1(0) \times \mathbb{R}^m : \mathcal{N}_M(P) \geq \alpha \} . 
\end{equation*}
Since the average of $u_a$ satisfies $\|D^2 u_a\|_{C^0(B_1(0))} \leq C(n,m) \,\varepsilon^2$, we can assign to each $P \in \mathcal{K}_M \cap B_1(0) \times \mathbb{R}^m$ a rotation $Q_P$ of $\mathbb{R}^{n+m}$ such that $Q_P (T_P M) = \mathbb{R}^n \times \{0\}$, $\|Q_P - I\| \leq C(n,m) \,\varepsilon$ and 
\begin{equation} \label{thm1_eqn1}
	\|Q_{P_1} - Q_{P_2}\| \leq C(n,m) \,\varepsilon^2 \,|P_1 - P_2|
\end{equation}
for all $P_1,P_2 \in \mathcal{K}_M \cap B_1(0) \times \mathbb{R}^m$.  For each $P \in \mathcal{K}_M \cap B_1(0) \times \mathbb{R}^m$ we let $\widehat{u}_P \in C^{1,1/2}(B_1(\pi_0 P), \mathcal{A}_2(\mathbb{R}^m))$  satisfy \eqref{hatu graph} with $\rho = 1$.

Set $\Lambda = 4^{-n/2} \|u_s\|_{L^2(B_4(0))}$.  Observe that by $\|u\|_{C^{1,1/2}(B_4(0))} \leq \varepsilon^2$, \eqref{thm1_hyp}, Lemma~\ref{rotate lemma} (with $u, \widehat{u}_P$ in place of $u, \widehat{u}$), and Corollary~\ref{lemma3_9}, for each $P \in \mathcal{K}_M^*$ we have that $\|\widehat{u}_P\|_{C^{1,1/2}(B_1(\pi_0 P))} \leq C \varepsilon$ and 
\begin{align} \label{thm1_eqn2} 
	&\int_{B_1(0)} \mathcal{G}\left(\frac{\widehat{u}_{P,s}(Z+X)}{\Lambda}, \varphi^{(0)}(X) \right)^2 dX
	\\&\hspace{15mm} \leq 3 \int_{B_1(Z)} \mathcal{G}\left(\frac{u_s}{\Lambda}, \varphi^{(0)} \right)^2 \nonumber
		+ \frac{3}{\Lambda^2} \int_{B_1(Z)} \mathcal{G}(u_s,\widehat{u}_{P,s})^2 
		\\&\hspace{30mm} + 3 \int_{B_1(Z)} \mathcal{G}(\varphi^{(0)}(X-Z), \varphi^{(0)}(X))^2 \,dX \nonumber 
	\\&\hspace{15mm} \leq 3 \int_{B_1(Z)} \mathcal{G}\left(\frac{u_s}{\Lambda}, \varphi^{(0)} \right)^2 
		+ C \|Q_P - I\|^2 + C \op{dist}^2(Z,\{0\} \times \mathbb{R}^{n-2}) \nonumber 
	\\&\hspace{15mm} 3 \int_{B_4(0)} \mathcal{G}\left(\frac{u_s}{\Lambda}, \varphi^{(0)} \right)^2 + C \|Du\|_{C^0(B_4(0))}^2
		\leq C \varepsilon^2 , \nonumber
\end{align}
where $Z = \pi_0 P$ and $C = C(n,m,\alpha,\varphi^{(0)}) \in (0,\infty)$ are constants. 

Suppose that there exists $P \in \mathcal{K}_M^*$ and $\rho \in [\vartheta,1]$ such that $\rho^{-1} \widehat{u}_P(\pi_0 P + \rho X)$ satisfies alternative (i) of Lemma~\ref{lemma1}.  By Lemma~\ref{lemma2_6}, $\mathcal{K}_M^* \subseteq B^2_{\tau(\varepsilon)}(0) \times \mathbb{R}^{n-2} \times B^m_{\varepsilon}(0)$ for some $\tau(\varepsilon)$ such that $\tau(\varepsilon) \rightarrow 0$ as $\varepsilon \downarrow 0$.  Hence trivially we have the desired conclusion with $S = \emptyset$ and $T = \mathcal{K}_M^*$.  Thus we may assume that $\rho^{-1} \widehat{u}_P(\pi_0 P + \rho X)$ does not satisfy alternative (i) of Lemma~\ref{lemma1} for any $P \in \mathcal{K}_M^*$ and $\rho \in [\vartheta,1]$.  

For each $k \in \{1,2,3,\ldots\} \cup \{\infty\}$ we define $\Upsilon_k$ to be the set of all points $P \in \mathcal{K}_M^*$ such that $\vartheta^{-i} \widehat{u}_P(\pi_0 P + \vartheta^i X)$ does not satisfy alternative (i) of Lemma~\ref{lemma1} for all $0 \leq i \leq k$ and either $k = \infty$ or $\vartheta^{-k-1} \widehat{u}_P(\pi_0 P + \vartheta^{k+1} X)$ satisfies alternative (i) of Lemma~\ref{lemma1}.  We want to show that the conclusion holds true with $S = \Upsilon_{\infty}$ and $T = \mathcal{K}_M^* \setminus \Upsilon_{\infty}$.

Suppose $P \in \Upsilon_{\infty}$.  Set $Z = \pi_0 P$.  By iteratively applying alternative (ii) in Lemma~\ref{lemma1} with $\vartheta^{-i} \widehat{u}_P(Z + \vartheta^i X)$ in place of $u$, we define $\varphi_i \in \widetilde{\Phi}_{\gamma \varepsilon_0}(\varphi^{(0)})$ for $i = 0,1,2,\ldots$ as follows.  Set $\varphi_0 = \varphi^{(0)}$.  For each $i = 1,2,\ldots$, choose $\varphi_i \in \widetilde{\Phi}_{\gamma \varepsilon_0}(\varphi^{(0)})$ such that 
\begin{align} \label{thm1_eqn3}
	&\vartheta^{-(n+2\alpha)i} \int_{B_{\vartheta^i}(0)} \mathcal{G}\left( \frac{\widehat{u}_{P,s}(Z+X)}{\Lambda}, \varphi_i(X) \right)^2 
	\\&\hspace{10mm} \leq \frac{1}{4} \vartheta^{-(n+2\alpha)(i-1)} \int_{B_{\vartheta^{i-1}}(Z)} 
		\mathcal{G}\left( \frac{\widehat{u}_{P,s}(Z+X)}{\Lambda}, \varphi_{i-1} \right)^2 
			+ \frac{1}{4} \vartheta^{(i-1)/2} [D\widehat{u}_P]_{1/2,B_{\vartheta^{i-1}}(Z))}. \nonumber 
\end{align}
In order to apply Lemma~\ref{lemma1} to find $\varphi_{i+1}$, we must to verify that,having found $\varphi_0,\varphi_1,\ldots,\varphi_i$ we have $\varphi_i \in \widetilde{\Phi}_{\varepsilon_0}(\varphi^{(0)})$.  By iterating \eqref{thm1_eqn3} and using \eqref{thm1_eqn2} and $\vartheta < 1/64$, 
\begin{align} \label{thm1_eqn4}
	&\vartheta^{-(n+2\alpha)i} \int_{B_{\vartheta^i}(0)} \mathcal{G}\left( \frac{\widehat{u}_{P,s}(Z+X)}{\Lambda}, \varphi_i \right)^2 
	\\&\hspace{15mm} \leq \frac{1}{4^i} \int_{B_1(0)} \mathcal{G}\left( \frac{\widehat{u}_{P,s}(Z+X)}{\Lambda}, \varphi^{(0)}(X) \right)^2 
		+ \sum_{j=1}^i \frac{\vartheta^{(i-j)/2}}{4^j} [D\widehat{u}_P]_{1/2,B_1(Z))} \nonumber
	\\&\hspace{15mm} \leq \frac{1}{4^i} \int_{B_1(0)} \mathcal{G}\left( \frac{\widehat{u}_{P,s}(Z+X)}{\Lambda}, \varphi^{(0)}(X) \right)^2 
			+ \frac{C \varepsilon^2}{4^i} 
	\leq \frac{C \varepsilon^2}{4^i} , \nonumber
\end{align}
where $C = C(n,m) \in (0,\infty)$.  Thus by the triangle inequality 
\begin{equation} \label{thm1_eqn5}
	\int_{B_1(0)} \mathcal{G}(\varphi_{i-1}, \varphi_i)^2 \leq \frac{C \varepsilon^2}{4^i} 
\end{equation}
for some constant $C = C(n,m) \in (0,\infty)$.  Hence 
\begin{equation*} 
	\int_{B_1(0)} \mathcal{G}(\varphi_i, \varphi^{(0)})^2 \leq C \varepsilon^2 
\end{equation*}
for some constant $C = C(n,m) \in (0,\infty)$.  In particular, using our assumption that $C \varepsilon \leq \varepsilon_0^2$ for some constant $C = C(n,m,\alpha,\varphi^{(0)}) \in [1,\infty)$, $\varphi_i \in \widetilde{\Phi}_{\varepsilon_0}(\varphi^{(0)})$. 

By \eqref{thm1_eqn5}, $(\varphi_i)$ converges in $L^2(B_1(0),\mathcal{A}_2(\mathbb{R}^m))$ to some $\widehat{\varphi}^{(P)} \in \widetilde{\Phi}_{\varepsilon_0}(\varphi^{(0)})$ and 
\begin{equation*}
	\int_{B_1(0)} \mathcal{G}(\varphi_i, \widehat{\varphi}^{(P)})^2 \leq \frac{C \varepsilon^2}{4^i}  
\end{equation*}
for each $i \in \{0,1,2,\ldots\}$ and some constant $C = C(n,m) \in (0,\infty)$.  Thus by \eqref{thm1_eqn4}
\begin{equation*}
	\vartheta^{-(n+2\alpha) i} \int_{B_{\vartheta^i}(0)} \mathcal{G}\left( \frac{\widehat{u}_{P,s}(Z+X)}{\Lambda}, \widehat{\varphi}^{(P)}(X) \right)^2 
		\leq \frac{1}{4^i} C \varepsilon^2 
\end{equation*}
for each $i \in \{0,1,2,\ldots\}$ and some constant $C = C(n,m) \in (0,\infty)$.  Hence given $\rho \in (0,1]$, we can choose $i$ so that $\vartheta^{i-1} < \rho \leq \vartheta^i$ to get 
\begin{equation} \label{thm1_eqn6}
	\rho^{-n-2\alpha} \int_{B_{\rho}(0)} \mathcal{G}\left( \frac{\widehat{u}_{P,s}(Z+X)}{\Lambda}, \widehat{\varphi}^{(P)}(X) \right)^2 
	\leq C \varepsilon^2 \rho^{2\overline{\mu}}, 
\end{equation}
where 
$C = C(n,m,\alpha) \in (0,\infty)$ is a constant.  Set $\varphi^{(P)}(X) = \Lambda \,Q_P^{-1} \widehat{\varphi}^{(P)}(Q_P X)$ for each $X \in T_P M$ to obtain $\varphi^{(P)} \in C^{1,1/2}(T_P M,\mathcal{A}_2(T_P M^{\perp}))$ such that, by \eqref{thm1_eqn6}, \eqref{thm1 concl2} holds true. 

For each $P \in \Upsilon_{\infty}$, since $\widehat{\varphi}^{(P)} \in \widetilde{\Phi}_{C(n,m) \,\varepsilon}(\varphi^{(0)})$, there exists a rotation $q_P$ of $\mathbb{R}^n$ such that $\widehat{\varphi}^{(P)} \circ q_P^{-1} \in \Phi_{C(n,m) \,\varepsilon}(\varphi^{(0)})$ and $q_P = e^A$ for some $n \times n$ skew-symmetric matrix $A = (A_{ij})$ with $A_{ij} = 0$ if $1 \leq i,j \leq 2$, $A_{ij} = 0$ if $3 \leq i,j \leq n$, and $\|A\| \leq \varepsilon$.  In particular, $\|q_P - I\| \leq C(n,m) \,\varepsilon$.  By Corollary~\ref{lemma3_9} and \eqref{thm1_eqn6}. 
\begin{equation} \label{thm1_eqn7}
	\op{dist}(q_P Q_P \pi_P(\mathcal{K}_M^* - P) \cap B_{\rho}(0), \{0\} \times \mathbb{R}^{n-2}) \leq C \varepsilon \rho^{1 + \bar \mu} 
\end{equation}
for all $P \in \Upsilon_{\infty}$ and $\rho \in (0,1/2]$, where $\pi_P$ is the orthogonal projection map onto $T_P M$ and $C = C(n,m,\alpha,\varphi^{(0)}) \in (0,\infty)$.  Let $P_1,P_2 \in \Upsilon_{\infty}$ with $|P_1 - P_2| \leq 1/4$ and set $Z_i = \pi_0(P_i)$ for $i = 1,2$, $\widehat{Z} = \pi_0 Q_{P_1}(P_2-P_1)$, and $\rho = |P_1-P_2|$.  Since $\widehat{u}_{P_1}$ and $\widehat{u}_{P_2}$ satisfy \eqref{hatu graph} with $P = P_1,P_2$ and $\rho = 1$, 
\begin{equation*}
	{\rm graph}\, \widehat{u}_{P_2} - P_2 = Q_{P_2} Q_{P_1}^{-1} ({\rm graph}\, \widehat{u}_{P_1} - P_1 - Q_{P_1}(P_2-P_1)) . 
\end{equation*}
Hence by Lemma~\ref{rotate lemma}, 
\begin{equation} \label{thm1_eqn8}
	\int_{B_{\rho}(\widehat{Z})} \mathcal{G}(\widehat{u}_{P_1,s}(Z_1+X), \widehat{u}_{P_2,s}(Z_2-\widehat{Z}+X))^2 \,dX
		\leq C \|Q_{P_1} - Q_{P_2}\|^2 \|\widehat{u}_{P_1}\|_{L^2(B_{4\rho}(Z_1))}^2  
\end{equation}
for some constant $C = C(n,m) \in (0,\infty)$.  By \eqref{thm1_eqn7}, 
\begin{equation} \label{thm1_eqn9}
	\op{dist}(q_{P_1} \widehat{Z}, \{0\} \times \mathbb{R}^{n-2}) \leq C \varepsilon \rho^{1 + \bar \mu}, 
\end{equation}
where $C = C(n,m,\alpha,\varphi^{(0)}) \in (0,\infty)$.  By \eqref{thm1_eqn1}, \eqref{thm1_eqn6}, \eqref{thm1_eqn8}, and \eqref{thm1_eqn9}
\begin{align*} 
	\int_{B_1(0)} \mathcal{G}(\varphi^{(P_1)}, \varphi^{(P_2)})^2 
	\leq\,& 4\rho^{-n-2\alpha} \int_{B_{\rho}(0)} \mathcal{G}\left(\frac{\widehat{u}_{P_1,s}(Z_1+\widehat{Z}+X)}{\Lambda}, 
		\varphi^{(P_1)}(\widehat{Z}+X) \right)^2 dX 
		\\&+ 4\rho^{-n-2\alpha} \int_{B_{\rho}(0)} \mathcal{G}\left(\frac{\widehat{u}_{P_2,s}(Z_2+X)}{\Lambda}, \varphi^{(P_2)}(X) \right)^2 dX 
		\\&+ \frac{4\rho^{-n-2\alpha}}{\Lambda^2} \int_{B_{\rho}(0)} \mathcal{G}(\widehat{u}_{P_1,s}(Z_1+\widehat{Z}+X), 
			\widehat{u}_{P_2,s}(Z_2+X))^2 \,dX 
		\\&+ 4\rho^{-n-2\alpha} \int_{B_{\rho}(0)} \mathcal{G}(\varphi^{(P_1)}(\widehat{Z}+X), \varphi^{(P_1)}(X))^2 \,dX 
	\\ \leq\,& C \varepsilon^2 \rho^{2\overline{\mu}} + C \|Q_{P_1} - Q_{P_2}\|^2 + C \op{dist}^2(q_{P_1} \widehat{Z}, \{0\} \times \mathbb{R}^{n-2}) 
		\nonumber
	\\ \leq\,& C \varepsilon^2 \rho^{2\overline{\mu}} \nonumber
\end{align*}
where $C = C(n,m,\alpha,\varphi^{(0)}) \in (0,\infty)$.
Thus   
\begin{equation} \label{thm1_eqn10}
	\|q_{P_1} - q_{P_2}\| \leq C \varepsilon |P_1 - P_2|^{\overline{\mu}} 
\end{equation}
for all $P_1,P_2 \in \Upsilon_{\infty}$ and some constant $C = C(n,m,\alpha,\varphi^{(0)}) \in (0,\infty)$.  By $\|u\|_{C^{1,1/2}(B_4(0))} \leq \varepsilon$, \eqref{thm1_eqn1}, \eqref{thm1_eqn7}, and \eqref{thm1_eqn10}, $\Upsilon_{\infty} \subseteq \op{graph} f \cap B_1(0) \times \mathbb{R}^m$ is the graph of a function $f \in C^{1,\bar \mu}(B^{n-2}_1(0),\mathbb{R}^{2+m})$ such that $\|f\|_{C^{1,\bar \mu}(B^{n-2}_{1/2}(0))} \leq C(n,m) \,\varepsilon$. 

Suppose instead that $P \in \Upsilon_k$ for some $1 \leq k < \infty$.  Now iteratively applying alternative (ii) in Lemma~\ref{lemma1} only gives us $\varphi_i \in \widetilde{\Phi}_{\gamma \varepsilon_0}(\varphi^{(0)})$ only gives us $\varphi_i \in \widetilde{\Phi}_{\gamma \varepsilon_0}(\varphi^{(0)})$ for $i = 0,1,2,\ldots,k+1$ such that $\varphi_0 = \varphi^{(0)}$ and \eqref{thm1_eqn3} holds.  Take $\widehat{\varphi}^{(P)} = \varphi_{k+1}$.  Let $q_P$ be a rotation of $\mathbb{R}^n$ such that $\varphi^{(P)} \circ q_P^{-1} \in \Phi_{\varepsilon_0}(\varphi^{(0)})$ and $|q_P - I| \leq C\bar \varepsilon$.  By the argument above, \eqref{thm1_eqn7} only holds true for all $\rho \in [\vartheta^{k+1},1/2]$.  Hence 

\begin{equation} \label{thm1_eqn11}
	\op{dist}(\mathcal{K}_M^* \cap B^{n+m}_{\rho}(P), P + \{0\} \times \mathbb{R}^{n-2} \times \{0\}^m) \leq C \varepsilon \rho . 
\end{equation}
for every $P \in \Upsilon_k$ and $\rho \in [\vartheta^{k+1},1/2]$, where $C = C(n,m,\alpha,\varphi^{(0)}) \in (0,\infty)$.  By the definition of $\Upsilon_k$, $\vartheta^{-k-1} \widehat{u}_P(\pi_0 P + \vartheta^{k+1} X)$ satisfies alternative (i) in Lemma~\ref{lemma1}, i.e.
\begin{equation*}
	\forall\,P \in \Upsilon_k, \, \exists\,Y \in B^{n-2}_{\vartheta^{k+1}/2}(0) \text{ such that } 
	\pi_0 Q_P(\mathcal{K}_M^* - P) \cap B_{\delta_0 \vartheta^{k+1}}(Y) \cap B_{\vartheta^{k+1}/2}(0) = \emptyset. 
\end{equation*}
Thus since $\|u\|_{C^{1,1/2}(B_4(0))} \leq \varepsilon$, provided $\varepsilon$ is sufficiently small  
\begin{align} \label{thm1_eqn12}
	&\forall\,P \in \Upsilon_k, \, \exists\, Y' \in P + \{0\}^2 \times B^{n-2}_{\vartheta^{k+1}/2}(0) \times \{0\}^m \\&\hspace{15mm} \text{such that } 
	\mathcal{K}_M^* \cap B^{n+m}_{\delta_0 \vartheta^{k+1}/2}(Y') \cap B^{n+m}_{\vartheta^{k+1}/2}(P) = \emptyset. \nonumber 
\end{align}
Now arguing as in the proof of Theorem 1 in~\cite{Sim93}, using \eqref{thm1_eqn11} and \eqref{thm1_eqn12} in place of (12) and (13) on p.~642 of~\cite{Sim93}, we obtain a covering of $\bigcup_{1 \leq j < \infty} \Upsilon_j$ by a finite collection of balls $B^{n+m}_{\rho_j}(X_j)$, $j = 1,\ldots,N$, such that $\sum_{j=1}^N \rho_j^{n-2} \leq 1-\delta_0$. 
\end{proof}

\begin{proof}[Proof of Theorem~\ref{theorem2}]
Similar to the proof of Theorem 2' in~\cite{Sim93}, so we will only sketch the proof here.  Let $M$ be a $C^{1,1/2}$ two-valued stationary graph as in Definition~\ref{twoval minimal graph}.  Since $\op{dim} \mathcal{K}^{(n-3)}_M \leq n-3$, it suffices to consider the set $\mathcal{K}^*_{M,\alpha}$ of points $P \in \mathcal{K}_M$ at which $M$ has a homogeneous degree $\alpha$ cylindrical tangent function.  Consider any $P \in \mathcal{K}^*_{M,\alpha}$.  Let $\varphi^{(0)} \in C^{1,1/2}(T_P M, T_P M^{\perp})$ be a tangent function of $M$ at $P$.  By the definition of tangent function and monotonicity formula for frequency functions, for every $\varepsilon > 0$ there exists $\sigma > 0$ and a two-valued function $\widetilde{u}_P \in C^{1,1/2}(\mathbf{B}_{\sigma R(\varepsilon)}(0,T_P M), \mathcal{A}_2(T_P M^{\perp}))$ such that 
\begin{gather*}
	(M - P) \cap \mathbf{C}_{\sigma R(\varepsilon)}(0,T_P M) = {\rm graph}\, \widetilde{u}_P, \\
	(\sigma R(\varepsilon))^{1/2} \|\widetilde{u}_P\|_{C^{1,1/2}(\mathbf{B}_{\sigma R(\varepsilon)}(0, T_P M))} < \delta(\varepsilon), \\ 
	N_{M,P_0}(\sigma R(\varepsilon)) - \alpha < \delta(\varepsilon), \\ 
	\int_{\mathbf{B}_1(0,T_P M)} \mathcal{G}\left( \frac{\widetilde{u}_{P_0,s}}{(\sigma R(\varepsilon))^{-n/2} \|\widetilde{u}_{P_0,s}\|_{
		L^2(\mathbf{B}_{\sigma R(\varepsilon)}(0,T_P M))}} , \varphi^{(0)} \right)^2 < \delta(\varepsilon), 
\end{gather*}
where $R(\varepsilon)$ and $\delta(\varepsilon)$ are as in Lemma~\ref{lemma2_4} and $\widetilde{u}_{P,s}$ is the symmetric part of $\widetilde{u}_P$ (as in \eqref{avg and free defn}).  Let $Q_P$ be a rotation such that $Q_P(T_P M) = \mathbb{R}^n \times \{0\}$ and set $M_0 = \sigma^{-1} \,Q_P (M - P)$.  
Set 
\begin{equation*}
	\mathcal{K}^+_{\alpha} = \{ X \in \mathcal{K}_{M_0} \cap \overline{B_1(0)} \times \mathbb{R}^m : \mathcal{N}_{M_0}(X) \geq \alpha \} . 
\end{equation*}
Given $\rho_0 \in (0,1/2]$, define the outer measure $\mu_{\rho_0}$ on $B_1(0)$ by 
\begin{equation*}
	\mu_{\rho_0}(A) = \inf \sum_{i=1}^N \omega_{n-2} \sigma_i^{n-2}
\end{equation*}
for every set $A \subseteq B_1(0)$, where the infimum is taken over all finite covers of $A$ by open balls $B_{\sigma_i}(Y_i)$, $i = 1,2,\ldots,N$, with $\sigma_i \leq \rho_0$.  Choose a cover of $\mathcal{K}^+_{\alpha}$ by a finite collection of open balls $B_{\sigma_i}(Y_i)$ such that 
\begin{equation*}
	\sum_{i=1}^N \omega_{n-2} \sigma_i^{n-2} \leq \mu_{\rho_0}(\mathcal{K}^+_{\alpha}) + 1. 
\end{equation*}
Remove the balls $B_{\sigma_i}(Y_i)$ that do not intersect $\mathcal{K}^+_{\alpha}$ from the collection and for the remaining balls $B_{\sigma_i}(Y_i)$ select $P_i \in \mathcal{K}^+_{\alpha} \cap B_{\sigma_i}(Y_i)$.  By Lemma~\ref{lemma2_4}, for $\varepsilon > 0$ sufficiently small, either:
\begin{enumerate}
	\item[(a)] there is a non-zero, cylindrical, homogeneous degree $\alpha$ two-valued function $\varphi_i \in C^{1,1/2}(T_{P_i} M_0,$ $T_{P_i} M_0^{\perp})$ such that 
\begin{equation} \label{thm2_eqn2}
	\int_{\mathbf{B}_1(0,T_{P_i} M_0)} \mathcal{G}\left( \frac{\widetilde{u}_{M_0,P_i,s}}{(2\sigma_i)^{-n/2} \|\widetilde{u}_{M_0,P_i,s}\|_{
		L^2(\mathbf{B}_{2\sigma_i}(0,T_{P_i} M_0))}}, \varphi^{(0)} \right)^2 < \varepsilon^2 , 
\end{equation}
where $\widetilde{u}_{M_0,P_i} \in C^{1,1/2}(\mathbf{B}_1(0,T_{P_i} M_0), \mathcal{A}_2(T_{P_i} M_0^{\perp}))$ is a two-valued function such that $(M_0 - P_i) \cap \mathbf{C}_2(0,T_{P_i} M_0) = {\rm graph}\, \widetilde{u}_{M_0,P_i}$ (as in Definition~\ref{tildeu defn}) and $\widetilde{u}_{M_0,P_i,s}$ is the symmetric part of $\widetilde{u}_{M_0,P_i}$ (as in \eqref{avg and free defn}), or 

	\item[(b)] there exists an $(n-3)$-dimensional linear subspace $L \subset T_{P_i} M_0$ such that 
\begin{align} \label{thm2_eqn3} 
	&\pi_{T_{P_i} M_0} \{ P \in (\mathcal{K}_{M_0} - P_i) \cap \overline{B^{n+m}_{2\sigma_i}(0)} : \mathcal{N}_{M_0}(P) \geq \alpha \} 
	\\&\hspace{30mm} \subset \{ X \in T_{P_i} M_0 : \op{dist}(X,Y_i+L) < 2 \varepsilon \sigma_i \} , \nonumber 
\end{align}
where $\pi_{T_{P_i} M_0}$ is the orthogonal projection onto $T_{P_i} M_0$. 
\end{enumerate}
Note that this application of Lemma~\ref{lemma2_4} uses the fact that the set of all degrees $\alpha$ of non-zero cylindrical two-valued functions is the discrete.  Using the fact that either \eqref{thm2_eqn2} or \eqref{thm2_eqn3} hold true, we iteratively apply Theorem~\ref{theorem1} much like in the argument in~\cite{Sim93} to conclude that $\mathcal{K}^+_{\alpha} \cap B^{n+m}_{\sigma_i}(P_i)$ is countably $(n-2)$-rectifiable with bounded $\mathcal{H}^{n-2}$-measure.
\end{proof}

\begin{proof}[Proof of Theorem~A  of the Introduction] 
Let $M$ be a $C^{1,\mu}$ two-valued stationary graph as in Definition~\ref{twoval minimal graph}.  By Theorem~\ref{theorem2}(b) we have that for $\mathcal{H}^{n-2}$-a.e.~$P \in \mathcal{K}_M$ there exists a unique, non-zero, symmetric, cylindrical, homogeneous degree, harmonic two-valued $\varphi^{(P)} \in C^{1,1/2}(T_P M, \mathcal{A}_2(T_P M^{\perp}))$, a number $\rho_P \in (0,1/2)$ and a two-valued function $\widetilde{u}_P \in C^{1,1/2}(\mathbf{B}_{\rho_P}(0,T_P M), \mathcal{A}_2(T_P M^{\perp}))$ such that \eqref{theorem2 conclb1} and \eqref{theorem2 conclb2} hold true.  Set $h_P = \widetilde{u}_{P,a}$ and note that, by~\cite{SimWic16}, $h_P$ is a $C^{1,1}$ single-valued function.  This choice of $h_P$ and \eqref{theorem2 conclb2} give us the desired asymptotic expansion for $\widetilde{u}_P$. 
\end{proof}

\begin{proof}[Proof of Theorem~A$^\prime$ of  the Introduction] 
Let $\mu \in (0,1]$ and $\varepsilon = \varepsilon(n,m,\mu) > 0$ to be later determined.  Suppose $u \in C^{1,\mu}(B_1(0),\mathcal{A}_2(\mathbb{R}^m))$ is a two-valued function such that $\|Du\|_{C^{0,\mu}(B_1(0))} \leq \varepsilon$ and $M$ is stationary in $B_1(0) \times \mathbb{R}^m$.  Set $h = u_a$, the average of the values of $u$ and note that, by~\cite{SimWic16}, $u_a$ is a $C^{1,1}$ single-valued function.  For each $Z \in \mathcal{K}_u \cap B_{1/2}(0)$, let $Q_Z$ be a rotation of $\mathbb{R}^{n+m}$ such that $Q_Z(T_{(Z,u_1(Z))} M) = \mathbb{R}^n \times \{0\}$ and $\|Q_Z - I\| \leq C(n,m) \,|Du(Z)|$.  Let $\widehat{u}_Z \in C^{1,1/2}(B_{1/4}(Z),\mathcal{A}_2(\mathbb{R}^m))$ such that 
\begin{equation*}
	{\rm graph}\,\widehat{u}_Z = (Z,u_1(Z)) + (Q_Z (M - (Z,u_1(Z)))) \cap B_{1/4}(Z) \times \mathbb{R}^m
\end{equation*}
where $u(Z) = \{u_1(Z),u_1(Z)\}$ for $u_1(Z) \in \mathbb{R}^m$ (as in Definition~\ref{hatu defn}) and let $\widehat{u}_{Z,s}$ be the symmetric part of $\widehat{u}_Z$ (as in \eqref{avg and free defn}).  By Theorem~A, for $\mathcal{H}^{n-2}$-a.e.~$Z \in \mathcal{K}_u \cap B_{1/2}(0)$ there exists a symmetric, homogeneous, cylindrical, harmonic two-valued function $\varphi^{(Z)} \in C^{1,1/2}(\mathbb{R}^n, \mathcal{A}_2(\mathbb{R}^m))$ and $\rho_Z \in (0,1/16]$ such that 
\begin{equation} \label{thmA eqn1}
	\rho^{-n-2\alpha} \int_{B_{\rho}(0)} \mathcal{G}(\widehat{u}_{Z,s}(X-Z), \varphi^{(Z)}(X))^2 \,dX \leq \overline{C}_Z \rho^{2\overline{\gamma}_Z} 
\end{equation}
for all $\rho \in (0,2\rho_Z]$ and some constants $\overline{\gamma}_Z \in (0,1)$ and $\overline{C}_Z \in (0,\infty)$.  Fix such a point $Z$.  We want to show that $u_s = u - h$ satisfies the asymptotic expansion as in Theorem~A$^\prime$.  We will do this by arguing as in the proof of Lemma~\ref{rotate lemma} with $\widehat{u}_Z, u$ in place of $\hat{u}, u$.

To simplify notation, translate $(Z,u_1(Z))$ to the origin and set $Q = Q_Z$ and $\widehat{u} = \widehat{u}_Z$.  Arguing like in the proof of Lemma~\ref{rotate lemma}, express $Q^{-1}$ (not $Q$) as the block matrix 
\begin{equation*}
	Q^{-1} = \left(\begin{matrix} Q_{11} & Q_{12} \\ Q_{21} & Q_{22} \end{matrix}\right) 
\end{equation*}
where $Q_{11}$, $Q_{12}$, $Q_{21}$ and $Q_{22}$ are $n \times n$, $n \times m$, $m \times n$ and $m \times m$ matrices.  

Suppose that $X \in B_{\rho_Z}(0) \setminus \mathcal{Z}_u$.  Set $d = d(X) = \op{dist}(X,\mathcal{Z}_u)$ and note that since $0 \in \mathcal{Z}_u$, $d \leq |X|$.  Let $u(X) = \{u_1(X),u_2(X)\}$ in $B_d(X)$ for single-valued functions $u_1,u_2 \in C^{1,1/2}(B_d(X);\mathbb{R}^m)$.  For each $l = 1,2$ define $\xi_l : B_d(X) \rightarrow \mathbb{R}^n$ by 
\begin{equation} \label{thmA eqn2} 
	\xi_l(Y) = \pi_0 Q (Y,u_l(Y))
\end{equation}
for all $Y \in B_d(X)$.  For $l = 1,2$ let $\widehat{u}_1,\widehat{u}_2 : \xi_l(B_d(X)) \rightarrow \mathbb{R}^m$ be functions such that 
\begin{equation} \label{thmA eqn3} 
	\left(\begin{matrix} Y \\ u_l(Y) \end{matrix}\right) 
		= Q^{-1} \left(\begin{matrix} \xi_l(Y) \\ \widehat{u}_l(\xi_l(Y)) \end{matrix}\right) 
		= \left(\begin{matrix} Q_{11} \,\xi_l(Y) + Q_{12} \,\widehat{u}_l(\xi_l(Y)) \\ 
			Q_{21} \,\xi_l(Y) + Q_{22} \,\widehat{u}_l(\xi_l(Y)) \end{matrix}\right) 
\end{equation}
for all $Y \in B_d(X)$.  By \eqref{minimal avg schauder}, \eqref{minimal sym schauder} and \eqref{L2_decay}, 
\begin{align}
	\label{thmA eqn4} \sup_{B_{|X|}(X)} |\widehat{u}| + |X| \sup_{B_{|X|}(X)} |D\widehat{u}| 
		&\leq \sup_{B_{3|X|/2}(0)} |\widehat{u}| + |X| \sup_{B_{3|X|/2}(0)} |D\widehat{u}|
		\leq C(n,m) \,\varepsilon \,|X|^{3/2}, \\
	\label{thmA eqn5} \sup_{B_{|X|}(X)} |\widehat{u}_s| + |X| \sup_{B_{|X|}(X)} |D\widehat{u}_s| 
		&\leq \sup_{B_{3|X|/2}(0)} |\widehat{u}_s| + |X| \sup_{B_{3|X|/2}(0)} |D\widehat{u}_s|
		 \\&\leq C(n,m) \,|X|^{-n/2} \|\widehat{u}_s\|_{L^2(B_{2|X|}(0)} \leq C(n,m) \,\varepsilon \,|X|^{\alpha}. \nonumber 
\end{align}
By taking differences in \eqref{thmA eqn3} over $l = 1,2$ we compute that  
\begin{align}
	\label{thmA eqn6} \frac{\xi_1(Y) - \xi_2(Y)}{2} &= -Q_{11}^{-1} \,Q_{12} \,\frac{\widehat{u}_1(\xi_1(Y)) - \widehat{u}_2(\xi_2(Y))}{2}, \\
	\label{thmA eqn7} \frac{u_1(Y) - u_2(Y)}{2} &= (Q_{22} - Q_{21} \,Q_{11}^{-1} \,Q_{12}) \,\frac{\widehat{u}_1(\xi_1(Y)) - \widehat{u}_2(\xi_2(Y))}{2} 
\end{align}
(see \eqref{rotate eqn7} and \eqref{rotate eqn8} in the proof of Lemma~\ref{rotate lemma}).  Provided $\varepsilon$ is sufficiently small, it follows from \eqref{thmA eqn6} using $\|Q - I\| \leq C(n,m) \,\varepsilon$ and $\|D\widehat{u}\|_{C^0(B_{4\rho_Z}(0))} \leq C(n,m) \,\varepsilon$ that 
\begin{equation} \label{thmA eqn8}
	\frac{|\xi_1(X) - \xi_2(X)|}{2} \leq C(n,m) \,\varepsilon \,|\widehat{u}_s(\xi_a(X))| \leq C(n,m) \,\varepsilon \,|X|^{\alpha} , 
\end{equation}
where $\xi_a(X) = (\xi_1(X) + \xi_2(X))/2$ and the last step follows by \eqref{thmA eqn5}. 
Hence by \eqref{thmA eqn7}, \eqref{thmA eqn4} and \eqref{thmA eqn8}, 
\begin{equation} \label{thmA eqn9}
	\mathcal{G}(u_s(X), B \widehat{u}_s(\xi_a(X))) 
		\leq |B| \sup_{B_{|X|}(X)} |D\widehat{u}| \,\frac{|\xi_1(X) - \xi_2(X)|}{2} \leq C(n,m) \,\varepsilon \,|X|^{\alpha+1/2} 
\end{equation}
for all $X \in B_{\rho_Z}(0) \setminus \mathcal{Z}_u$, where we set $B = Q_{22} - Q_{21} \,Q_{11}^{-1} \,Q_{12}$.  Notice that \eqref{thmA eqn9} trivially holds true for all $X \in B_{\rho_Z}(0) \cap \mathcal{Z}_u$, as then $u_s(X) = \widehat{u}_s(\xi_a(X)) = \{0,0\}$. 

By summing \eqref{thmA eqn3} over $l = 1,2$ we compute that  
\begin{equation} \label{thmA eqn10} 
	\xi_a(X) = (Q_{11} - Q_{12} \,Q_{22}^{-1} \,Q_{21})^{-1} (X - Q_{12} \,Q_{22}^{-1} \,u_a(X))
\end{equation}
for all $X \in B_{\rho_Z}(0)$ (see \eqref{rotate eqn13} in the proof of Lemma~\ref{rotate lemma}).  Thus $\xi_a \in C^{1,1}(B_{\rho_Z}(0), \mathbb{R}^n)$.  
By \eqref{thmA eqn3}, \eqref{thmA eqn4} and $\|Q - I\| \leq C(n,m) \,\varepsilon$, 
\begin{equation} \label{thmA eqn11} 
	|X - Q_{11} \xi_a(X)| 
		\leq C(n,m) \,\varepsilon \sup_{B_{|X|}(X)} |\widehat{u}| \leq C(n,m) \,\varepsilon \,|X|^{3/2}
\end{equation}
By \eqref{thmA eqn9}, \eqref{thmA eqn5} and \eqref{thmA eqn11}, 
\begin{align} \label{thmA eqn12}
	\mathcal{G}(u_s(X), B \widehat{u}_s(Q_{11}^{-1} X)) 
		&\leq |B| \,\mathcal{G}(\widehat{u}_s(\xi_a(X)), \widehat{u}_s(Q_{11}^{-1} X)) + C(n,m) \,\varepsilon \,|X|^{\alpha+1/2} \\
		&\leq |B| \,\sup_{B_{3|X|/2}(0)} |D\widehat{u}_s| \,|\xi_a(X) - Q_{11}^{-1} X| + C(n,m) \,\varepsilon \,|X|^{\alpha+1/2} \nonumber \\
		&\leq C(n,m) \,\varepsilon \,|X|^{\alpha+1/2} \nonumber 
\end{align}
for all $X \in B_{\rho_Z}(0)$.  Therefore, assuming $\varepsilon = \varepsilon(n,m)$ is sufficiently small, by \eqref{thmA eqn1} and \eqref{thmA eqn12}, 
\begin{equation*}
	\rho^{-n-2\alpha} \int_{B_{\rho}(0)} \mathcal{G}(u_s(X-Z), B \varphi^{(Z)}(Q_{11}^{-1} X))^2 \,dX \leq C_Z \rho^{2\gamma_Z} 
\end{equation*}
for all $\rho \in (0,\rho_Z]$ where $\gamma_Z = \min\{1/2,\overline{\gamma}_Z\}$ and $C_Z \in (0,\infty)$ is a constant.  This gives us the desired conclusion of Theorem~A$^{\prime}$ with $A_Z = Q_{11}^{-1}$ and $B \varphi^{(Z)}$ in place of $\varphi^{(Z)}$. 
\end{proof}

\begin{proof}[Proof of Theorem~B of the Introduction]
Let $\varphi^{(0)}$ be as in Definition~\ref{varphi0 defn} with $\alpha = 3/2$.  We claim that for every $\delta \in (0,1/2)$ there exists $\varepsilon = \varepsilon(n,m,\varphi^{(0)},\delta) \in (0,1/2)$ such that if $(u,\Lambda) \in \mathcal{F}_{\varepsilon}(\varphi^{(0)})$, then $\{ X \in \mathcal{K}_u \cap B_{1/2}(0) : \mathcal{N}_M(X,u_1(X)) \geq 3/2 \} \cap B_{\delta}(0,y) \neq \emptyset$ for each $y \in B^{n-2}_1(0)$.  To see this, first notice that since $u_s/\Lambda \rightarrow \varphi^{(0)}$ in $C^1(B_{1/2}(0),\mathcal{A}_2(\mathbb{R}^m))$, we can choose $\varepsilon > 0$ small enough that $\mathcal{B}_u \cap B_{1/2}(0) \cap B_{\delta}(0,y) \neq \emptyset$ for each $y \in B^{n-2}_1(0)$.  By Lemma~\ref{freq_mono_thm}(iii), $\mathcal{N}_M(Z,u_1(Z)) \geq 3/2$ whenever $Z \in \mathcal{B}_u \cap B_{1/2}(0) \cap B_{\delta}(0,y)$ for some $y \in B^{n-2}_1(0)$. 

Now let $M$ be a $C^{1,1/2}$ two-valued stationary graph (as in Definition~\ref{twoval minimal graph}) and $P_0 \in \mathcal{K}_M$ with $\mathcal{N}_M(P_0) = 3/2$.  Let $\varphi^{(0)}$ be a tangent function to $M$ at $P_0$.  By translating and rotating, we may assume that $P_0 = 0$ and $T_{P_0} M = \mathbb{R}^n \times \{0\}$.  In the claim in the preceding paragraph, take $\delta = \delta(n,m,\varphi^{(0)})$ as in Lemma~\ref{lemma1} and then choose $\varepsilon = \varepsilon(n,m,\varphi^{(0)})$ sufficiently small that Lemma~\ref{lemma1} and the above claim both apply.  After scaling, we may take $M$ to be the graph of a two-valued function $u \in C^{1,1/2}(B_4(0),\mathcal{A}_2(\mathbb{R}^m))$ with $u(0) = \{0,0\}$, $Du(0) = \{0,0\}$, $[Du]_{1/2,B_4(0)} < \varepsilon$ and $\int_{B_4(0)} \mathcal{G}(u/\|u\|_{L^2(B_4(0))}, \varphi^{(0)})^2 < \varepsilon^2$.  Now we can argue as in the proof of Theorem~\ref{theorem1}, using the claim in the preceding paragraph to rule out alternative (i) of Lemma~\ref{lemma1} in each application of Lemma~\ref{lemma1} in the argument.  We conclude that $\op{sing} M \cap B_1(0) \times \mathbb{R}^m$ is a $C^{1,\alpha}$ $(n-2)$-dimensional submanifold for some $\alpha \in (0,1)$ and that the asymptotic expansion in Theorem~A holds true for all $P \in \op{sing} M \cap B_1(0)$ with $k_P = 3$, $\rho_P = 1$ and $C_P \in (0,\infty)$ and $\gamma_P = \gamma \in (0,1)$ depending only on $n$, $m$, and $\varphi^{(0)}$. 

It only remains to establish the sup estimate on $\{\pm \epsilon^{(P)} \}$.  Fix $P \in \op{sing} M \cap B^{n+m}_1(0)$ and $\sigma \in (0,1/2]$.  Let $Q_P$ be a rotation of $\mathbb{R}^{n+m}$ such that $Q_P (T_P M) = \mathbb{R}^n \times \{0\}$ and $\|Q_P - I\| \leq C(n,m) \,\varepsilon$ and let $\widehat{u} = \widehat{u}_P \in C^{1,1/2}(B_2(0),\mathcal{A}_2(\mathbb{R}^m))$ satisfy \eqref{hatu graph} with $\rho = 2$.  

Let $\widehat{\varepsilon}(X) = Q_P \varepsilon^{(P)}(Q_P^{-1} X)$ for all $X \in B_1(0)$ where $\varepsilon^{(P)}$ is as in Theorem~A.  It suffices to prove a sup estimate for $\{ \pm \widehat{\varepsilon}\}$.  By Corollary~\ref{lemma3_9} and the asymptotic expansion in Theorem~A, 
\begin{equation*} 
	\mathcal{B}_{\widehat{u}} \cap B_{\sigma/2}(0) \subset \{ X : \op{dist}(X, \{0\} \times \mathbb{R}^{n-2}) \leq C \sigma^{1+\gamma/2} \}
\end{equation*}
for some constant $C = C(n,m,\varphi^{(0)}) \in (0,\infty)$.  Hence if $X \in B_{\sigma/4}(0)$ with $\op{dist}(X, \{0\} \times \mathbb{R}^{n-2}) \geq C \sigma^{1+\gamma/n} \equiv \tau$, then $\mathcal{B}_{\widehat{u}} \cap B_{\tau/2}(X) = \emptyset$.  Thus by standard elliptic estimates and the asymptotic expansion in Theorem~A, 
\begin{equation} \label{thmB eqn2}
	|\widehat{\varepsilon}(X)|^2 \leq C \tau^{-n} \int_{B_{\tau/2}(X)} |\widehat{\varepsilon}|^2 
		\leq C \left(\frac{\sigma}{\tau}\right)^n \sigma^{-n} \int_{B_{\sigma}(0)} |\widehat{\varepsilon}|^2 
		\leq C \left(\frac{\sigma}{\tau}\right)^n \sigma^{3+2\gamma} = C \sigma^{3 + \gamma} 
\end{equation}
where $C = C(n,m,\varphi^{(0)}) \in (0,\infty)$.  If instead $X \in B_{\sigma/4}(0)$ with $\op{dist}(X, \{0\} \times \mathbb{R}^{n-2}) < \tau$, set $r = \tfrac{1}{2} \op{dist}(X,\mathcal{B}_{\widehat{u}})$.  By Theorem~A, the triangle inequality and standard elliptic estimates 
\begin{equation} \label{thmB eqn3}
	|\widehat{\varepsilon}(X)|^2 \leq 2 |\widehat{u}_s(X)|^2 + C_1 \op{dist}^3(X,\{0\} \times \mathbb{R}^{n-2}) 
		\leq C r^{-n} \int_{B_r(X)} |\widehat{u}_s|^2 + C_1 \tau^3 
\end{equation}
where $\widehat{u}_s$ is the symmetric part of $\widehat{u}$ (as in \eqref{avg and free defn}) and $C_1 = C_1(\varphi^{(0)}) \in (0,\infty)$ and $C = C(n,m) \in (0,\infty)$ are constants.  Select $Y \in \mathcal{B}_{\widehat{u}}$ with $|X-Y| = \op{dist}(X,\mathcal{B}_{\widehat{u}})$.  Let $\widehat{M} = {\rm graph}\,\widehat{u}$.  Let $Q_Y$ be a rotation of $\mathbb{R}^{n+m}$ such that $Q_Y (T_{(Y,\widehat{u}_1(Y))} \widehat{M}) = \mathbb{R}^n \times \{0\}$ and $\|Q_Y - I\| \leq C(n,m) \,\varepsilon$, where $\widehat{u}(Y) = \{\widehat{u}_1(Y),\widehat{u}_1(Y)\}$ for some $\widehat{u}_1(Y) \in \mathbb{R}^m$, and let $\widehat{u}_Y \in C^{1,1/2}(B_{1/2}(0), \mathcal{A}_2(\mathbb{R}^m))$ such that 
\begin{equation*}
	{\rm graph}\, \widehat{u}_Y = (Y,\widehat{u}_1(Y)) + Q_Y (\widehat{M} - (Y,\widehat{u}_1(Y))) \cap B_{1/2}(0) \times \mathbb{R}^m
\end{equation*}
(as in Definition~\ref{hatu defn}).  By applying Lemma~\ref{rotate lemma} twice and using \eqref{L2_decay}, 
\begin{align} \label{thmB eqn4}
	r^{-n} \int_{B_r(X)} |\widehat{u}_s|^2 &\leq r^{-n} \int_{B_{3r}(Y)} |\widehat{u}_s|^2 \leq 4r^{-n} \int_{B_{6r}(Y)} |\widehat{u}_{Y,s}|^2 
		\\&\leq C r^3 \int_{B_{1/2}(Y)} |\widehat{u}_{Y,s}|^2 \leq 4C r^3 \int_{B_1(Y)} |\widehat{u}_s|^2 \leq 4C r^3 \int_{B_2(0)} |\widehat{u}_s|^2 , \nonumber 
\end{align}
where $\widehat{u}_{Y,s}$ denotes the symmetric part of $\widehat{u}_Y$ (as in \eqref{avg and free defn}) and $C = C(n,m) \in (0,\infty)$ is a constant.  Thus by \eqref{thmB eqn3} and \eqref{thmB eqn4}, 
\begin{equation} \label{thmB eqn5}
	|\widehat{\varepsilon}(X)|^2 \leq C \tau^3 = C \sigma^{3+3\gamma/n}
\end{equation}
where $C = C(n,m,\varphi^{(0)}) \in (0,\infty)$ (and in particular $C$ is independent of $Z$ and $\sigma$).  By \eqref{thmB eqn2} and \eqref{thmB eqn5} we obtain $|\widehat{\varepsilon}(X)| \leq C \sigma^{3/2 + \gamma/(2n)}$ for some constant $C = C(n,m,\varphi^{(0)}) \in (0,\infty)$. 
\end{proof}

\begin{proof}[Proof of Theorem~C of the Introduction]
Let $M$ be the stationary graph of a two-valued function $u \in C^{1,\mu}(\Omega,\mathcal{A}_2(\mathbb{R}^m))$.  Let $B \subseteq \Omega$ be a closed ball.  Observe that if $\mathcal{H}^{n-2}(\op{sing} M \cap (B \times \mathbb{R}^m)) = 0$, then $B \setminus \mathcal{B}_u$ is simply connected (see the appendix of~\cite{SimWic16}).  Since locally in $B \setminus \mathcal{B}_u$, $u$ decomposes into a pair of smooth single-valued functions, $u$ decomposes into a pair of smooth single-valued functions in $B$.  Hence $\mathcal{B}_u \cap B = \emptyset$, that is $\op{sing} M \cap (B \times \mathbb{R}^m) = \emptyset$.

Suppose $B \subseteq \Omega$ be a closed ball with $\mathcal{H}^{n-2}(\op{sing} M \cap (B \times \mathbb{R}^m)) > 0$.  By Theorem~\ref{theorem2}, there exists a finite set $\{\alpha_1,\alpha_2,\ldots,\alpha_k\} \subset \{3/2,2,5/2,\ldots\}$ such that if $P \in \mathcal{K}_M \cap (B \times \mathbb{R}^m)$ and $M$ has a cylindrical tangent function at $P$, then $\mathcal{N}_M(P) = \alpha_j$ for some $j \in \{1,2,\ldots,k\}$.  Arrange the values of $\alpha_j$ so that $\alpha_j < \alpha_{j+1}$ for all $1 \leq j < k$.  For each $j \in \{1,2,\ldots,k\}$ let $V_j = V_{\alpha_j}$, where $V_{\alpha}$ is the open set given in Theorem~\ref{theorem2}, so that 
\begin{equation*}
	V_j \supset \{ P \in \mathcal{K}_M : \mathcal{N}_M(P) = \alpha_j \text{ and $M$ has a cylindrical tangent function at $P$} \}.
\end{equation*}  
Set $\alpha_0 = 3/2$ and $\alpha_{k+1} = \infty$ and for $j \in \{1,2,\ldots,k\}$ let 
\begin{align*}
	\Gamma_j &= \{ P \in \mathcal{K}_M \cap (B \times \mathbb{R}^m) : \alpha_j \leq \mathcal{N}_M(P) < \alpha_{j+1} \} \cap V_j, \\
	\widetilde{\Gamma}_j &= \{ P \in \mathcal{K}_M \cap (B \times \mathbb{R}^m) : \alpha_j \leq \mathcal{N}_M(P) < \alpha_{j+1} \} \setminus V_j .
\end{align*}
By Theorem~\ref{theorem2}, $\Gamma_j$ locally $(n-2)$-rectifiable with locally finite $\mathcal{H}^{n-2}$-measure.  Since $\widetilde{\Gamma}_j \subseteq \mathcal{K}_M^{(n-3)}$, by Lemma~\ref{stratification lemma}, $\dim_{\mathcal{H}}(\widetilde{\Gamma}_j) \leq n-3$.  By Lemma~\ref{semicont freq lemma1}, $\mathcal{N}_M$ is upper semi-continuous and thus $\Gamma_j$ and $\widetilde{\Gamma}_j$ are both the intersection of open and closed set and hence are both locally compact.  Of course, $\op{sing} M \cap (B \times \mathbb{R}^m) = \bigcup_{j=1}^k (\Gamma_j \cup \widetilde{\Gamma}_j)$. 
\end{proof}

\begin{proof}[Proof of Theorem~D of the Introduction]
 Let $P \in \mathcal{B}$ and let $L_{P}$ be a hyperplane such that one tangent cone to $V$ at $P$ is $2|L_{P}|$.  By~\cite{MW}, there exists $\delta > 0$ and $\mu 
 \in (0, 1)$ such that ${(V - P)} \res {\bf C}_{\delta}(0, L_{P})$ (where ${\bf C}_{\delta}(0, L_{P}) = \{X + Y \, : \, X \in L_{P}, \;\; |X| < \delta, \;\; Y \in L_{P}^{\perp}
 \}$)  is the varifold associated with the normal graph of a $C^{1,\mu}$ two-valued function $\widetilde{u}$ over the ball 
 ${\bf B}_{\delta}(0, L_{P}) = \{ X \in L_{P} \, : \, |X| < \delta\}$, with $\|\widetilde{u}\|_{C^{1, \mu}({\bf B}_{\delta}(0, L_{P}))}$ bounded by a fixed constant depending only on $n$. Hence, since $P \not\in \Omega$, 
 $(\op{spt}\|V\| - P) \cap  {\bf C}_{\delta}(0, L_{P})$ is the normal graph of a $C^{1,\mu}$ two-valued function $\widetilde{u}_P$ over ${\bf B}_{\delta}(0, L_{P}).$ In particular, $2|L_{P}|$ is the unique tangent cone to $V$ at $P.$  By~\cite[Theorems 7.1 and 7.4]{SimWic16}, $\widetilde{u}_P$ is a $C^{1,1/2}$ two-valued function.  By Theorem~A of the present paper, $V$ has a unique tangent function at $P$ and $\widetilde{u}_P$ satisfies an asymptotic expansion as in Theorem~A.  By Theorem~C of the present paper, for each closed ball $B \subseteq B^{n+1}_{\delta}(P)$, $\mathcal{B} \cap B$ is either empty or has positive $\mathcal{H}^{n-2}$-measure and is a finite union of pairwise disjoint, locally compact, locally $(n-2)$-rectifiable sets.  If instead we have a closed ball $B \subseteq B^{n+1}_1(0) \setminus (\{ P : \Theta^n(\|V\|, P) \geq 3 \} \cup \mathcal{S})$, then noting that $\mathcal{B} \cap B$ is compact, by a straightforward covering argument we again have that $\mathcal{B} \cap B$ is either empty or has positive $\mathcal{H}^{n-2}$-measure and is a finite union of pairwise disjoint, locally compact, locally $(n-2)$-rectifiable sets. The rest of the conclusions follow from \cite{MW}.
\end{proof}

\printbibliography

@article{M,
	author = {Minter, Paul},
	journal = {arXiv preprint arXiv:2108.02614},
	title = {The structure of stable codimension one integral varifolds near classical cones of density 5/2},
	year = {2021}}

@article{MW,
	author = {Minter, Paul and Wickramasekera, Neshan},
	journal = {arXiv Preprint},
	title = {A structure theory for stable codimension 1 integral varifolds with applications to area minimising hypersurfaces mod p},
	year = {2021}}

@book{Almgren,
	author = {Almgren, Jr., Frederick J.},
	publisher = {World Scientific Publishing Co. Inc.},
	title = {Almgren's big regularity paper: Q-valued functions minimizing Dirichlet's integral and the regularity of area-minimizing rectifiable currents up to codimension 2},
	volume = {1},
	year = {2000}}

@article{AKS,
	author = {Aronszajn, Nachman and Krzywicki, Andrzej and Szarski, Jacek},
	journal = {Arkiv f{\"o}r Matematik},
	number = {5},
	pages = {417--453},
	publisher = {Springer},
	title = {A unique continuation theorem for exterior differential forms on Riemannian manifolds},
	volume = {4},
	year = {1962}}

@article{Ambrosio,
	author = {Ambrosio, Luigi},
	journal = {Annali della Scuola Normale Superiore di Pisa-Classe di Scienze},
	number = {3},
	pages = {439--478},
	title = {Metric space valued functions of bounded variation},
	volume = {17},
	year = {1990}}

@article{Chang,
	author = {Chang, Sheldon Xu-Dong},
	journal = {Journal of the American Mathematical Society},
	number = {4},
	pages = {699--778},
	title = {Two-dimensional area minimizing integral currents are classical minimal surfaces},
	volume = {1},
	year = {1988}}

@book{DeLSpa11,
	author = {De Lellis, Camillo and Spadaro, Emanuele},
	publisher = {American Mathematical Society},
	title = {$Q$-valued functions revisited},
	volume = {991},
	year = {2011}}

@article{DeLSpa-I,
	author = {De Lellis, Camillo and Spadaro, Emanuele},
	journal = {Geometric and Functional Analysis},
	number = {6},
	pages = {1831--1884},
	publisher = {Springer},
	title = {Regularity of area minimizing currents I: gradient $L^p$ estimates},
	volume = {24},
	year = {2014}}

@article{DeLSpa-II,
	author = {De Lellis, Camillo and Spadaro, Emanuele},
	journal = {Annals of Mathematics},
	number = {2},
	pages = {499--575},
	publisher = {JSTOR},
	title = {Regularity of area minimizing currents II: center manifold},
	volume = {183},
	year = {2016}}

@article{DeLSpa-III,
	author = {De Lellis, Camillo and Spadaro, Emanuele},
	journal = {Annals of Mathematics},
	number = {2},
	pages = {577--617},
	publisher = {JSTOR},
	title = {Regularity of area minimizing currents III: blow-up},
	volume = {183},
	year = {2016}}

@article{DSS-I,
	author = {De Lellis, Camillo and Spadaro, Emanuele and Spolaor, Luca},
	journal = {Transactions of the American Mathematical Society},
	number = {3},
	pages = {1783--1801},
	title = {Regularity theory for 2-dimensional almost minimal currents I: Lipschitz approximation},
	volume = {370},
	year = {2018}}

@article{DSS-II,
	author = {De Lellis, Camillo and Spadaro, Emanuele and Spolaor, Luca},
	journal = {Annals of PDE},
	number = {2},
	pages = {1--85},
	publisher = {Springer},
	title = {Regularity theory for 2-dimensional almost minimal currents II: branched center manifold},
	volume = {3},
	year = {2017}}

@article{DSS-III,
	author = {De Lellis, Camillo and Spadaro, Emanuele and Spolaor, Luca},
	journal = {Journal of Differential Geometry},
	number = {1},
	pages = {125--185},
	publisher = {Lehigh University},
	title = {Regularity theory for $2 $-dimensional almost minimal currents III: Blowup},
	volume = {116},
	year = {2020}}

@article{GL87,
	author = {Garofalo, Nicola and Lin, Fang-Hua},
	journal = {Communications on Pure and Applied Mathematics},
	number = {3},
	pages = {347--366},
	publisher = {Wiley Online Library},
	title = {Unique continuation for elliptic operators: a geometric-variational approach},
	volume = {40},
	year = {1987}}

@article{Krum16,
	author = {Krummel, Brian},
	journal = {Communications in Analysis and Geometry},
	number = {4},
	pages = {877--935},
	publisher = {International Press of Boston},
	title = {Existence and regularity of multivalued solutions to elliptic equations and systems},
	volume = {27},
	year = {2019}}

@article{KrumWic1,
	author = {Krummel, Brian and Wickramasekera, Neshan},
	journal = {arXiv preprint arXiv:1311.0923},
	title = {Fine properties of branch point singularities: two-valued harmonic functions},
	year = {2013}}

@article{KrumWic2,
	author = {Krummel, Brian and Wickramasekera, Neshan},
	journal = {arXiv preprint arXiv:1711.06222},
	title = {Fine properties of branch point singularities: Dirichlet energy minimizing multi-valued functions},
	year = {2017}}

@article{NV,
	author = {Naber, Aaron and Valtorta, Daniele},
	journal = {Journal of the European Mathematical Society},
	number = {10},
	pages = {3305--3382},
	title = {The singular structure and regularity of stationary varifolds},
	volume = {22},
	year = {2020}}

@article{Resh1,
	author = {Reshetnyak, Yu G.},
	journal = {Siberian Mathematical Journal},
	number = {3},
	pages = {567--583},
	publisher = {Springer},
	title = {Sobolev-type classes of functions with values in a metric space},
	volume = {38},
	year = {1997}}

@article{Resh2,
	author = {Reshetnyak, Yu G.},
	journal = {Siberian Mathematical Journal},
	number = {4},
	pages = {709--721},
	publisher = {Springer},
	title = {Sobolev-type classes of functions with values in a metric space. II},
	volume = {45},
	year = {2004}}

@article{Ros10,
	author = {Rosales, Leobardo},
	journal = {Calculus of Variations and Partial Differential Equations},
	number = {1},
	pages = {59--84},
	publisher = {Springer},
	title = {The geometric structure of solutions to the two-valued minimal surface equation},
	volume = {39},
	year = {2010}}

@article{Sim93,
	author = {Simon, Leon},
	journal = {Journal of Differential Geometry},
	number = {3},
	pages = {585--652},
	publisher = {Lehigh University},
	title = {Cylindrical tangent cones and the singular set of minimal submanifolds},
	volume = {38},
	year = {1993}}

@article{Sim93b,
	author = {Simon, Leon},
	journal = {Surveys in Differential Geometry},
	number = {1},
	pages = {246--305},
	publisher = {International Press of Boston},
	title = {Rectifiability of the singular sets of multiplicity 1 minimal surfaces and energy minimizing maps},
	volume = {2},
	year = {1993}}

@article{SimWic07,
	author = {Simon, Leon and Wickramasekera, Neshan},
	journal = {Journal of Differential Geometry},
	number = {1},
	pages = {143--173},
	publisher = {Lehigh University},
	title = {Stable branched minimal immersions with prescribed boundary},
	volume = {75},
	year = {2007}}

@article{SimWic16,
	author = {Simon, Leon and Wickramasekera, Neshan},
	journal = {Communications on Pure and Applied Mathematics},
	number = {7},
	pages = {1213--1258},
	publisher = {Wiley Online Library},
	title = {A frequency function and singular set bounds for minimal immersions},
	volume = {69},
	year = {2016}}

@article{Wic08,
	author = {Wickramasekera, Neshan},
	journal = {Journal of Differential Geometry},
	number = {1},
	pages = {79--173},
	publisher = {Lehigh University},
	title = {A regularity and compactness theory for immersed stable minimal hypersurfaces of multiplicity at most 2},
	volume = {80},
	year = {2008}}

@article{Wic,
	author = {Wickramasekera, Neshan},
	title = {Unpublished work},
	year = {2021}}

\bigskip
\hskip-.2in\vbox{\hsize3in\obeylines\parskip -1pt 
  \small 
Brian Krummel
School of Mathematics \& Statistics 
University of Melbourne
Parkville,VIC  3010, Australia
\vspace{4pt}
{\tt brian.krummel@unimelb.edu.au}} 
\vbox{\hsize3in
\obeylines 
\parskip-1pt 
\small 
Neshan Wickramasekera
DPMMS 
University of Cambridge 
Cambridge CB3 0WB, United Kingdom
\vspace{4pt}
{\tt N.Wickramasekera@dpmms.cam.ac.uk}
}

\end{document}